\documentclass[11pt]{article}
\usepackage[top=1in,bottom=1.25in,left=1.25in,right=1.25in]{geometry}

\usepackage{graphicx}
\usepackage{ragged2e}
\usepackage{tabto}
\usepackage{amssymb}
\usepackage{float}
\usepackage{amsmath}
\usepackage{amsthm}
\usepackage{epsfig}
\usepackage{subfigure}
\usepackage{amscd}
\usepackage{color}
\usepackage{url}
\usepackage{hyperref}

\usepackage{cleveref}
\usepackage{afterpage}
\usepackage{fancybox}
\usepackage{ifthen}
\usepackage{multicol}
\usepackage[font=footnotesize]{caption}
\usepackage{bigints}

\usepackage{thmtools,thm-restate}

\usepackage{hyperref}

\usepackage[sf,bf,medium]{titlesec}

\usepackage{dsfont}

\floatstyle{plain}
\restylefloat{figure}

\newcommand\PI{\operatorname{PI}}

\newcommand\loc{\operatorname{loc}}
\newcommand\Dom{\operatorname{Dom}}

\newcommand\dist{\operatorname{dist}}

\newcommand\Span{\operatorname{span}}
\newcommand\tot{\operatorname{tot}}
\newcommand\In{\operatorname{in}}

\newcommand\out{\operatorname{out}}

\newcommand\cX{\mathcal X}
\newcommand\cY{\mathcal Y}

\newcommand\cQ{\mathcal Q}
\newcommand\cK{\mathcal K}

\newcommand\cB{\mathcal B}

\newcommand\cP{\mathcal P}

\newcommand\bEta{\boldsymbol \eta}
\newcommand\tbEta{\widetilde{\boldsymbol \eta}}

\newcommand\bj{\boldsymbol j}

\newcommand\tpsi{\widetilde \psi}

\newcommand\tD{\widetilde D}

\newcommand\tL{\widetilde L}

\newcommand\cW{\mathcal W}

\newcommand\bxi{\boldsymbol \xi}

\newcommand{\sdrx}{\rho(x)}

\newcommand\Ker{\operatorname{ker}}

\newcommand\thop{\theta_0^{\bot}}

\newcommand\cA{\mathcal{A}}
\newcommand\tcA{\widetilde{\mathcal{A}}}
\newcommand\cC{\mathcal{C}}

\newcommand\cS{\mathcal{S}}
\newcommand\lL{\lambda_L}

\newcommand\cD{\mathcal{D}}

\newcommand\cZ{\mathcal Z}

\renewcommand\Im{\operatorname{Im}}

\newcommand\bbC{\mathbb C}

\newcommand\bbN{\mathbb N}

\newcommand\bbR{\mathbb R}
\newcommand\bbZ{\mathbb Z}

\newcommand\pa{\partial}

\newcommand\restrictedto{\upharpoonright}

\newcommand\subsubset{\subset\!\subset}

\newcommand\Id{\operatorname{Id}}

\DeclareMathOperator{\dR}{dR}

\newtheorem{theorem}{Theorem}
\newtheorem{proposition}{Proposition}
\newtheorem{corollary}{Corollary}
\newtheorem{lemma}{Lemma}

\theoremstyle{definition}

\theoremstyle{remark}
\newtheorem{remark}{Remark}

\numberwithin{equation}{section}

\crefname{equation}{}{}

\begin{document}


\begin{titlepage}

  \raggedleft
  {\texttt{arXiv Preprint\\
    \today}}
  
  \hrulefill

  \vspace{2\baselineskip}

  \raggedright
  {\LARGE \sffamily\bfseries Type-I Superconductors in the Limit as the London Penetration Depth Goes to 0}
  
\bigskip
  {\bf\large This paper is dedicated to the memory of J.J.~Kohn, who loved the Hodge decomposition.} 
  
 \vspace{\baselineskip}
 
 \vspace{\baselineskip}
 
  \normalsize Charles L. Epstein \\
    \small \emph{Flatiron Institute, Center for Computational Mathematics \\
    New York, NY 10010}
 
  \texttt{cepstein@flatironinstitute.org}
  
 \vspace{\baselineskip}
    \normalsize Manas Rachh \\
    \small \emph{Flatiron Institute, Center for Computational Mathematics \\
    New York, NY 10010}
    
    \texttt{mrachh@flatironinstitute.org}
  
 \vspace{\baselineskip}
    \normalsize Yuguan Wang \\
    \small \emph{University of Chicago, Department of Statistics \\
    Chicago, IL 55455}
    
    \texttt{yuguanw@uchicago.edu}

    \normalsize

 \vspace{2\baselineskip}
\end{titlepage}

\begin{abstract}
\justifying
\noindent

This paper provides an explicit formula for the approximate solution of the
static London equations. These equations describe the currents and magnetic
fields in a Type-I superconductor.  We represent the magnetic field as a 2-form
and the current as a 1-form, and assume that the superconducting material is
contained in a bounded, connected set, $\Omega,$ with smooth boundary. The
“London penetration depth” gives an estimate for the thickness of the layer near
$\pa\Omega$ where the current is largely carried.  In an earlier
paper,~\cite{EpRa1}, we introduced a system of Fredholm integral equations of
second kind, on $\pa\Omega,$ for solving the physically relevant scattering
problems in this context. In real Type-I superconductors the penetration depth
is very small, typically about $100$nm, which often renders the integral
equation approach computationally intractable. In this paper we provide an
explicit formula for approximate solutions, with essentially optimal error
estimates, as the penetration depth tends to zero. Our work makes extensive use
of the Hodge decomposition of differential forms on manifolds with boundary, and thus
evokes Kohn’s work on the tangential Cauchy-Riemann equations.
\\
    \vspace*{1ex}
    
\end{abstract}

\tableofcontents


\section{Introduction}
Type-I superconducting materials are a subclass of materials that are
superconducting at sufficiently low temperatures, as evidenced by zero
resistance to the conduction of electrical currents and the Meissner effect, see~\cite{tinkham2004introduction}. In a
type-I superconductor the transition to a normal, ohmic conductor is an abrupt
first order phase transition that occurs when the applied magnetic field
strength, current or temperature exceed certain critical thresholds. As such,
this phase is usually described by a phase diagram with axes for temperature,
applied field strength and current density. See Figure~\ref{phase_diag}. In
these materials Ohm's law relating the current to the electric field,
$\bxi=\sigma\bj,$ is replaced with the London equations. In a static type-I
superconductor the electric field is zero, and the London equations relate the
current, $\bj,$ to the magnetic field, $\bEta,$
\begin{equation}
  d^*\bEta=\bj,\quad d\bj=\frac{-1}{\lambda_L^2}\bEta.
\end{equation}
Throughout this paper, electric fields and currents are represented by 1-forms, and magnetic field by 2-forms.

\begin{figure}[!htbp]
\begin{center}
\includegraphics[width=0.7\linewidth]{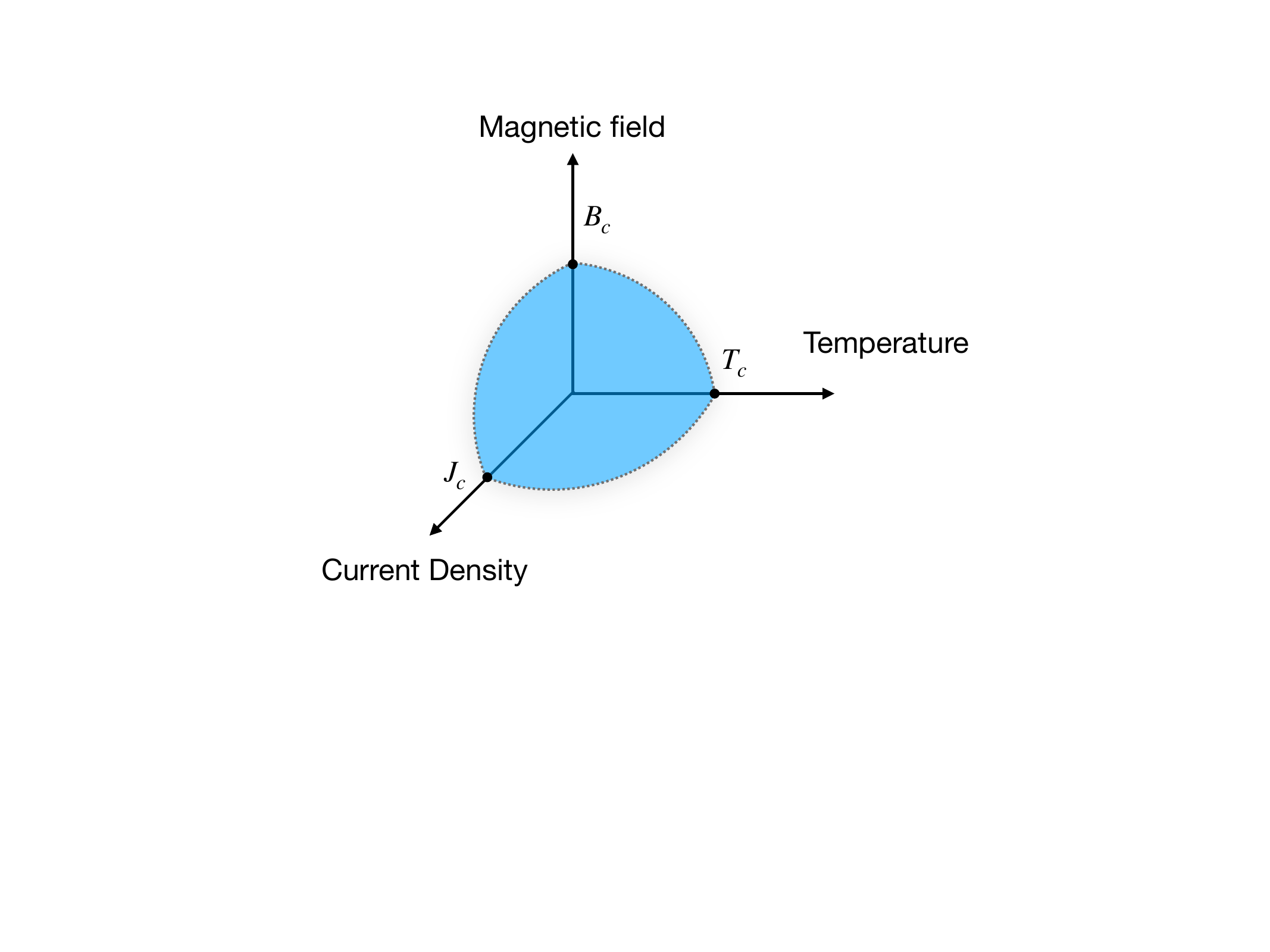}
\caption{Phase diagram for a typical type-I superconductor; the material is superconducting in the interior of the shaded region. Figure from~\cite{Ainslie}.}
\label{phase_diag}
\end{center}
\end{figure}

Type-I superconductors have a characteristic length, called the London
penetration depth ($\lambda_L$), which is a measure of the thickness of the
boundary layer that carries most of the current in the superconducting
state. For materials used in engineering applications, this length is very
small, typically ranging from 10 to 200nm. The material itself is often many
multiples of this length. The large separation of scales, between the dimensions
of the superconducting material and the London penetration depth, presents a
major challenge for the efficient and accurate numerical computation of the
fields and currents generated by superconductors. In~\cite{EpRa1} we present a
method for solving these problems, using an invertible system of Fredholm
equations of second kind. While this works, in principle, for any $\lambda_L>0,$
the method becomes impractical to implement when $\lambda_L$ is very small, compared to the size of the superconductor.

In this work, we analyze two aspects of the theory of type-I superconductors,
which are essential to the efficiency of their numerical simulation in
physically relevant regimes.  We derive the limiting differential equations
governing the magnetic fields and currents in superconductors as $\lambda_{L}
\to 0^{+}$ (which is a singular limit). In this limit, we show that governing
equations reduce to a well-posed version of the magneto-statics problem in the
exterior of the superconducting material, the currents inside the
superconducting material converge to a current sheet on the boundary of the
superconducting material, and finally the magnetic field inside the
superconducting material converges to zero.  Using the solution to the limiting
differential equations, we construct an explicit approximate solution for
the $\lambda_L>0$
London equations, which is accurate, in the $L^2$-topology,
to $O(\lambda_{L}^{1-\varepsilon}),$ for any  $\varepsilon > 0$.

Let $\Omega\subset\bbR^3$ denote a bounded open set with a smooth boundary,
which contains the superconducting material. We first analyze the $\lambda_L\to
0^+$ limit assuming that $\pa\Omega$ connected, and then extend our results to a
more general class of domains, where the boundary $\pa \Omega$ contains 2 
connected components. This is important in applications as the current in a
superconductor is mostly carried in a neighborhood of $\pa \Omega$ with
thickness comparable to $\lambda_{L}$. The domains we consider include thin
shells of superconducting material, similar to those typically used in
practice. Note, however, that our analysis of the behavior of the solution as
$\lambda_L\to 0^+$ requires that the distance between components of $\pa\Omega$
is a fixed positive number.
 
 The rest of the paper is organized as follows.  In~\Cref{sec:math-prelim}, we
 set the notation, provide some mathematical preliminaries, and review
 elementary properties of layer potentials involved in analyzing the limiting
 behavior of the London equations.  In~\Cref{sec:london}, we review the
 scattering problem for the London equations.  In~\Cref{sec:connected}, we
 derive a limiting form for the London equations when the boundary of the
 superconducting region has one connected component, and in~\Cref{sec3} we
 derive the analogous results for the case of thin shells where the boundary has
 two connected components. In the process, we also provide a different existence
 proof for the solution of the London equations in these settings.
 In~\Cref{sec:numerics}, we estimate the rate of convergence to the limiting
 solution of our approximate solutions to the London equations, and demonstrate
 the efficacy of using the limiting equations for solving the problem via a
 battery of numerical examples.  Finally, we provide some concluding thoughts,
 describe plans for future research, and outline some open questions
 in~\Cref{sec:conclusions}.

\begin{remark}
 This article makes extensive usage of the Hodge decomposition on manifolds with
 boundary. In the context of several complex variables, Joe Kohn made very
 important contribution to Hodge theory, starting with his very early
 papers,~\cite{Ko1,Ko2,KoRo}. Of particular import to the first author (CLE) was
 Kohn's paper on the closed range property for $\overline{\pa}_b$ on an
 embeddable CR-manifold,~\cite{Ko3}, where extensive usage is made of Hodge
 theory.
 \end{remark}

\section{Mathematical preliminaries}
\label{sec:math-prelim}

Suppose that $\Omega$ is a bounded region, with a smooth connected  boundary $\pa \Omega$.
Let $n(x)$ denote the outward pointing normal at $x \in \pa \Omega$.  Let $H(x)$
denote the mean curvature for $x \in \pa \Omega$ relative to this choice of
normal vector field.  Let $\pi(x)$ denote the nearest point map given by
\begin{equation}
   \pi(x)=\arg\min_{p\in\pa\Omega}\dist(x,p) \, .
 \end{equation}
 Since the surface $\pa\Omega$ is smooth and embedded, there is a
 neighborhood $U\supset\pa\Omega$ in which the nearest point map
 is single valued and smooth.  The signed distance function $\sdrx$ given by
 \begin{equation}
 \label{eq:sign-dist}
   \sdrx=\begin{cases} &-\dist(x,\pi(x))\text{ for }x\in\Omega,\\
   &\phantom{-}\dist(x,\pi(x))\text{ for }x\in\Omega^c,
   \end{cases}
 \end{equation}
 is also smooth in $U.$  
For $p\in\pa\Omega,$ the set $\pi^{-1}(p)\cap U$ is
 a straight line meeting $\pa\Omega$ orthogonally at $p.$ 
Let
 $N(x)$ denote the outward pointing, unit vector tangent to the line
 $\pi^{-1}(\pi(x)),$ for $x \in U$. It follows that $n(x) = N(x)$ for $x \in \pa \Omega$.

As noted, throughout this paper we use the description of electromagnetic fields
in terms of differential forms: electric currents and fields are 1-forms, and
magnetic field are 2-forms.  The analysis of this problem uses many standard
facts from Hodge theory on manifolds with boundary, which can be found
in~\cite{Taylor2}[Chapter 5]. The notation $\star$ refers to the Hodge-star
operator defined on $\bbR^3$ by the standard volume form, $dV,$ and $\star_2$
refers to the Hodge star-operator on an oriented surface with outer unit normal
$n;$ the induced orientation is defined by $dS_{\pa \Omega} = i_{n}dV.$ We also
use $\pa_{n}$ to denote the action, along $\pa\Omega,$ of the unit vector field
$N,$
   \begin{equation}
     \pa_n u= N u\restrictedto_{\pa\Omega}.
   \end{equation}
   Away from $\pa\Omega,$ the action  of $N$ is sometimes denoted by $\pa_{r}.$

We make extensive usage of the $L^2$-Sobolev spaces $H^s(X)$, and $H^s_{0}(X)$
for $s\in\bbR ,$ with $X$ being typically being $\Omega$ or $\pa \Omega$,
see~\cite{Taylor1}[Chapter 4]. These are, in fact, usually Sobolev spaces of
sections of vector bundles, e.g. $\Lambda^k(X)\to X,$ but to keep the notation
simple we leave this out of our notation.  For any two Banach spaces $\cX,\cY$,
and a bounded linear operator $Q:\cX\to\cY,$ let
 \begin{equation}
   \|Q\|_{\cX\to \cY}=\sup_{x\neq 0}\frac{\|Qx\|_{\cY}}{\|x\|_{\cX}} \, ,
 \end{equation}
denote the norm. Here, and in the sequel, $\|\cdot\|_{\cX}$ denotes the norm on
$\cX,$ etc.

We use the notation $\langle u,v\rangle_X$ for the $L^2$-inner product of
functions defined on $X.$ We also use this notation for its natural extension as
the Hermitian duality pairing $H^s(X)\times [H^{s}]'(X)\to\bbC.$ If $M$ is a
compact manifold \emph{without} boundary, then there are natural pairings $H^s(M)\times
H^{-s}(M)\to\bbC.$  For $X$ a manifold with boundary  we recall
that, for $-\frac 12 <s<\frac 12,$ we have the natural isomorphism
$H^s_0(X)\simeq H^s(X).$ We use $\| \cdot \|$ to either denote operator norms or
norms of sections over surfaces or regions, and $| \cdot |$ to denote the
point-wise norm of forms, or points in Euclidean space.

We use the following trace theorem, estimating the norm of the restriction of a
closed 2-form to the boundary in terms of its norm in the volume in appropriate
$H^{s}$ spaces.
\begin{restatable}{lemma}{tracelemma}
\label{tr_lem} Let $\Omega\subset\bbR^d$ be an open set with $\pa\Omega$ a smooth submanifold.
Let $s\geq 0,$ if $\theta\in H^s(\Omega;\Lambda^2)$ is closed, that is
$d\theta=0,$ in the sense of distributions, then
$\theta\restrictedto_{\pa\Omega}$ is well defined as an element of $H^{s-\frac
  12}(\pa\Omega;\Lambda^2),$ and there is a constant $C_s$ so that
   \begin{equation}
     \|\theta\restrictedto_{\pa\Omega}\|_{H^{s-\frac 12}(\pa\Omega)}\leq
     C_s \|\theta\|_{H^s(\Omega)}.
   \end{equation}
 \end{restatable}
The result is well known. For $s>\frac 12$ it is a consequence of the standard
trace theorem.  Since it is so central to our analysis we give a complete proof
in~\Cref{app_tr_lem}

\subsection{Layer potentials}\label{subsec:lp}
In this section we review some elementary properties of layer potentials for the
Laplace equation. These results are classical and can be found
in~\cite{ColtonKress} for example.  Let $k(x,y)$ be a kernel function defined in
$\bbR^3\times \bbR^3,$ which is typically singular on the diagonal and smooth
otherwise. This kernel defines layer potential over the boundary $\pa \Omega$,
by setting
\begin{equation}
  \cK f(x)=\int_{\pa\Omega}k(x,y)f(y)dS_y\text{ for }f\in \cC^{\infty}(\pa\Omega),
\end{equation}
with $x\notin\pa\Omega.$ 
We then let $Kf(x)$ denote the appropriate limit of $\cK f(x)$ as the point of
evaluation approaches $\pa\Omega.$  Depending on the nature of the singularity in the kernel of $\cK,$ the operator
$f\mapsto Kf$ may denote the restriction $\cK f,$ its principal value, or its
finite part. Typically, this limit extends to define a bounded map
$K: H^s(\pa \Omega) \to H^t(\pa \Omega),$ for appropriate $s,t.$

Recall, that the Green's function for the Laplace equation is given by
\begin{equation}
g(x,y) = \frac{1}{4\pi |x-y|} \, .
\end{equation}
For $x \in \bbR^{3} \setminus \pa \Omega$, let 
$\cS [\sigma]$, $\cD[\sigma]$ denote the Laplace single and double layer
potentials given by integrating over $\pa\Omega:$
\begin{equation}
\cS[\sigma](x) = \int_{\pa \Omega} g(x,y) \sigma(y) \, dS  \, , \quad  
\cD[\sigma](x) = \int_{\pa \Omega} \left( n(y) \cdot \nabla_{y} g(x,y)  \right)\sigma(y) \, dS \, \, .
\end{equation}
When necessary to denote which boundary, $W,$ the layer potential density is
supported on we use a subscript $\cS_{W}[\sigma],$ $\cD_{W}[\sigma],$ etc.  Let
$S[\sigma]$ and $D[\sigma]$ denote the restrictions of $\cS[\sigma],
\cD[\sigma],$ to $x \in \pa \Omega,$ which are given by
\begin{equation}
S[\sigma](x) = \int_{\pa \Omega} g(x,y) \sigma(y) \, dS  \, , \quad  
D[\sigma](x) = {\rm p.v.}\int_{\pa \Omega} \left( n(y) \cdot \nabla_{y} g(x,y)  \right)\sigma(y) \, dS \, \, .
\end{equation}
These are pseudodifferential operators of order $-1.$ The single layer operator
$S:L^2(\pa\Omega)\to L^2(\pa\Omega)$ is a positive, self adjoint operator, that
is $\langle S\varphi,\varphi\rangle_{\pa\Omega}>0$ for any $\varphi\neq 0$ in
$L^2(\pa\Omega).$

Let $S'[\sigma]$ and $D'[\sigma]$ denote the normal
derivatives of $\cS[\sigma]$ and $\cD[\sigma]$ restricted to $x \in \pa \Omega,$
they are given by
\begin{equation}
\begin{aligned}
S'[\sigma](x) &={\rm p.v.} \left(n(x) \cdot \nabla_{x} \left( \int_{\Gamma} g(x,y) \sigma(y) \, dS  \right)  \right)  \, , \\
D'[\sigma](x) &= {\rm f.p.}\left( n(x) \cdot \nabla_{x} \left( \int_{\Gamma} \left( n(y) \cdot \nabla_{y} g(x,y)  \right)\sigma(y) \, dS  \right) \right) \, .
\end{aligned}
\end{equation}
The abbreviation ``p.v.'' denotes a principal value integral and the
abbreviation ``f.p.'' denotes the finite part of the integral. $S'$ is an
operator of order $-1,$ and $D'$ is an operator of order $1.$
The layer potentials satisfy the following jump relations and Calderon identities, see~\cite{ColtonKress,Nedelec}, for example.
Suppose that $x_{0} \in \pa \Omega$, then 
\begin{equation}
\label{eq:identites}
\begin{aligned}
\lim_{\substack{x\to x_{0} \\ x \in \Omega^c}} \cS[\sigma] &= S[\sigma](x_{0})\,, \\
\lim_{\substack{x\to x_{0} \\ x \in \Omega}} \cS[\sigma] &= S[\sigma](x_{0})\,, \\
\lim_{\substack{x \to  x_{0} \\ x \in \Omega^c}} \nabla_{x} \cS [\sigma] &= \frac{\sigma(x_{0}) n(x_{0})}{2} + {\rm p.v.} \nabla_{x} S[\sigma](x_{0}) \, ,\\
\lim_{\substack{x \to  x_{0} \\ x \in \Omega}} \nabla_{x} \cS [\sigma] &= -\frac{\sigma(x_{0}) n(x_{0})}{2} + {\rm p.v.} \nabla_{x} S[\sigma](x_{0}) \, ,\\
DS[\sigma](x_{0}) &= SS'[\sigma](x_{0}) \, .
\end{aligned}
\end{equation}

\section{The London equations} \label{sec:london}
In this section, we review the standard scattering problem for the London equations describing the fields inside the superconducting material
and the magnetic field in the exterior of the superconductor. This problem is discussed in detail in~\cite{EpRa1}, and we reproduce here a brief
summary for completeness.
Suppose that the superconducting material is
homogeneous and isotropic, occupying a bounded region $\Omega.$ We assume that $\pa\Omega$ is smooth surface embedded in $\bbR^3$.

Within $\Omega$ there is a magnetic
field, which we describe as a 2-form $\bEta^0,$
\begin{equation}\label{eqn1.001}
  \bEta^0=\eta^0_1dx_2\wedge dx_3+\eta^0_2dx_3\wedge dx_1+\eta^0_3dx_1\wedge dx_2,
\end{equation}
and a current, which we describe
as a 1-form $\bj^0,$
\begin{equation}
  \bj^0=j^0_1dx_1+j^0_2dx_2+j^0_3dx_3.
\end{equation}
According to the London equations these satisfy the first
order system:
\begin{equation}\label{eqn1}
  d^*\bEta^0=\bj^0,\text{ and } d\bj^0=-\frac{1}{\lambda_L^2}\bEta^0\text{ within }\Omega.
\end{equation}
Recall that,  on $p$-forms in $\bbR^3,$ the formal adjoint of $d$ is
$d^*=(-1)^p\star d\star,$ where $\star$ is the Hodge-star operator defined by
the Euclidean metric. These equations imply that
\begin{equation}\label{eqn2}
  d\bEta^0=0\text{ and }d^*\bj^0=0.
\end{equation}
The constant $\lambda_L$ is called the London penetration depth.  It measures
the thickness of the current carrying portion of a superconducting material. 
In the complementary region there is a static magnetic field $\bEta^{\tot}$ (also described as a 2-form),  which
satisfies the usual  equations of magneto-statics
\begin{equation}\label{eqn3}
  d\bEta^{\tot}=0,\text{ and } d^*\bEta^{\tot}=0\text{ in }\Omega^c.
\end{equation}

The current $\bj^0(x)$ is supported in $\Omega$ and decays exponentially in $\dist(x,\pa\Omega).$ As
there is no current sheet on the boundary, the physically
reasonable boundary conditions are
\begin{equation}
  \bEta^0\restrictedto_{\pa\Omega}= \bEta^{\tot}\restrictedto_{\pa\Omega}\text{ and }
  \star\bEta^0\restrictedto_{\pa\Omega}= \star\bEta^{\tot}\restrictedto_{\pa\Omega},
\end{equation}
i.e, the normal \emph{and} tangential components of the magnetic field are
continuous across the boundary.
 This boundary condition and $d\star\bEta^{\tot}=0$  imply that
\begin{equation}
  d_{\pa\Omega}[\star\bEta^0\restrictedto_{\pa\Omega}]=0,
\end{equation}
that is $\star\bEta^0\restrictedto_{\pa\Omega}$ is a closed 1-form on
$\pa\Omega.$ The London equation shows that $i_{n}\bj^0=0$ is
equivalent to
$d_{\pa\Omega}[\star\bEta^0\restrictedto_{\pa\Omega}]=0,$ hence the
current is tangent to $\pa\Omega.$  

In the standard ``scattering'' problem for this set-up the magnetic field in
$\Omega^c$ is split into an incoming magneto-static field $\bEta^{\In},$ and an
outgoing scattered field, $\bEta^+.$ The field $\bEta^+$ is defined in all of
$\Omega^c;$ the assumption that $\bEta^+$ is an outgoing field means that
\begin{equation}
  |\bEta^+(x)|=o(|x|^{-1}).
\end{equation}
Suppose that the incoming field is a solution to Maxwell's equations generated by
sources that are a positive distance from $\overline{\Omega}.$ 
The boundary conditions can therefore be written as:
\begin{equation}\label{eq:bc1}
  \bEta^0\restrictedto_{\pa\Omega}-\bEta^+\restrictedto_{\pa\Omega}= \bEta^{\In}\restrictedto_{\pa\Omega}\text{ and }
  \star\bEta^0\restrictedto_{\pa\Omega}-\star\bEta^+\restrictedto_{\pa\Omega}
  = \star\bEta^{\In}\restrictedto_{\pa\Omega}.
\end{equation}

If the domain $\Omega$ is non-contractible, then the continuity of the 
magnetic field across the boundary is not sufficient to imply uniqueness of
solutions to the London equations. There are non-trivial solutions with $\bEta^{\In}=0.$
Suppose $\pa \Omega$ has 1 connected component and genus $g$, 
then there are $g$ A-cycles $\{A_1,\dots,A_g\},$ which bound surfaces
$\{S_{A_1},\dots,S_{A_g}\}$ contained within $\Omega,$ and $g$
B-cycles $\{B_1,\dots,B_g\},$ which bound surfaces
$\{S_{B_1},\dots,S_{B_g}\}$ contained in $\Omega^c,$  see Figure~\ref{fig2.1}. Since
$[\star\bEta^{\tot}]\restrictedto_{\pa\Omega}
=[\star\bEta^0]\restrictedto_{\pa\Omega}$ are closed 1-forms, their
integrals over a cycle in $\pa\Omega$ depend only on the homology
class of the cycle. It follows from the boundary condition, and Stokes
theorem that
 \begin{equation}
   \int\limits_{B_j}\star\bEta^0= \int\limits_{B_j}\star\bEta^{\tot}=
   \int\limits_{S_{B_j}}d\star[\bEta^{\In}+\bEta^+]=0,
 \end{equation}
 provide that $\bEta^{\In}$ is defined in a contractible neighborhood of
 $\Omega.$ If not, then
  \begin{equation}\label{eq:bc-top1}
  b_j= \int\limits_{B_j}\star\bEta^0= \int\limits_{B_j}\star\bEta^{\tot}=
   \int\limits_{B_j}\star\bEta^{\In},
  \end{equation}
  which may not be zero, but is determined by $\bEta^{\In}.$
\begin{figure}[h]
  \centering
  {\includegraphics[width=.45\linewidth]{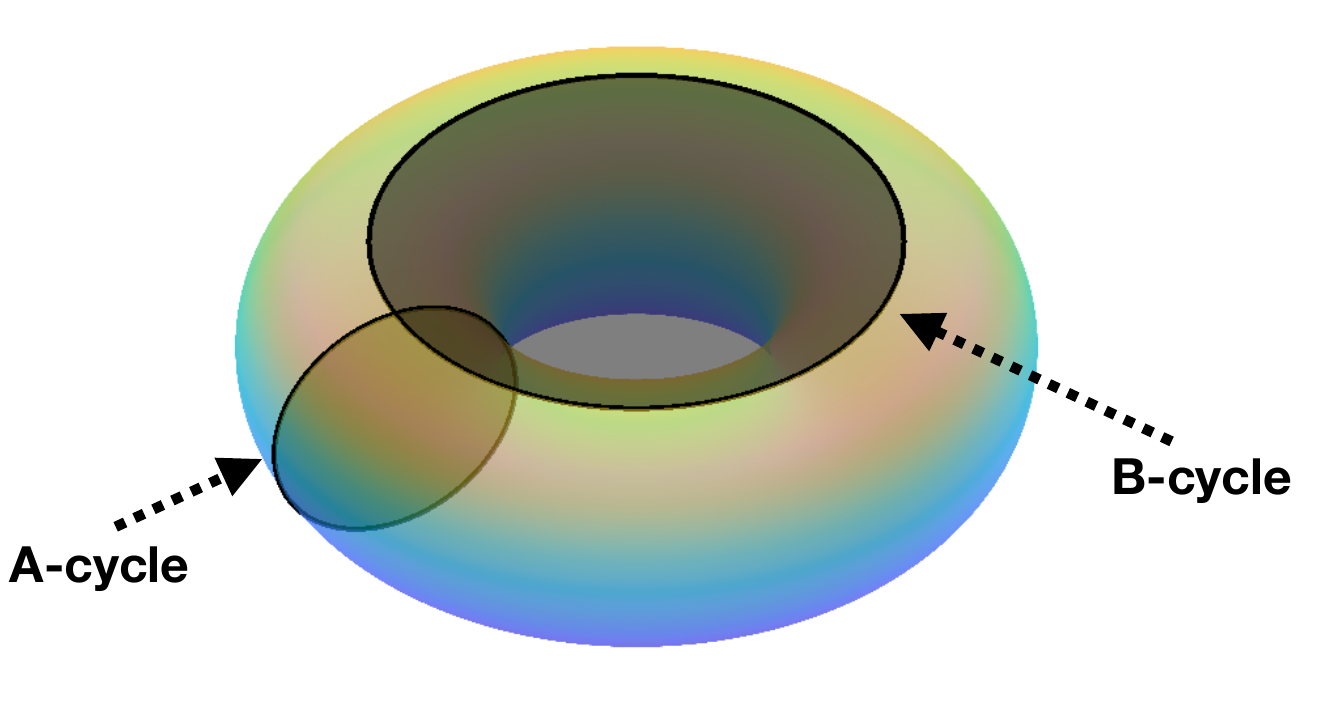}} 
\caption{Torus showing an $A$-cycle and a $B$-cycle, as
  well as the spanning surfaces.}
\label{fig2.1}
\end{figure}

 On the other hand, the periods
 \begin{equation}\label{eq:bc-top2}
   a_j = \int\limits_{A_j}\star\bEta^{\tot}
 =\int\limits_{A_j}\star\bEta^0= \int\limits_{S_{A_j}}d\star \bEta^0=
   \int\limits_{S_{A_j}}\star\bj^0,\text{ for }j=1,\dots, g,
 \end{equation}
 are not determined a priori, and in fact, constitute additional data
 that must be specified to get a unique solution. As $\bEta^{\In}$ is
 defined in a neighborhood of $\overline{\Omega},$ and $\pa
 S_{A_j}=A_j,$ Stokes theorem shows that
 \begin{equation}
   \int\limits_{A_j}\star\bEta^{\tot}=\int\limits_{A_j}\star\bEta^{+}.
 \end{equation}

 In our earlier paper,~\cite{EpRa1}, assuming that the
 $b_j=0,\,j=1,\dots,g,$ we proved, among other things, that this
 problem has a unique solution, $(\bEta^+,\bEta^0,\bj^0),$ for any
 $\lL\in(0,\infty),$ for any incoming field $\bEta^{\In},$ defined in a
 contractible neighborhood of $\Omega,$ and periods $\{a_j\}.$ 

\section{The Limit $\lL\to 0^+$ with $\pa \Omega$ Connected} \label{sec:connected}
The analysis of the $\lL\to 0^+$-limit proceeds by finding a system of equations
for the magnetic fields alone, with no mention of the current in the
superconducting region.  The rest of the section is dedicated to deriving these
results when $\pa \Omega$ is a connected set. This is somewhat simpler to
discuss and sets the stage for the more involved analysis needed when $\pa
\Omega$ has $2$ components.

\begin{remark}
  In this and the following section $C$ refers to a
variety of positive constants that do not depend on either $\lL$ or the data.
\end{remark}

Let $(\bEta^+_{\lL},\bEta^0_{\lL},\bj^{0}_{\lL})$ denote the solution,
 for a given value of $\lL,$ with fixed periods,
 $\{a_1,\dots,a_g\},$ and incoming field, $\bEta^{\In}.$ Combining the equations in $\Omega$, we see
 that $\bEta^0_{\lL}$ satisfies
 \begin{equation}\label{eqn2.1.2}
   \left[dd^*+\frac{1}{\lL^2}\right]\bEta^0_{\lL}=0.
 \end{equation}
 This equation implies that $d\bEta^0_{\lL}=0.$ Furthermore, if $\bEta^0_{\lL}$ solves this
 equation, and we set $\bj^0=d^*\bEta^0_{\lL},$ then the pair
 $(\bEta^0_{\lL},\bj^0_{\lL})$ solves~\eqref{eqn1}. The boundary conditions are as
 before given by~\cref{eq:bc1}, along with the topological constraints~\cref{eq:bc-top1,eq:bc-top2}. 
 
 In the limit, we expect the interior magnetic field to converge to zero, and the current in the superconducting
 region to converge to a current sheet. Thus, in the limit, while the normal components of the field
 remain continuous, the tangential components of the 
 total magnetic field would no longer be continuous. This implies that the limiting exterior field, denoted
 by $\bEta^{\out +}$ should be the outgoing solution to
 \begin{equation}\label{eqn4.2.100}
d \bEta^{\out +} = 0 , \text{ and } d^{*} \bEta^{\out +} = 0 ,
 \end{equation}
 in $\Omega^{c}$, satisfying the boundary conditions
 \begin{equation}
 \bEta^{\out +}\restrictedto_{\pa \Omega} = -\bEta^{\In}\restrictedto_{\pa \Omega} \, ,
 \end{equation}
 and the topological constraints
 \begin{equation}\label{eqn4.4.100}
 \int_{A_{j}} \star\bEta^{\out +} = a_{j} \text{ for }j=1,\dots, g.
 \end{equation}
 If $\bEta^{\In}$ is smooth, then it is well-known that there exists a unique
 outgoing solution to~\eqref{eqn4.2.100}--\eqref{eqn4.4.100}, which is in
 $H^{s}_{\loc}(\Omega^c)$ for any $s \in \bbR$.
 
Our analysis proceeds by observing that $\bEta_{\lL}^{+}$ equals $\bEta^{\out
  +}$ up to a harmonic correction, i.e. $\bEta_{\lL}^{+} = \bEta^{\out +} +
\star d u_{\lL}$ with $u_{\lL}$ being a harmonic function. The interior fields
are split into three parts, the first part denoted by $\beta_{\lL}$ is a closed
2-form which shares the same tangential data as $\bEta^{\out +} + \bEta^{\In}$,
the second part is the harmonic extension of $u_{\lL}$ to $\Omega$ denoted by $v_{\lL};$
the third part is a correction term denoted by $\bEta^{00}_{\lL}$. This term must also
be closed and have $0$ tangential data, since the tangential components of the
total fields are already continuous. The continuity of the normal components of
the magnetic fields relates the jump in the normal derivatives of $u_{\lL}$ and
$v_{\lL}$ to $(\beta_{\lL} + \bEta^{00}_{\lL})\restrictedto_{\pa
  \Omega}$.

To summarize, our ansatz states that
  \begin{equation}\label{top_ansatz}
   \begin{split}
     \bEta^{+}_{\lL}&=\bEta^{\out +}+\star d u_{\lL}\text{ in }\Omega^c\\
     \bEta^{0}_{\lL}&=\bEta^{00}_{\lL}+\beta_{\lL}+\star d v_{\lL}\text{ in }\Omega,
   \end{split}
 \end{equation}
 where $d\beta_{\lL} = 0$ in $\Omega$ and $\star \beta_{\lL}\restrictedto_{\pa
   \Omega} = \star (\bEta^{\out +} + \bEta^{\In})\restrictedto_{\pa \Omega}$,
 $d\bEta^{00}_{\lL} = 0$ in $\Omega$ and
 $\star\bEta^{00}_{\lL}\restrictedto_{\pa \Omega} = 0$, the functions $u_{\lL}$ and
 $v_{\lL}$ are harmonic and satisfy  transmission boundary conditions given below in~\eqref{eqn4.6.100}.

\subsection{Construction of $\beta_{\lL}$}\label{sec:beta}
First, we turn our attention to the construction of $\beta_{\lL}$ which is selected to have the following basic properties
 \begin{equation}
   d\beta_{\lL}=0,\quad \star\beta_{\lL}\restrictedto_{\pa\Omega}= \star(\bEta^{\out+} + \bEta^{\In})\restrictedto_{\pa\Omega}.
 \end{equation}
It is easy to construct many such fields, however we require one for which we
have control on the size of $\|(dd^*+\lL^{-2})\beta_{\lL}\|_{L^2(\Omega)}$ as
$\lL\to 0^+.$ In the following lemma, we provide an explicit construction of
$\beta_{\lL}$ that achieves this goal.  Before stating and proving the result,
recall that $U\subset \bbR^{3}$ is a neighborhood of $\pa \Omega$ where the
nearest point map, $\pi(x),$ and distance function, $\rho(x),$ are smooth.  Let $r_0$ denote a positive number so that
 \begin{equation}
   V_{r_0}=\{x:\:-2r_0\leq \rho(x)\leq 0\}\subset U.
 \end{equation}

 \begin{lemma}\label{lem1}
   Let $\gamma_0$ be a smooth 1-form defined on $\pa\Omega,$ and 
   $\psi\in\cC^{\infty}_c(-2r_0,2r_0),$ with $\psi(r)=1$ for $|r|\leq r_0.$ 
   We let $\gamma=\pi^*(\gamma_0),$ and define the 2-form in $\overline{\Omega}$
   \begin{equation}\label{eqn4.8.101}
     \beta_{\lL}= d (E_{\lL}(\rho(x)) \wedge \gamma) =
     e^{\frac{\rho(x)}{\lambda_L}}\psi(\rho(x))d\rho\wedge\gamma+E_{\lL}(\rho(x))d\gamma,
   \end{equation}
   where
   \begin{equation}
     E_{\lL}(r)=\int^{r}_{-\infty}e^{\frac{s}{\lL}}\psi(s)ds.
   \end{equation}
   This 2-form is compactly supported in the set $V_{r_0}$ and can
   therefore be smoothly extended by 0 to be defined in all of
   $\overline{\Omega}.$ It satisfies the conditions
   \begin{equation}
      \star\beta_{\lL}\restrictedto_{\pa\Omega}=\star_2\gamma_0,\text{
        and }    d\beta_{\lL}=0,
   \end{equation}
   as well as the  estimates
   \begin{equation}\label{eqn2.15.2}
     \begin{split}
     \|\beta_{\lL}\|_{L^2(\Omega)}\leq C\sqrt{\lL},\quad &
     \|[dd^*+\lL^{-2}]\beta_{\lL}\|_{L^2(\Omega)}\leq \frac{C}{\sqrt{\lL}},\\
       \|\beta_{\lL}\restrictedto_{\pa\Omega}&\|_{H^s(\pa\Omega)}\leq
       C_s\lL,\text{ for any }s\in\bbR.
       \end{split}
   \end{equation}
 \end{lemma}
 
 \begin{proof}
   That $\beta_{\lL}$ has support in $V_{r_0},$ is closed and
   can be smoothly extended by 0 is clear from its definition.

   To compute its boundary value we introduce a local oriented,
   orthonormal coframe of the form $\{d\rho,\omega_1,\omega_2\};$ hence
   $dV=\omega_1\wedge \omega_2\wedge d\rho,$ and
   $dS_{\pa\Omega}=i_{N}dV=\omega_1\wedge\omega_2.$ Note that
   $\omega_1(N)=\omega_2(N)=0$ and $d\rho(N)=1,$ implying that
   \begin{equation}
  \star d\rho=\omega_1\wedge \omega_2,\quad   \star\omega_1=\omega_2\wedge d\rho,\quad
  \star\omega_2=-\omega_1\wedge d\rho.
   \end{equation}
   The form $\gamma=\pi^*(\gamma_0),$ and therefore
   $d\gamma=\pi^*(d_{\pa\Omega}\gamma_0).$ Since $\pi_*(N)=0,$ it therefore
   follows that, locally,
   \begin{equation}
     \gamma=a\omega_1+b\omega_2\text{ and }d\gamma=c\omega_1\wedge\omega_2.
   \end{equation}
   Computing in this basis we see that
   \begin{equation}
     \star\beta_{\lL}\restrictedto_{\pa\Omega}=a\omega_2-b\omega_1\restrictedto_{\pa\Omega}=\star_2\gamma_0,
   \end{equation}
   as stated.

   We now turn to the estimates. Integration by parts gives
   \begin{equation}
     E_{\lL}(r)=\lL e^{\frac{r}{\lL}}\psi(r)-\lL\int_{-\infty}^re^{\frac{s}{\lL}}\psi'(s)ds,
   \end{equation}
   which shows that
   \begin{equation}\label{eqn4.16.101}
     |E_{\lL}(r)|\leq\begin{cases} 0&\text{ for }r\leq -2r_0,\\
     C\lL e^{\frac{r}{\lL}}&\text{ for }-2r_0<r\leq 0.
     \end{cases}
   \end{equation}
   Since $\beta_{\lL}\restrictedto_{\pa\Omega}=E_{\lL}(0)d\gamma_0,$ the estimates on
   $\|\beta_{\lL}\restrictedto_{\pa\Omega}\|_{H^s(\pa\Omega)}$ follow from this
   bound and the smoothness of $\gamma_0.$ Using the local coframe, we easily
   see that
   \begin{equation}
   \|\beta_{\lL}\|^2_{L^2(\Omega)}=  \int_{\Omega}\beta_{\lL}\wedge\star \beta_{\lL}\leq
     C\int_{-2r_0}^{0}e^{\frac{2r}{\lL}}dr\leq C\lL.
   \end{equation}

   This leaves only an estimate on $\|[dd^*+\lL^{-2}]\beta_{\lL}\|_{L^2(\Omega)}.$
   Computing in the local coframe we see that
   \begin{equation}
     dd^*\beta_{\lL}=-(e^{\frac{\rho}{\lL}}\psi(\rho))''d\rho\wedge\gamma+\frac{e^{\frac{\rho}{\lL}}}{\lL}\theta_{\lL},
   \end{equation}
   where $\theta_{\lL}$ is a bounded 2-form, as $\lL\to 0^+.$ Arguing as before
   we see that
   \begin{equation}
    [dd^*+\lL^{-2}]\beta_{\lL}=\frac{e^{\frac{\rho}{\lL}}}{\lL}\theta'_{\lL},
   \end{equation}
   where $\theta'_{\lL}$ is a uniformly bounded 2-form supported in $V_{r_0}.$ Hence
   \begin{equation}
      \|[dd^*+\lL^{-2}]\beta_{\lL}\|^2_{L^2(\Omega)}\leq
      C\int_{-2r_0}^0\frac{e^{\frac{2r}{\lL}}}{\lL^2}dr\leq \frac{C}{\lL},
   \end{equation}
   as claimed.
 \end{proof}

 \begin{remark}
   Observe that $|\beta_{\lL}(x)|\leq
   Ce^{\frac{\rho(x)}{\lL}}\chi_{[-2r_0,0]}(\rho(x)),$  which implies that
   \begin{equation}\label{eqn4.21.106}
     \|\beta_{\lL}\|_{L^1(\Omega)}\leq C\lL.
   \end{equation}
 \end{remark}

For  $\beta_{\lL}$ we take  the closed 2-form constructed
using  Lemma~\ref{lem1} with
$$\gamma_0=-\star_2[\star (\bEta^{\out+}
  +\bEta^{\In})\restrictedto_{\pa\Omega}],$$
so that
 \begin{equation}
 \beta_{\lL} = -d\left[E_{\lL}(\rho(x)) \wedge \pi^{*}\left(\star_{2} (\star (\bEta^{\out+} + \bEta^{\In})\restrictedto_{\pa\Omega}) \right)\right] \,,
 \end{equation}
 and $\star \beta_{\lL}\restrictedto_{\pa\Omega} = \star (\bEta^{\out+} + \bEta^{\In})\restrictedto_{\pa\Omega}$.
 
 \subsection{The harmonic functions $u_{\lL}$ and $v_{\lL}$}
 Next, we turn our attention to the harmonic functions $u_{\lL}$ and $v_{\lL}$. Assuming $\bEta^{00}_{\lL}$ were given,   
 $u_{\lL}$, $v_{\lL}$ would satisfy the system of equations
 \begin{equation}\label{eqn4.6.100}
 \begin{aligned}
 \Delta u_{\lL} &= 0, \text{ in } \Omega^{+}, \\
 \Delta v_{\lL} &= 0, \text{ in } \Omega, \\
 u_{\lL} &= v_{\lL}, \text{ on } \pa \Omega, \\
 [\partial_{n} u_{\lL} - \partial_{n} v_{\lL}]_{\pa\Omega}&= \frac{[\bEta^{00}_{\lL}+
   \beta_{\lL}]\restrictedto_{\pa\Omega}}{dS_{\pa\Omega}}, \\
   |u_{\lL}| &\to 0 \text{ as } |x| \to \infty \, .
 \end{aligned}
 \end{equation}

The system of equations in~\eqref{eqn4.6.100} has a unique solution that can be
expressed in terms of the single layer potential as follows,
 \begin{equation}
 \begin{aligned}
 u_{\lL} &= -\cS\left[\frac{[\bEta^{00}_{\lL}+
   \beta_{\lL}]\restrictedto_{\pa\Omega}}{dS_{\pa\Omega}} \right] \quad \text{ in } \Omega^c \\
   v_{\lL} &= -\cS\left[\frac{[\bEta^{00}_{\lL}+
   \beta_{\lL}]\restrictedto_{\pa\Omega}}{dS_{\pa\Omega}} \right] \quad \text{ in } \Omega .
   \end{aligned}
 \end{equation}
The fact that this is a solution follows from the properties of single layer
potentials given in~\Cref{subsec:lp}, in particular the jump formula for the
normal derivative.

This construction can then be used to eliminate $\star d u_{\lL}$ and $\star d v_{\lL}$ from the ansatz. 
Let $B$ denote the operator, which acts on closed 2-forms $\theta$ 
defined in $\overline{\Omega}$ given by 
\begin{equation}
   B\theta=-\star d \cS \circ i^*\theta,\text{ where }
   i^*\theta=\frac{\theta\restrictedto_{\pa\Omega}}{dS_{\pa\Omega}}.
 \end{equation}
 Then
 \begin{equation}
 \bEta^{0}_{\lL} =  \bEta^{00}_{\lL} + \beta_{\lL} + B(\bEta^{00}_{\lL} + \beta_{\lL}) \, .
 \end{equation}

 Let $\cZ^2(\Omega)$ denote closed 2-forms defined in $\Omega.$  Note that the range of $B$ is contained in $\cZ^2(\Omega).$  The theorem below establishes certain useful properties of $B.$ 
 \begin{theorem}\label{thm1}
   The operator $B:L^2(\Omega;\cZ^2)\to L^2(\Omega;\cZ^2)$ is
   bounded and
   self adjoint; there is an $0<m<1,$ so that for $\theta\neq 0,$ it satisfies the estimate
   \begin{equation}\label{eqn2.51.2}
     -1<-m<\frac{\langle B\theta,\theta\rangle_{\Omega}}
     {\langle \theta,\theta\rangle_{\Omega}} \leq 0.
   \end{equation}
 \end{theorem}
   \begin{proof}
   From~\Cref{app_tr_lem}, we know that $i^{*}$ is a bounded operator
   from $L^{2}(\Omega; \cZ^2)$  to $H^{-1/2}(\pa \Omega).$
   The single layer operator is a bounded operator from $H^{-1/2}(\pa \Omega)$ to
   $H^{1}(\Omega).$ Finally the exterior derivative is a bounded map from 
   $H^{1}(\Omega)$ to $L^{2}(\Omega)$.
 The boundedness of  $B:L^2(\Omega;\cZ^2)\to L^2(\Omega;\cZ^2)$ is thereby proved as $B$ is a composition of the following bounded maps
   \begin{equation}\label{eqn4.27.105}
     L^2(\Omega;\cZ^2)\overset{i^*}{\longrightarrow} H^{-\frac
       12}(\pa\Omega)\overset{-\cS} {\longrightarrow} H^1(\Omega)\overset{\star
       d}{\longrightarrow} L^2(\Omega;\cZ^2).
   \end{equation}
   
   To prove that it is self adjoint, we
  note that integration by parts shows that
  \begin{equation}
    \langle B\theta_1,\theta_2\rangle_{\Omega}=-\langle S\tau_1,\tau_2\rangle_{\pa\Omega}
  \end{equation}
  where $i^*\theta_i=\tau_i.$ The self adjointness of $B$ follows from
  classical fact that $S$ is self adjoint. It is also well known that
  $S:H^{-\frac 12}(\pa\Omega)\to H^{\frac 12}(\pa\Omega),$ is a
  positive operator, see~\cite{NedelecPlanchard1973}, showing that $B\leq 0.$

  It remains to prove the lower bound. Recall that single ($S$) and double layer potentials ($D$),
  and the normal derivative of the single layer potential  ($S'$) are related by the  Calderon identity
  $$DS = SS',$$
   see~\cite{Nedelec}, for example.
  As $D^*=S',$ this implies that $S^{\frac 12}S' S^{-\frac 12}$ is a compact, self
  adjoint operator,  of order -1. Here $S^{\frac 12}$ is the positive
  square root of $S.$
Hence there is a complete orthonormal basis, $\{h_j:\: j\geq 0\},$ for $L^2(\pa\Omega),$
   which satisfies
  \begin{equation}
    S^{\frac 12}S' S^{-\frac 12}h_j=\mu_j h_j.
  \end{equation}
  As $S^{\frac 12}S' S^{-\frac 12}$ is a pseudodifferential operator
  of order $-1$ the eigenfunctions belong to
  $\cC^{\infty}(\pa\Omega).$ If we let $g_j=S^{-\frac 12}h_j,$ then we
  see that  $S'g_j=\mu_jg_j,$ and
  \begin{equation}
    \langle S'g_j,g_j\rangle_{\pa\Omega}= \mu_j\langle g_j,g_j\rangle_{\pa\Omega}.
  \end{equation}
  It is a classical fact that the spectrum of $S'$ is contained
  in $[-\frac 12,\frac 12),$ hence $\mu_j\in [-\frac 12,\frac 12),$
    see~\cite{ColtonKress}. We let
    $\mu_0=-\frac 12,$ which is easily seen to be a simple
    eigenvalue. The harmonic function $\cS g_0$ has vanishing normal
    derivative and is therefore constant in $\Omega.$

 If we define
    \begin{equation}
      \theta_j=\star d\cS g_j,\text{ for }j>0,
    \end{equation}
    then
    \begin{equation}
      \theta_j\restrictedto_{\pa\Omega}=\left(\mu_j+\frac 12\right)g_j dS_{\pa\Omega}.
    \end{equation}
    Note that $d\cS g_0=0.$ These forms satisfy
    \begin{equation}
      B\theta_j=-\left(\mu_j+\frac 12\right)\theta_j.
    \end{equation}
    The non-trivial lower bound in~\eqref{eqn2.51.2} follows from that facts that, for all $j,$
    \begin{equation}
      0\leq \mu_j+\frac 12<m<1,
    \end{equation}
    and the $\Span\{\theta_j\}$ is dense in $[\Ker B]^{\bot}.$ To see this, suppose that
    $\langle\theta_j,\theta\rangle=0,$ for all $j>0.$ This is equivalent to
    \begin{equation}
     0= \langle B\theta_j,\theta\rangle_{\Omega}=(\mu_j+\frac
     12)\langle S g_j,i^*\theta\rangle_{\pa\Omega}=
     (\mu_j+\frac
     12)\langle h_j,S^{\frac 12}i^*\theta\rangle_{\pa\Omega}.
    \end{equation}
    The orthogonal complement of $\Span\{h_j:\:j>0\}$ are multiples  of $h_0.$
    If $S^{\frac 12}i^*\theta=ch_0,$ then $i^*\theta=ag_0,$ and
    therefore $B\theta=0.$ This completes the proof of the theorem.

 \end{proof}
   \begin{remark}
     The $\Ker B$ is infinite dimensional, containing the subspace
     $\{\theta\in L^2(\Omega;\cZ^2):
     \: i^*(\theta)=0\}.$
   \end{remark}
   Using the decomposition of $B$ in~\eqref{eqn4.27.105} and Lemma~\ref{tr_lem}
   we easily prove the following corollary.
   \begin{corollary}\label{cor1}
     For any $s\geq 0$ there is a constant $C_s$ so that for any closed 2-form $\theta$
     \begin{equation}
       \|B\theta\|_{H^s(\Omega)}\leq C_s\|\theta\|_{H^s(\Omega)}\text{ and }
        \|B\theta\|_{H^s(\Omega)}\leq C_s\|i^*\theta\|_{H^{s-\frac 12}(\pa\Omega)}
     \end{equation}
   \end{corollary}

\subsection{Analysis of $\bEta^{00}_{\lL}$} \label{sec:a-analysis}
Finally, we turn our attention to the analysis of the closed 2-form $\bEta^{00}_{\lL}$.  
Combining the
construction of $u_{\lL}$, $v_{\lL}$, and $\beta_{\lL}$, we note that~\eqref{eqn2.1.2} implies that $\bEta^{00}_{\lL}$  satisfies
 \begin{equation}\label{eqn2.46.05}
\cA_{\lL}\bEta^{00}_{\lL} :=  \left[dd^*+\frac{\Id +B}{\lL^2}\right]\bEta^{00}_{\lL}=
    -\left[dd^*+\frac{1}{\lL^2}\right]\beta_{\lL}-\frac{1}{\lL^2}B\beta_{\lL}\,,\\
 \end{equation}
Note that the above relation does not vacuously imply that $\bEta^{00}_{\lL} = -\beta_{\lL}$ since
$\bEta^{00}_{\lL}$ must have $0$ tangential data. Thus the operator $\cA_{\lL}$ must be understood in
an appropriate domain.

We begin with the Laplacian, $\Delta_A=dd^*+d^*d$ on the domain
 \begin{equation}
 	\Dom(\Delta_A) = \{ \eta \in H^2(\Omega; \Lambda^2): \star \eta \restrictedto_{\pa \Omega} = \star d \eta \restrictedto_{\pa \Omega} = 0 \}.
 \end{equation}
$\Delta_A$ is a non-negative, self adjoint operator, the operator
 $\Delta_A+\mu$ is invertible, for $\mu>0,$ with the estimate
 \begin{equation}
   \|[\Delta_A+\mu]^{-1}\|_{L^2(\Omega)\to L^2(\Omega)}\leq \frac{1}{\mu}.
 \end{equation}
 Furthermore, it is a consequence of the Hodge decomposition that,
 for $\theta\in\cZ^2(\Omega),$ we have
 \begin{equation}
   [\Delta_A+\mu]^{-1}\cZ^2(\Omega)\subset  H^2(\Omega;\cZ^2).
 \end{equation}
 Thus if $\theta\in L^2(\Omega;\cZ^2),$
 i.e. $d\theta=0,$ in the sense of distributions, then
 \begin{equation}
   (dd^*+\mu)[\Delta_A+\mu]^{-1}\theta=\theta,\text{ and
   }(\star[\Delta_A+\mu]^{-1}\theta)\restrictedto_{\pa\Omega}=0. 
 \end{equation}

 The operator $\cA_{\lL}$ is quite similar
 \begin{proposition}
   The operator $\cA_{\lL}$ is invertible for any $\lL>0,$ with
   \begin{equation}
     \cA_{\lL}^{-1}:L^2(\Omega;\cZ^2)\longrightarrow \Dom(\Delta_A)\cap \cZ^2(\Omega).
   \end{equation}
There are positive constants, $C_0, C_1, C_2,$ so that the inverse satisfies the   estimates
   \begin{equation}\label{eqn2.68.055}
   \|\cA_{\lL}^{-1}\|_{L^2\to L^2}\leq C_0\lL^2,\,
   \|\cA_{\lL}^{-1}\eta\|_{H^1}\leq C_1\lL\|\eta\|_{L^2},\,
    \|\cA_{\lL}^{-1}\|_{L^2\to H^2}\leq C_2.
 \end{equation}
 \end{proposition}
\begin{proof}
 As the operator $B$ is bounded, self adjoint, and has its range in closed
 2-forms, it follows from elementary perturbation theory that $\cA_{\lL}$ is
 also self adjoint on $\Dom(\Delta_A)\cap \cZ^2(\Omega),$ see~\cite{Kato}.
 Theorem~\ref{thm1} shows that for
 $\theta\in\Dom(\Delta_A)\cap\cZ^2(\Omega)$ we have
 \begin{equation}\label{eqn2.65.2}
   \langle \cA_{\lL}\theta,\theta\rangle_{\Omega}\geq
   \langle d^*\theta,d^*\theta\rangle_{\Omega}+\frac{(1-m)}{\lL^2}\langle \theta,\theta\rangle_{\Omega},
 \end{equation}
 as $1-m>0,$ this shows that the nullspace of $\cA_{\lL}$ is trivial for any
 $\lL\in (0,\infty).$ As $\cA_{\lL}$ is self adjoint, it is invertible for any
 $\lL>0,$ with
 \begin{equation}
   \cA_{\lL}^{-1}:L^2(\Omega;\cZ^2)\longrightarrow \Dom(\Delta_A)\cap \cZ^2(\Omega).
 \end{equation}
 From~\eqref{eqn2.65.2} it follows that
 \begin{equation}\label{eqn2.68.05}
   \|\cA_{\lL}^{-1}\|_{L^2\to L^2}\leq C_0\lL^2.
 \end{equation}

 The estimate in~\eqref{eqn2.65.2} and standard elliptic theory imply
 that, for $\theta\in\Dom(\Delta_A)$ we have the estimate
 \begin{equation}
   \begin{split}
   \|\theta\|_{H^1(\Omega)}^2&\leq C \langle
   \cA_{\lL}\theta,\theta\rangle_{\Omega}\\
   &\leq C\|\cA_{\lL}\theta\|_{L^2(\Omega)}\|\theta\|_{L^2(\Omega)}.
     \end{split}
 \end{equation}
 Setting $\theta=\cA_{\lL}^{-1}\eta,$ we see  that
 \begin{equation}
   \|\cA_{\lL}^{-1}\eta\|_{H^1}\leq C\lL\|\eta\|_{L^2},
 \end{equation}
 hence
 \begin{equation}\label{eqn2.71.05}
    \|\cA_{\lL}^{-1}\|_{L^2\to H^1}\leq C_1\lL.
 \end{equation}
 Using these estimates and the equation $\cA_{\lL}\theta=\eta,$ we
 easily show that
 \begin{equation}
   \|\cA_{\lL}^{-1}\|_{L^2\to H^2}\leq C_2.
 \end{equation}
\end{proof}

 We now set
 \begin{equation}
   \bEta^{00}_{\lL}=-\cA_{\lL}^{-1}\left[(dd^*+\lL^{-2})\beta_{\lL}+\lL^{-2}B\beta_{\lL}\right].
 \end{equation}
 From the second estimate in~\eqref{eqn2.15.2} it follows that
 \begin{equation}
   \|B \beta_{\lL}\|_{L^2(\Omega)}\leq C\lL,
 \end{equation}
 and therefore we have the estimates
 \begin{equation}\label{eqn4.52.105}
   \|\bEta^{00}_{\lL}\|_{L^2(\Omega)}\leq C\lL, \text{ and }
   \|\bEta^{00}_{\lL}\|_{H^1(\Omega)}<C.
 \end{equation}
 A standard interpolation inequality then implies that
 \begin{equation}\label{eqn4.54.106}
   \|\bEta^{00}_{\lL}\|_{H^{\frac 12}(\Omega)}\leq C\sqrt{\lL},
 \end{equation}
 which, using  Lemma~\ref{tr_lem}, implies that
 \begin{equation}
   \|i^*\bEta^{00}_{\lL}\|_{L^2(\pa\Omega)}\leq C\sqrt{\lL}.
 \end{equation}

 Now recall that $u_{\lL},v_{\lL}\restrictedto_{\pa\Omega} = -S[q],$ where
 $q=i^*(\bEta^{00}_{\lL}+\beta_{\lL}).$ It follows from the jump-relations for
 the single layer potential in~\cref{eq:identites}, the fact that $S$ is a
 pseudodifferential operator of order $-1,$ Green's identity,~\eqref{eqn4.52.105}
 and~\Cref{tr_lem} that
 \begin{equation}\label{eqn4.55.105}
   \begin{split}
   \|du_{\lL}\|^2_{L^2(\Omega^c)}+\|dv_{\lL}\|^2_{L^2(\Omega)}&=\langle
   Sq,q\rangle_{\pa\Omega}\\
   &\leq\|Sq\|_{H^{\frac 12} (\pa \Omega)}\|q\|_{H^{-\frac 12}(\pa \Omega)}\leq
   C\|q\|^2_{H^{-\frac 12}(\pa \Omega)}\\
   &\leq C(\|\bEta^{00}_{\lL}\|_{L^2(\Omega)}+\lL)^2\leq C\lL^2.
   \end{split}
 \end{equation}
 Using the estimates above and Lemma~\ref{tr_lem} we can also show that
 $$\|S[q]\|_{H^1(\pa\Omega)}\leq C\|q\|_{L^2(\pa\Omega)} \leq C'\sqrt{\lL},$$
 and therefore
 \begin{equation}\label{eqn4.57.106}
   \|v_{\lL}\|_{H^{\frac 32}(\Omega)}\leq C\sqrt{\lL},\quad
  \|u_{\lL}\|_{H^{\frac 32}_{\loc}(\Omega^{c})}\leq C\sqrt{\lL},
 \end{equation}
 which shows that
 \begin{equation}\label{eqn4.55.101}
   \|\star dv_{\lambda}\|_{L^2(\pa\Omega)}\leq C\sqrt{\lL},\quad
   \|\star du_{\lambda}\|_{L^2(\pa\Omega)}\leq C\sqrt{\lL}.
 \end{equation}

 \begin{remark}
Note that for fixed $\lL$, the fields $\bEta_{\lL}^+, \bEta_{\lL}^0,
\bj_{\lL}^{0}$ constructed above satisfy the London equations along with the
continuity of the magnetic fields across the boundary, the topological
conditions, and the outgoing condition at $\infty$. Hence this gives a new proof
of the existence of a solution to the London boundary value problem.
\end{remark}

 \subsection{Initial Estimates on $\bEta_{\lL}^{+}, \bEta_{\lL}^{0}, \bj_{\lL}^{0}$}
 We can now consider the rate at which the solutions $(\bEta^{\tot
   +}_{\lL},\bEta^0_{\lL},\bj^0_{\lL})$ converge to their limiting values as
 $\lL\to 0^+.$ The magnetic fields converge   in $L^2(\Omega)$ to $(\bEta^{\In}+\bEta^{\out +},0).$
The currents converge, in the sense of distributions, to a current
 sheet on $\pa\Omega.$
 Recall that the magnetic fields are given by
 \begin{equation}
   \begin{split}
     \bEta^{\tot +}_{\lL}&=\bEta^{\In}+\bEta^{\out+}+\star du_{\lL}\text{ in }\Omega^c,\\
     \bEta^0_{\lL}&=\beta_{\lL}+\bEta^{00}_{\lL}+\star dv_{\lL}\text{ in }\Omega.
     \end{split}
 \end{equation}
 Our initial estimates are summarized in the following proposition.
 \begin{proposition}
   As $\lL\to 0^+$ the exterior fields satisfy
   \begin{equation}
 \label{eq:normal_ext_est}
    \|\bEta^{\tot +}_{\lL}\restrictedto_{\pa\Omega}\|_{L^2(\pa\Omega)}\leq
    C\sqrt{\lL}\text{ and }
    \|\bEta^{\tot +}_{\lL}-(\bEta^{\In}+\bEta^{\out+})\|_{L^2(\Omega)}\leq
    C\lL.
   \end{equation}
   The interior magnetic field, $\bEta^{0}_{\lL},$  tends to $0$ in $L^2(\Omega)$ as 
$\lL \to 0^+,$ with
  \begin{equation}\label{eqn4.61.105}
   \|\beta_{\lL}\|_{L^2(\Omega)}\propto \sqrt{\lL},\text{ and }
   \|\bEta^{00}_{\lL}+\star dv_{\lL}\|_{L^2(\Omega)}\leq C\lL,
 \end{equation}
   \end{proposition}
 
 \begin{proof}
   As $[\bEta^{\In}+\bEta^{\out+}]_{\pa\Omega}=0,$ the first estimate in~\eqref{eq:normal_ext_est}
   follows from~\eqref{eqn4.55.101} and the second estimate from~\eqref{eqn4.55.105}. The estimates
   in~\eqref{eqn2.15.2} and~\eqref{eqn4.55.105} imply~\eqref{eqn4.61.105}.
 \end{proof}
 \begin{remark}
   Using~\eqref{eqn4.54.106}  and~\eqref{eqn4.57.106} we can also show that  $\|\bEta^{00}_{\lL}+\star
   dv_{\lL}\|_{H^{\frac 12}(\Omega)}\leq C\sqrt{\lL}.$
 \end{remark}

 The London equations imply that the current  is given by
 $$\bj^{0}_{\lL}=d^*\bEta^0_{\lL}=d^*\beta_{\lL}+d^*\bEta^{00}_{\lL};$$ the
 analysis of its behavior as $\lL\to 0^+$ requires some care.
 \begin{proposition}
   Let $\Psi$ be a smooth 2-form defined in $\overline{\Omega}.$
   \begin{equation}
     \int_{\Omega}\bj^0_{\lL}\wedge\Psi=\int_{\pa\Omega}\star\beta_{\lL}\wedge\star\Psi+O(\lL\|\Psi\|_{\cC^1(\Omega)}).
   \end{equation}
 \end{proposition}
 \begin{proof}
To estimate the contribution from $d^*\bEta^{00}_{\lL}$ we integrate by parts to
obtain
  \begin{equation}
 \int_{\Omega}d^*\bEta^{00}_{\lL}\wedge
 \Psi=\int_{\Omega}\bEta^{00}_{\lL}\wedge
 d^*\Psi+\int_{\pa\Omega}\star\bEta^{00}_{\lL}\wedge\star\Psi.
  \end{equation}
  From~\eqref{eqn4.52.105}  we see that the volumetric
  term of the right hand
  side is $O(\lL\|\Psi\|_{H^1(\Omega)});$  Lemma~\ref{tr_lem}
  and~\eqref{eqn4.52.105} show that the boundary term is bounded by
  $$C\|\bEta^{00}_{\lL}\|_{H^{-\frac 12}(\pa\Omega)}\|\star\Psi\restrictedto_{\pa\Omega}\|_{H^{\frac
      12}(\pa\Omega)}=O(\lL \|\Psi\|_{H^{1}(\Omega)}).$$

  The 1-form $d^*\beta_{\lL}$ converges to a current sheet on $\pa\Omega.$
  Integrating by parts, as before, we obtain from~\eqref{eqn4.8.101}
  and~\eqref{eqn4.16.101} that
  \begin{equation}
    \begin{split}
    \int_{\Omega}d^*\beta_{\lL}\wedge \Psi=&
    \int_{\pa\Omega}\star\beta_{\lL}\wedge \star\Psi+
    \int_{\Omega}\beta_{\lL}\wedge d^*\Psi\\
    &= \int_{\pa\Omega}\star[\bEta^{\out +}+\bEta^{\In}]\wedge \star\Psi+ O(\lL\|d^*\Psi\|_{L^{\infty}(\Omega)}).
    \end{split}
  \end{equation}
  Here we use the fact that $\|\beta_{\lL}\|_{L^1(\Omega)}<C\lL,$
  see~\eqref{eqn4.21.106}. Both error terms are
  $O(\lL\|\Psi\|_{\cC^1(\Omega)}).$
 \end{proof}
 
  This verifies our assertions that $\bEta^{\tot +}_{\lL}$ converges to
  $\bEta^+_0=\bEta^{\In}+\bEta^{\out +},$ $\bEta^0_{\lL}$ converges to $0,$ and $\bj^0_{\lL}$ to a current
  sheet on $\pa\Omega.$
 The limiting total field in the exterior is $\bEta^{\tot
   +}_0=\bEta^{\In}+\bEta^{\out +},$ which is the field that would result from
 scattering $\bEta^{\In}$ off of a perfect electrical conductor lying in
 $\Omega.$ The interior magnetic field is zero, and the limiting current,
 $\bj^0_0,$ is the current sheet along $\pa\Omega$ that exactly cancels the
 normal component of $\bEta^{\In}.$ This is the current sheet
 needed to expel the magnetic field from the interior of $\Omega.$ From this
 perspective, in the limit that $\lL\to 0^+,$ a ``London equations''
 superconductor behaves like a perfect electrical conductor, which exhibits the
 strict Meissner effect.

\subsection{A Better Error Estimate}\label{ss2.3}
Our analysis of the limiting form for the Debye source
representation, see~\cite{EpRa3},  suggests that the $L^2$-norm of the normal components
$$\bEta^{\tot +}_{\lL}\restrictedto_{\pa\Omega}=[\bEta^{+}_{\lL} + \star d
  u_{\lL}]\restrictedto_{\pa\Omega}=[\bEta^{00}_{\lL}+\star
  dv_{\lL}+\beta_{\lL}]\restrictedto_{\pa\Omega}$$ should be $O(\lL),$ rather
than the error bound of $\sqrt{\lL}$ given above
in~\eqref{eq:normal_ext_est}. As will be clear from the discussion below, the
$\sqrt{\lL}$-decay rate is an artifact of our ansatz \eqref{top_ansatz}. We
prove the following improved result
\begin{theorem}\label{thm2.107}
  For any $0<\alpha<\frac 12$  there is a
  $C_{\alpha}$ so that
\begin{equation}\label{eqn2.129.77}
   \|(\bEta^{00}_{\lL}+\star dv_{\lL}+\beta_{\lL})\restrictedto_{\pa\Omega}\|_{L^2(\pa\Omega)}\leq C_{\alpha}\lL^{\frac 12+\alpha}.
\end{equation}
\end{theorem}
\begin{remark}
In this section we give a simplified discussion of the argument, focusing on the
case of zero incoming field wherein $\bEta^{+}_{\lL}\restrictedto_{\pa \Omega} =
0.$ We also make several simplifications in the calculations below. A
complete proof, without simplifying assumptions, is given  in Section~\ref{sec3.3.12}.
\end{remark}
\begin{proof}[Proof of Theorem]
We need to estimate
\begin{equation}
 [\bEta^{00}_{\lL}+\star dv_{\lL}+\beta_{\lL}]\restrictedto_{\pa\Omega}.
\end{equation}
The last term
\begin{equation}
  \beta_{\lL}\restrictedto_{\pa\Omega}=E_{\lL}(0)d\gamma,
\end{equation}
which satisfies
$\| \beta_{\lL}\restrictedto_{\pa\Omega}\|_{H^s(\pa\Omega)}\leq C_s\lL,$ for any
$s.$ Thus we are left to consider
\begin{equation}
  \bEta^{00}_{\lL}+\star dv_{\lL}=(\Id+B)\bEta^{00}_{\lL}+B\beta_{\lL}.
\end{equation}
We first prove an estimate for $ \|\bEta^{00}_{\lL}+\star
dv_{\lL}\|_{L^2(\Omega)}.$ Using~\eqref{eqn2.46.05}
\begin{equation}
 \bEta^{00}_{\lL}+\star dv_{\lL}=  -(\Id+B)\cA_{\lL}^{-1}
    \left(\left[dd^*+\frac{1}{\lL^2}\right]\beta_{\lL}+\frac{1}{\lL^2}B\beta_{\lL}\right)+B\beta_{\lL}.
\end{equation}
Since $\|\left[dd^*+\lL^{-2}\right]\beta_{\lL}\|_{L^2(\Omega)}\leq C/\sqrt{\lL},$
it follows from~\eqref{eqn2.68.05} that
\begin{equation}
  \left\| (\Id+B)\cA_{\lL}^{-1}
    \left(\left[dd^*+\frac{1}{\lL^2}\right]\beta_{\lL}\right)  \right\|_{L^2(\Omega)}\leq
    C\lL^{\frac 32}.
\end{equation}

This leaves the term
\begin{equation}
B\beta_{\lL} -(\Id+B)\left[\lL^2dd^*+\Id+B\right]^{-1}B\beta_{\lL},
\end{equation}
which  only depends on
\begin{equation}
  B\beta_{\lL}=-E_{\lL}(0)\star d\cS(d\gamma\restrictedto_{\pa\Omega}).
\end{equation}
Note that $E_{\lL}(0)\leq C\lL;$ we let $\chi=\star d\cS(d\gamma\restrictedto_{\pa\Omega}),$ which does
\emph{not} depend on $\lL,$ and set
\begin{equation}
D_{\lL}=\|\chi-(\Id+B)\left[\lL^2dd^*+\Id+B\right]^{-1}\chi\|_{L^2(\Omega)}.
\end{equation}
It is easy to see that
\begin{equation}
  \lim_{\lL\to 0^+}D_{\lL}=0;
\end{equation}
the \emph{rate} at which $ D_{\lL}$ goes to zero determines the
size of this term.
For the moment, we consider a simplified problem,
which  illustrates the idea, by replacing $D_{\lL}$ with
\begin{equation}
  \tD_{\lL}=\|\chi-\left[\lL^2dd^*+\Id\right]^{-1}\chi\|_{L^2(\Omega)}.
\end{equation}
The argument proving the analogous estimates for $D_{\lL}$ is given in Section~\ref{sec3.3.12}.

\begin{lemma}\label{lem3.106}
  For any $0<\alpha<\frac 12$ there is a constant $C_{\alpha}$ so that
  \begin{equation}
    \tD_{\lL}\leq C_{\alpha}\lL^{\alpha}.
  \end{equation}
\end{lemma}

\begin{proof}[Proof of Lemma~\ref{lem3.106}]
  To prove this lemma we use the spectral representation of $dd^*$ on
  $\Dom(\Delta_A)\cap\cZ^2(\Omega).$ 
Let $\{(\eta_j,\mu_j):\:j=1,\dots\}$ denote a complete orthonormal basis of
eigenpairs for $\Delta_A\restrictedto_{\cZ^2}.$ We then have an expansion for
$\chi:$
\begin{equation}
  \chi=\sum_{j=1}^{\infty}a_j\eta_j.
\end{equation}
In terms of this expansion
\begin{equation}
\left[\lL^2dd^*+\Id\right]^{-1}
\chi=\sum_{j=1}^{\infty}\frac{a_j\eta_j}{1+\lL^2\mu_j}\in\Dom(\Delta_A),\text{ and }
\tD_{\lL}^2=\sum_{j=1}^{\infty}\frac{\lL^4\mu_j^2|a_j|^2}{(1+\lL^2\mu_j)^2}.
\end{equation}
For $0\leq s,$ we recall that the
form $\chi\in\Dom(\Delta_A^s)$ if and only if
\begin{equation}
  \sum_{j=0}^{\infty}\mu_j^{2s}|a_j|^2<\infty.
\end{equation}\label{eqn4.15.101}
Notice that if $\chi\in\Dom(\Delta_A^s),$ then
\begin{equation}\label{eqn4.79.106}
  \sum_{j=1}^{\infty}\frac{\lL^4\mu_j^2|a_j|^2}{(1+\lL^2\mu_j)^2}\leq
  \lL^{4s}\sum_{j=1}^{\infty}\frac{\lL^{4-4s}\mu_j^{2-2s}}{(1+\lL^2\mu_j)^2}\mu_j^{2s}|a_j|^2\leq
  \lL^{4s}\sum_{j=1}^{\infty}\mu_j^{2s}|a_j|^2.
\end{equation}
The last estimate follows as
\begin{equation}
  \frac{x^{2(1-s)}}{(1+x)^2}\leq 1,\text{ for }x\in [0,\infty),\, s\in
    [0,1].
\end{equation}

The rate at which $ \tD_{\lL}$ tends to zero is therefore
related to the largest $s$ for which $\chi\in\Dom(\Delta_A^s).$ In
Proposition~\ref{prop1} we show that $\chi\in\Dom(\Delta_A^s)$ for any
$s<\frac 14.$ Assuming that, \eqref{eqn4.79.106} shows that, for any
$0<\alpha<\frac 12,$ there is a $C_{\alpha}$ so that
\begin{equation}
  \tD_{\lL}\leq C_{\alpha}\lL^{\alpha}.
\end{equation}
\end{proof}

In Section~\ref{sec3} we show that the estimate
established for $\tD_{\lL},$ also holds for
$D_{\lL}.$ Hence for any
$\alpha<\frac 12,$ there is a constant $C_{\alpha}$ so that
\begin{equation}\label{eqn4.77.100}
\|B\beta_{\lL} -(\Id+B)\left[\lL^2dd^*+\Id+B\right]^{-1}B\beta_{\lL}\|_{L^2(\Omega)}\leq
C_{\alpha}\lL^{1+\alpha},
\end{equation}
and therefore
\begin{equation}
  \|\bEta^{00}_{\lL}+\star dv_{\lL}\|_{L^2(\Omega)}\leq C'_{\alpha}\lL^{1+\alpha}.
\end{equation}

To prove $\chi\in\Dom(\Delta_M^s)$ for $0\leq s<\frac 14$ we use  the following general result:
\begin{proposition}\label{prop1}
  Let $\Omega$ be a smoothly bounded domain in $\bbR^n$ and let $M$ be
  a self adjoint, first order boundary condition for $dd^*+d^*d$
  acting on $k$-forms. We denote the self adjoint operator by
  $\Delta_M,$ with $\Dom(\Delta_M)\supset H^2_0(\Omega;\Lambda^2).$ If
  $\theta\in\cC^{\infty}(\overline{\Omega};\Lambda^k),$ then
  $\theta\in\Dom(\Delta_M^s)$ for $0\leq s<\frac 14.$
\end{proposition}
\begin{remark}
  This result, in the scalar case, appears in~\cite{Fu1967}. Our method of proof
  is suggested by the discussion in Section 5.A of~\cite{Taylor1}.
\end{remark}
\begin{proof}[Proof of Proposition~\ref{prop1}]
  The proof follows using standard results in interpolation theory. For $0<s<1,$
  the $\cD_s=\Dom(\Delta_M^s)$ is the space obtained by interpolation between
  $\cD_1=\Dom(\Delta_M),$ with the graph norm $\|u\|_{\cD_1}=\|\Delta
  u\|_{L^2}+\|u\|_L^2,$ and $L^2(\Omega)=\Dom(\Delta_M^0),$ with the standard
  norm.  The norm on $\cD_s$ is
  \begin{equation}
    \|u\|_{\cD_s}=\|\Delta_M^su\|_{L^2}+\|u\|_{L^2}.
  \end{equation}

  The space $H^2_0(\Omega)$ is a subspace of $\Dom(\Delta_M)$ and, for $u\in
  H^2_0(\Omega),$ we have the trivial estimate
  \begin{equation}
    \|\Delta_M u\|_{L^2}+\|u\|_{L^2}\leq C\|u\|_{H^2_0}.
  \end{equation}
  For $2s\neq \frac 12,\,\frac 32,$ the spaces $H^{2s}_0(\Omega)$ are the spaces obtained by interpolation between
  $H^2_0(\Omega)$ and $L^2(\Omega),$ see~\cite{Taylor1}[\S 4.5]. Calderon's complex  interpolation theory then
  implies that the inclusion map $j:H^2_0(\Omega)\hookrightarrow \Dom(\Delta_M)$
  extends as a bounded map to
  \begin{equation}
    j:H^{2s}_0(\Omega)\hookrightarrow \Dom(\Delta^s_M)
  \end{equation}
  for $s\in [0,1]\setminus \{\frac 14,\frac 34\}.$
  see~\cite{Calderon1964}.
  Thus, for these values of $s,$ there is a constant $C_s$ so that for $u\in H^{2s}_0(\Omega)$
  we have the estimate
  \begin{equation}
    \|\Delta_M^s u\|_{L^2}\leq C_s\|u\|_{H^{2s}_0}.
  \end{equation}
  The proof is completed by recalling that, for $0<s<\frac 14,$ the
  spaces $H^{2s}_0(\Omega)$ and $H^{2s}(\Omega)$ are equal,
  see~\cite{Taylor1}[\S 4.5], and therefore, for $s$ in this range,
  $\cC^{\infty}(\overline{\Omega})\subset\ H^{2s}(\Omega)\subset
  \Dom(\Delta_M^s).$
\end{proof}

To improve the estimate in~\eqref{eq:normal_ext_est} on the normal components of the
magnetic field restricted to the boundary, we also need to estimate
$\|\bEta^{00}_{\lL}+\star dv_{\lL}\|_{H^1(\Omega)}.$
\begin{proposition}\label{prop5.106}
  For any $0<\alpha<\frac 12,$ there is a constant $C_{\alpha}$ so that
  \begin{equation}
    \|\bEta^{00}_{\lL}+\star dv_{\lL}\|_{H^1(\Omega)}\leq C_{\alpha}\lL^{\alpha}.
  \end{equation}
\end{proposition}
\begin{remark}
  In the proof of this proposition we again simplify by replacing  the quantity
  $\|(\Id+B) \left[\lL^2dd^*+ \Id+B\right]^{-1}B\beta_{\lL}\|_{H^1(\Omega)}$
  with $ \|(\Id+B)[\lL^2dd^*+\Id]^{-1}B\beta_{\lL}\|_{H^1(\Omega)}$
  after~\eqref{eqn4.95.106}. The complete proof is given in Section~\ref{sec3.3.12}.
\end{remark}
\begin{proof}[Proof of Proposition~\ref{prop5.106}]
From Corollary~\ref{cor1}
it follows that $B:H^s(\Omega;\cZ^2)\to H^s(\Omega;\cZ^2)$ boundedly for any
$0\leq s.$ From~\eqref{eqn2.71.05} we conclude that
\begin{equation}
   \left\| (\Id+B)\cA_{\lL}^{-1}
    \left[dd^*+\frac{1}{\lL^2}\right]\beta_{\lL} \right\|_{H^1(\Omega)}\leq
    C\lL^{\frac 12}.
\end{equation}
The other term to be estimated is
\begin{equation}\label{eqn4.95.106}
  \|B\beta_{\lL}-(\Id+B)\left[\lL^2dd^*+\Id+B\right]^{-1}B\beta_{\lL}\|_{H^1(\Omega)}.
\end{equation}
Using Corollary~\ref{cor1} and~\eqref{eqn2.15.2} we see that  first term satisfies the estimate
\begin{equation}
  \|B\beta_{\lL}\|_{H^1(\Omega)}\leq C \|i^*\beta_{\lL}\|_{H^{\frac 12}(\pa\Omega)}\leq C\lL.
\end{equation}

As noted above, we replace the remaining term,
$\|(\Id+B)\left[\lL^2dd^*+\Id+B\right]^{-1}B\beta_{\lL}\|_{H^1(\Omega)}$ with
the simpler expression
$\|(\Id+B)\left[\lL^2dd^*+\Id\right]^{-1}B\beta_{\lL}\|_{H^1(\Omega)}.$ To
estimate it we use the fact that, for $u\in\Dom(\Delta_A)\cap\cZ^2(\Omega),$ with
\begin{equation}
  u=\sum_{j=1}^{\infty}a_j\eta_j,
\end{equation}
we have the estimate
\begin{equation}
  \|u\|^2_{H^1}\leq C\left[\sum_{j=1}^{\infty}\mu_j|a_j|^2+\|u_H\|^2_{L^2(\Omega)}\right],
\end{equation}
where $u_H$ is the projection into the nullspace of $\Delta_A.$ Because of the
$\|u_H\|_{L^2}$-term some additional care must be taken. To that end we write
\begin{equation}
  B\beta_{\lL}=E_{\lL}(0)\chi=E_{\lL}(0)(\chi_0+\chi_H),
\end{equation}
with $\chi_H$ the projection of $\chi$ into $\Ker\Delta_A,$ and
$\chi_0\bot\Ker\Delta_A.$ It is clear that
\begin{equation}
   \|B\beta_{\lL}-E_{\lL}(0)\chi_H\|_{H^1(\Omega)}\leq  C\lL.
\end{equation}

Observe that
\begin{equation}
  \left[\lL^2dd^*+\Id\right]^{-1}(\chi_0+\chi_H)=\left[\lL^2dd^*+\Id\right]^{-1}\chi_0+\chi_H,
\end{equation}
and therefore, Proposition~\ref{prop1} shows that, for any $\alpha<\frac
12,$ we have the estimate
\begin{equation}
  \begin{split}
    E^2_{\lL}(0)\left\|\left[\lL^2dd^*+\Id\right]^{-1}\chi-\chi_H\right\|^2_{H^1(\Omega)}&\leq
    C\left[E_{\lL}(0)^2\sum_{j=1}^{\infty}\frac{\mu_j|a_j|^2}{(1+\lL^2\mu_j)^2}\right]\\
  &\leq C\lL^{2\alpha}
  \sum_{j=1}^{\infty}\frac{(\lL^2\mu_j)^{1-\alpha}\mu_j^{\alpha}|a_j|^2}{(1+\lL^2\mu_j)^2}\\
    &\leq C\lL^{2\alpha}
  \sum_{j=1}^{\infty}\mu_j^{\alpha}|a_j|^2.
  \end{split}
\end{equation}
Given our simplifying assumption, it follows that, for any $\alpha<\frac 12,$
\begin{equation}\label{eqn4.100.106} 
  \|\bEta^{00}_{\lL}+\star dv_{\lL}\|_{H^1(\Omega)}\leq C_{\alpha}\lL^{\alpha}.
\end{equation}
\end{proof}

Using interpolation between this estimate and~\eqref{eqn4.77.100}, we see that
\begin{equation}
   \|\bEta^{00}_{\lL}+\star dv_{\lL}\|_{H^{\frac 12}(\Omega)}\leq
   C_{\alpha}\lL^{\frac 12+\alpha}.
\end{equation}
Lemma~\ref{tr_lem} along with
this estimate implies that, for any $\alpha<\frac 12,$ there is a
  $C_{\alpha}$ so that
\begin{equation}\label{eqn2.129.7}
   \|(\bEta^{00}_{\lL}+\star
   dv_{\lL})\restrictedto_{\pa\Omega}\|_{L^2(\pa\Omega)}\leq
   C_{\alpha}\lL^{\frac 12+\alpha} ,
\end{equation}
which completes the proof of the theorem.
\end{proof}

While this falls just short of the expected first order decay rate, it strongly suggests that the
$L^2$-norms of the normal components, 
$\bEta^{0}_{\lL},\bEta^{+}_{\lL}\restrictedto_{\pa\Omega},$ are bounded by a constant
times $\lL.$

This also show that the convergence of the current $\bj^0_{\lL}$ to the current
sheet defined as the limit of $d^*\beta_{\lL}$ as $\lL\to 0^+$ actually occurs
in the $L^2(\Omega)$-norm. As $d^*(\bEta^{00}_{\lL}+\star
dv_{\lL})=d^*\bEta^{00}_{\lL},$ it follows from~\eqref{eqn4.100.106} that, for
any $\alpha<\frac 12,$ there is a $C_{\alpha}$ so that
\begin{equation}
  \|\bj^0_{\lL}-d^*\beta_{\lL}\|_{L^2(\Omega)}=\|d^*\bEta^{00}_{\lL}\|_{L^2(\Omega)}\leq C_{\alpha}\lL^{\alpha},
\end{equation}

\begin{remark}
  We believe that the sum defining $D_{\lambda_{L}}$ is actually
  $O(\sqrt{\lambda_L}).$ In this remark we demonstrate this claim for the simple example of
  the Dirichlet Laplacian, $\Delta_D,$ on the unit disk $D_1\subset\bbR^2.$ The
  eigenfunctions of $\Delta_D$ are of the form
  \begin{equation}
    \varphi_{j,m}=n_{j,m}e^{im\theta}J_m(k_{j,m}r), \text{ for }m\in\bbZ,\,j\in\bbN.
  \end{equation}
  Here $J_m(z)$ are $J$-Bessel functions,  $n_{j,m}$ are normalizing
  constants, and $\{k_{j,m}\}$ are zeros:
  \begin{equation}
   n_{j,m} \int_{0}^1|J_m(k_{j,m}r)|^2rdr=1,\text{ and }J_m(k_{j,m})=0.
  \end{equation}
 If the zeros are ordered with $k_{j,m}<k_{j+1,m},$ then $m<k_{1,m}$ and they satisfy the
  asymptotic formul{\ae}
  \begin{equation}
    k_{j,m}\sim \left(j+\frac m2-\frac 14\right)\pi.
  \end{equation}
  As $J_m(z)\sim [c_+e^{iz}+c_-e^{-iz}]/\sqrt{z},$ we see that $n_{j,m}\approx
  \sqrt{k_{j,m}},$ see~\cite{NIST}[10].

Let $u\in\cC^{\infty}(\overline{D}_1),$ and write $u$ as a Fourier series
\begin{equation}
  u(r,\theta)=\sum_{j=-\infty}^{\infty}u_m(r)e^{im\theta}.
\end{equation}
We can also express $u$ in terms of eigenfunctions $\{\varphi_{j,m}\}$ as
\begin{equation}
  u(r,\theta)=\sum_{m=-\infty}^{\infty}\sum_{j=1}^{\infty}a_{j,m}\varphi_{j,m},
\end{equation}
where
\begin{equation}
  a_{j,m}=n_{j,m}\int_{0}^1u_m(r)J_m(k_{j,m})rdr.
\end{equation}
A simple analysis, using the asymptotics of the Bessel functions, the smoothness
of $u$ as a function of $\theta,$ and integration by parts, shows that there are
constants $\{c_m\},$ so that
\begin{equation}
  |a_{j,m}|\leq \frac{c_m}{k_{j,m}},
\end{equation}
where $c_m=O(m^{-N}),$ for any $N>0.$  Indeed this also gives the asymptotic
behavior of $a_{j,m}$ as $j,m\to\infty.$ Combining these observations, we see that
\begin{equation}
  D_{\lambda_L}^2\leq
  C\sum_{m=-\infty}^{\infty}\sum_{j=1}^{\infty}\frac{\lambda_L^4c_m^2k_{j,m}^2}
  {(1+\lambda_L^2k_{j,m}^2)^2}.
\end{equation}
Using the asymptotic formul{\ae} for the zeros $\{k_{j,m}\}$ we see that
\begin{equation}
  \sum_{j=1}^{\infty}\frac{\lambda_L^4k_{j,m}^2}
      {(1+\lambda_L^2k_{j,m}^2)^2}\leq
      C\int_{m}^{\infty}\frac{\lambda_L^4x^2dx}{(1+\lambda_L^2x^2)^2}\leq
      C\lambda_L\int_0^{\infty}\frac{s^2ds}{(1+s^2)^2},
\end{equation}
and therefore
\begin{equation}
  D_{\lambda_L}^2\leq C\lambda_L\sum_{m=-\infty}^{\infty}c_m^2<C'\lambda_{L},
\end{equation}
as claimed. On the other hand, it is also the case that in order for 
\begin{equation}
  \sum_{m=-\infty}^{\infty}\sum_{j=1}^{\infty}|a_{j,m}|^2k_{j,m}^{4\alpha}<\infty,
\end{equation}
we must take $\alpha<\frac 14.$
\end{remark}
A similar analysis applies to the ball in $\bbR^3,$ with the Laplacian acting on
closed 2-forms with vanishing tangential components. Note that the
smoothness of the data in directions tangent to the boundary played an essential role in
this analysis.

\section{The Limit $\lL\to 0^+$ for Thin Shells}\label{sec3}
In the previous sections and~\cite{EpRa1} we assume
that the superconducting material 
occupies a bounded region $\Omega,$ with
$\pa\Omega$  a smooth, \emph{connected}
surface embedded in $\bbR^3.$ In this section we consider a more
physically interesting situation by assuming that $\Omega_0$ is a
bounded region with a smooth boundary comprised of 2 smooth
surfaces, which is diffeomorphic to a thin neighborhood of an embedded
surface $\Sigma\hookrightarrow\bbR^3,$ i.e.
\begin{equation}
  \Omega_0\simeq
\{x:\:\dist(x,\Sigma)<\epsilon\}.
\end{equation}

This geometry arises naturally when modeling the confinement vessel of
standard fusion reactor designs like the Tokamak and
Stellarator. Since the current in a superconductor is largely confined
to a small neighborhood of the boundary, it makes sense to use such a
thin shell as a model for the confinement magnet. That is, we replace
the coil windings with a thin sheet of superconducting material. In
this way we can learn the ``equilibrium distribution'' of current that
generates a specified field geometry. In this analysis $\Omega_0$ is
fixed; as yet we are unable to let $\epsilon\to 0$ as $\lL\to0.$

\subsection{Problem setup}
 
The assumption that $\pa\Omega_0$ has 2 components implies that
$\overline{\Omega}_0^c$ also consists of two connected components. We let
\begin{equation}
  \overline{\Omega}_0^c=\Omega_+\cup\Omega_-,
\end{equation}
with $\Omega_+$ the unbounded component and $\Omega_-$ the bounded component. We let
\begin{equation}
  \pa\Omega_0=\pa\Omega_+\cup\pa\Omega_-,
\end{equation}
denote the corresponding components of $\pa\Omega_0.$ When we need to emphasize that
these boundary components are approached from $\Omega_0,$ or are oriented as the
boundary of $\Omega_0,$ we denote them as $\pa\Omega_{0+},\,\pa\Omega_{0-}.$ 

Let $\bEta^{0}, \bj^{0}$ denote the magnetic field and the current in the
superconducting region, $\Omega_0,$ as before, and let $\bEta^{\tot \pm}$ denote
the magneto-static fields in the two connected components, $\Omega_{\pm},$ of the
$\Omega_0^c.$ The fields in the interior of $\Omega_0$ satisfy the usual London
equations,
\begin{equation}
  d^*\bEta^0=-\frac{1}{\lL^2}\bj^0,\quad d\bj^0=\bEta^0,\text{ in }\Omega_0,
\end{equation}
and $\bEta^{\tot \pm}$ satisfy the equations of magneto-statics
\begin{equation}\label{eqn3.9}
  d\bEta^{\tot \pm}=0\text{ and } d^*\bEta^{\tot \pm}=0\text{ in }\Omega_{\pm}.
\end{equation}

In the standard ``scattering'' problem for this set-up the magnetic field in
$\Omega_+$ is split into an incoming magneto-static field $\bEta^{\In}_{+},$
generated by sources in $\Omega_+,$ which is defined in a neighborhood, of
$\overline{\Omega_0}\cup\Omega_-,$ and an outgoing scattered field, $\bEta^+,$
so that, near to $\pa\Omega_+$ we have
$$\bEta^{\tot +}=\bEta^{\In}_{+} + \bEta^+.$$
The assumption that $\bEta^+$ is an outgoing field means that
\begin{equation}
  |\bEta^+(x)|=o(|x|^{-1}).
\end{equation}
The boundary conditions can then be written as:
\begin{equation}\label{eqn3.12.11}
  \bEta^0\restrictedto_{\pa\Omega_+}-\bEta^+\restrictedto_{\pa\Omega_+}= \bEta^{\In}_+\restrictedto_{\pa\Omega_+}\text{ and }
  \star\bEta^0\restrictedto_{\pa\Omega_+}-\star\bEta^+\restrictedto_{\pa\Omega_+}
  = \star\bEta^{\In}_+\restrictedto_{\pa\Omega_+}.
\end{equation}

We can also allow an ``incoming field,'' $\bEta^{\In}_-,$ generated by
sources located in $\Omega_-.$  This is a harmonic field defined in a
neighborhood of $\Omega_+\cup\overline{\Omega_0}.$ Near to
$\pa\Omega_{-}$ the field
\begin{equation}
  \bEta^{\tot -}=\bEta^-+\bEta^{\In}_-.
\end{equation}
The boundary conditions on the inner boundary are then
\begin{equation}\label{eqn3.14.11}
  \bEta^0\restrictedto_{\pa\Omega_-}-\bEta^-\restrictedto_{\pa\Omega_-}= \bEta^{\In}_-\restrictedto_{\pa\Omega_-}\text{ and }
  \star\bEta^0\restrictedto_{\pa\Omega_-}-\star\bEta^-\restrictedto_{\pa\Omega_-}
  = \star\bEta^{\In}_-\restrictedto_{\pa\Omega_-}.
\end{equation}

The domain $\Omega_0$ is always non-contractible, as $H_2(\Omega_0)\neq 0.$ As
oriented cycles
\begin{equation}
  \pa\Omega_0=\pa\Omega_+-\pa\Omega_-.
\end{equation}
The components of $\pa\Omega_0$ are therefore homologous, with each generating
$H_2(\Omega_0).$ The Mayer-Vietoris sequence gives a good way to understand the
topology of $\Omega_0.$ The non-trivial parts of this sequence read:
\begin{equation}
  \begin{split}
    0\longrightarrow H_2(\pa\Omega_+)\oplus H_2(\pa\Omega_-)&\longrightarrow H_2(\Omega_0)\oplus H_2(\Omega_+) \longrightarrow 0\\
      0\longrightarrow H_1(\pa\Omega_+)\oplus H_1(\pa\Omega_-)&\longrightarrow H_1(\Omega_0)\oplus[ H_1(\Omega_+)\oplus H_1(\Omega_-)] \longrightarrow 0.
  \end{split}
\end{equation}
given our assumptions on the geometry of $\Omega_0,$ the components $\pa\Omega_{0}$
are surfaces of genus $g,$ which are isotopic in $\bbR^3.$

There are curves $\{A^{\pm}_1,\dots, A^{\pm}_{g};B^{\pm}_1,\dots,
B^{\pm}_{g}\}\subset\pm\Omega_{\pm},$ generating $H_1(\pa\Omega_{\pm}).$ Associated with these curves there
are surfaces (actually 2--cycles),
\begin{equation}
  \begin{split}
    \{S_{A^{\pm}_1},\dots, S_{A^{\pm}_{g}};S_{B^{\pm}_1},\dots,S_{B^{\pm}_{g}}\},&
    \text{ with }
    \pa S_{A^{\pm}_j}= A^{\pm}_j\text{ and }  \pa S_{B^{\pm}_j}= B^{\pm}_j;\\
  S_{A^+_j}\subset \overline{\Omega_0}\cup\Omega_-,\, S_{B^+_j}\subset \Omega_+,\, \text{ and }&
  S_{A^-_j}\subset \Omega_-,\, S_{B^-_j}\subset
  \overline{\Omega_0}\cup\Omega_-;\, \text{ for }j=1,\dots,g.
  \end{split}
\end{equation}
We can arrange to have $S_{A_j^-}\subset S_{A^+},$ and $S_{B_j^+}\subset
S_{B_j^-},$ so that $\{A_j^+:\:j=1,\dots,g\}$ generates $H_1(\Omega_+);$
$\{B_j^-:\:j=1,\dots,g\}$ generates $H_1(\Omega_-);$ and
$\{A_j^+\sim A_j^-;B_j^+\sim B_j^-:\: j=1,\dots,g\}$ 
generates $H_1(\Omega_0).$ The relative homology group
$H^2(\Omega_0,\pa\Omega_0)$ is generated by the 2-cycles $\{S_{A_j^+}-
S_{A_j^-},S_{B_j^+}- S_{B_j^-}:\:j=1,\dots,g\}.$  Figure~\ref{fig1} shows an
example of 2 nested tori, with the $A$- and $B$-cycles, along with the spanning
surfaces.
\begin{figure}[h]
  \centering
  \subfigure[Bottom view]
  {\includegraphics[width=.45\linewidth]{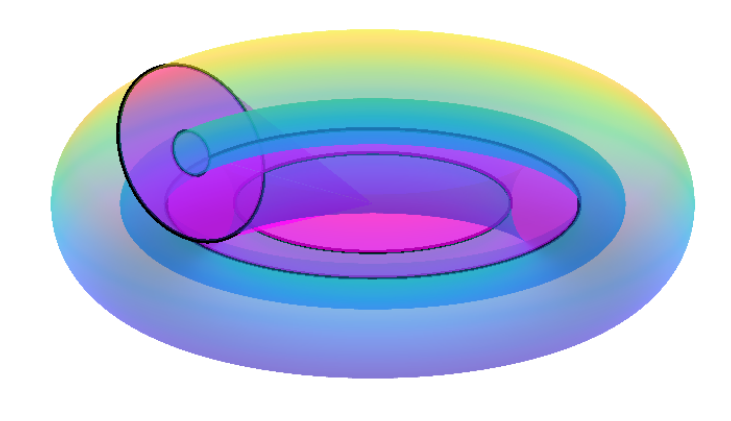}} \quad
  \subfigure[Top view]
  {\includegraphics[width=.45\linewidth]{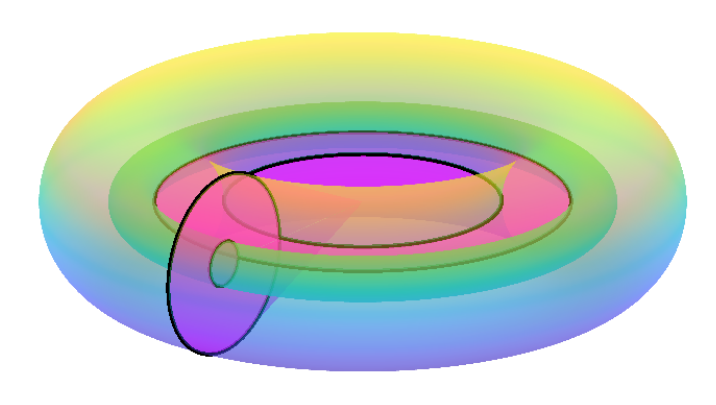}} 
\caption{Two nested tori, showing the $A$-cycles and $B$-cycles (in black), as
  well as the spanning surfaces (in magenta).}
\label{fig1}
\end{figure}

As in earlier cases, topological information must be specified for the
scattering problem to have a unique solution.  Since
$[\star\bEta^{\tot \pm}]\restrictedto_{\pa\Omega_{\pm}}
=[\star\bEta^0]\restrictedto_{\pa\Omega_{\pm}}$ are closed 1-forms, their
integrals over cycles in $\pa\Omega_{\pm}$ depend only on the homology class of
the cycle.  We assume that $\bEta^{\In}_+$ is defined in a neighborhood of
$\overline{\Omega_0}\cup\Omega_-,$ and  $\bEta^{\In}_-$ is defined in a neighborhood of
$\overline{\Omega_0}\cup\Omega_+.$ It then follows from Stokes theorem that
\begin{equation}\label{eqn1.15.1}
  \begin{split}
   &b_j^+=\int\limits_{B^+_j}\star\bEta^0= \int\limits_{B^+_j}\star\bEta^{\tot +}=
   \int\limits_{{B^+_j}}\star\bEta^{\In}_+,\, j=1,\dots,g\\
    &a_j^-=\int\limits_{A^-_j}\star\bEta^0= \int\limits_{A^-_j}\star\bEta^{\tot -}=
   \int\limits_{{A^-_j}}\star\bEta_{-}^{\In},\, k=1,\dots,g,
   \end{split}
\end{equation}
are determined by the incoming fields. On the other hand the periods
 \begin{equation}\label{eqn1.16.1}
   \begin{split}
   a^+_j&=\int\limits_{A^+_j}\star\bEta^{\tot +} =\int\limits_{A^+_j}\star\bEta^0,\text{ for }j=1,\dots,
   g,\\
   b^-_j&=\int\limits_{B^-_j}\star\bEta^{\tot -}=\int\limits_{B^-_j}\star\bEta^0,\text{ for }j=1,\dots, g,
 \end{split}
 \end{equation}
 are not determined a priori, and  constitute additional data
 that must be specified to get a unique solution. Our assumptions on
 the incoming fields, and Stokes theorem  imply that, for $j=1,\dots, g,$
 \begin{equation}\label{eqn1.16.11}
   \begin{split}
   \int\limits_{A^+_j}\star\bEta^{\In}_+ =
   &\int\limits_{B^-_j}\star\bEta^{\In}_{-}=0,\text{ showing }\\
  a_j^+ =\int\limits_{A^+_j}\star\bEta^{+} =\int\limits_{A^+_j}\star\bEta^0,&\quad
   b_j^-=\int\limits_{B^-_j}\star\bEta^{-}=\int\limits_{B^-_j}\star\bEta^0.
 \end{split}
 \end{equation}
 Stokes theorem and the ``curl''-equation $d^*\bEta^{0}  = \bj^{0}$  imply that  the differences
 \begin{equation}
       a_j^+-a_j^-=\int_{S_{A_j^+}-S_{A^-_j}}\star\bj_0\,\text{ and }\,b_j^+-b_j^-=\int_{S_{B_j^+}-S_{B^-_j}}\star\bj_0,
 \end{equation}
 give the fluxes of the current, $\bj_0,$ through the generators for the
 relative homology
 $H_2(\Omega_0,\pa\Omega_0).$ Note that $d\star\bj_0=0$ and $\star\bj_0\restrictedto_{\pa\Omega_0}=0.$
If the topological data, $\{a_j^+,b_j^-\}$ are not all 0, then this problem has a
nontrivial solution, even if $\bEta^{\In}_{\pm}=0.$

There several further
 topological observations that are crucial in our subsequent analysis: since
 $\bEta_0=-\lL^2 d\bj^0$ it follows that
 \begin{equation}\label{eqn1.17.2}
   \int_{\pa\Omega_{\pm}}\bEta_0= -\lL^2\int_{\pa\Omega_{\pm}}d\bj^0=0.
 \end{equation}
 Since
 $\bEta^0\restrictedto_{\pa\Omega_-}=[\bEta^-+\bEta^{\In}_-]_{\pa\Omega_-},$
 this implies
 \begin{equation}
   \int_{\pa\Omega_-}[\bEta^-+\bEta^{\In}_-]=0.
 \end{equation}
 As $d\bEta^-=0$ in $\Omega_-,$ we see that
 \begin{equation}\label{eqn3.24.12}
    \int_{\pa\Omega_-}\bEta^{\In}_-=0
 \end{equation}
 is a necessary condition on the incoming field $\bEta^{\In}_-$ for
 solvability. There is no such restriction on $\bEta^{\In}_+.$
 
 To summarize, the London boundary value problem for the fields
 $[\bEta^-, (\bEta^0, \bj^0), \bEta^+]$ in thin shells 
 for given incoming fields $\bEta^{\In}_{\pm}$ and topological constants
 $a_{j}^{+}, b_{j}^{-}$, $j=1,2,\ldots g$, is given by
 \begin{equation}
 \label{eq:scat_thin_shell}
 \begin{aligned}
 d^* \bEta^{0} &= \bj^0 \quad  \text{ in } \Omega_{0}, \\
  d \bj^{0} &= -\frac{1}{\lL^2}\bEta^0 \quad  \text{ in } \Omega_{0}, \\
  d^{*} \bEta^{\pm} = d \bEta^{\pm} &= 0 \quad  \text{ in } \Omega^{\pm}, \\
  |\bEta^{+}| &= o(|x|^{-1}) \quad \text{ as } |x| \to \infty, \\
  \bEta^{0} \restrictedto_{\pa \Omega^{\pm}} - \bEta^{\pm} \restrictedto_{\pa \Omega^{\pm}} &=  \bEta^{\In}_{\pm}\restrictedto_{\pa \Omega^{\pm}},\\  
    \star\bEta^{0} \restrictedto_{\pa \Omega^{\pm}} - \star\bEta^{\pm} \restrictedto_{\pa \Omega^{\pm}} &=  \star\bEta^{\In}_{\pm}\restrictedto_{\pa \Omega^{\pm}},\\  
   \int_{A_{j}^{+}} \star \bEta^{+} =    \int_{A_{j}^{+}} \star \bEta^{0} &= a_{j}^{+}, \quad j=1,2,\ldots g, \\
      \int_{B_{j}^{+}} \star \bEta^{+} =    \int_{B_{j}^{+}} \star \bEta^{0} &= b_{j}^{-}, \quad j=1,2,\ldots g.
 \end{aligned}
 \end{equation}
 The proof of Theorem 1 from~\cite{EpRa1} is easily adapted to show
 that, if it exists, the solution,
 $[\bEta^+,(\bEta^0,\bj^0),\bEta^-],$ to the scattering problem
 described above is uniquely determined by the data
 $\{\bEta^{\In}_{\pm};a^+_1,\dots,a^+_{g};b^-_1,\dots,b^-_{g}\}.$ 

 \subsection{Existence of Solutions}\label{ss.exist}
In this section, we prove that the solution to the scattering problem given
in~\cref{eq:scat_thin_shell} exists for any $\lL>0$. In~\Cref{sec:connected},
we established the existence of solutions to the scattering problem for London equations, 
by eliminating the current and studying a scattering problem involving only the magnetic fields.
An similar procedure can be applied for proving existence of solutions in the thin-shell 
case with a few minor modifications.

For our study of the solutions to the scattering problem and the limit as
$\lL\to 0^+,$ we let
$[\bEta^+_{\lL},(\bEta^0_{\lL},\bj^{0}_{\lL}),\bEta^-_{\lL}]$ denote the
solution, for the given value of $\lL,$ with fixed periods, and incoming fields,
$\{\bEta^{\In}_{\pm};a^+_1,\dots,a^+_{g};b^-_1,\dots,b^-_{g}\}.$ 
The interior magnetic field
as before satisfies the second order equation
 \begin{equation}\label{eqn2.1.9}
   \left[dd^*+\frac{1}{\lL^2}\right]\bEta^0_{\lL}=0,
 \end{equation}
which eliminates the current in the system of equations in~\cref{eq:scat_thin_shell}.

One of the key differences from the analysis of thin shells as
compared to the earlier analysis is that the ``exterior'' region has
two connected components $\Omega^{\pm}$.  In the limit, as before we
expect the interior magnetic field to converge to zero, and the
current in the superconducting region to converge to a current
sheet on the boundary. Thus, in the limit, while the normal components of the field
remain continuous, the tangential components of the total magnetic
field will no longer be continuous. However, the exterior region
having two connected components requires a  modification in the
construction of the limiting magnetic fields in $\Omega^{\pm}$.
 
 The limiting magnetic field in $\Omega^{+}$, denoted
 by $\bEta^{\out +}$ should be the unique outgoing 2-form satisfying
 \begin{equation}\label{eqn5.2.100}
d \bEta^{\out +} = 0 , \text{ and } d^{*} \bEta^{\out +} = 0 ,
 \end{equation}
 in $\Omega^{+}$, satisfying the boundary conditions
 \begin{equation}
 \bEta^{\out +}\restrictedto_{\pa \Omega^{+}} = -\bEta^{\In}_{+}\restrictedto_{\pa \Omega^{+}} \, ,
 \end{equation}
 and the topological constraints
 \begin{equation}\label{eqn5.4.100}
 \int_{A_{j}^{+}} \bEta^{\out +} = a_{j}^{+} \text{ for }j=1,\dots, g.
 \end{equation}
 Similarly, the limiting magnetic field in $\Omega^{-}$, denoted by
 $\bEta^{\out -}$ is the unique 2-form satisfying
\begin{equation}\label{eqn5.2.101}
d \bEta^{\out -} = 0 , \text{ and } d^{*} \bEta^{\out -} = 0 ,
 \end{equation}
 in $\Omega^{-}$, along with the boundary conditions
 \begin{equation}
 \bEta^{\out -}\restrictedto_{\pa \Omega^-} = -\bEta^{\In}\restrictedto_{\pa \Omega^-} \, ,
 \end{equation}
 and the topological constraints
 \begin{equation}\label{eqn5.4.101}
 \int_{B_{j}^{-}} \bEta^{\out -} = b_{j}^{-} \text{ for }j=1,\dots, g.
 \end{equation}
 The condition~\eqref{eqn3.24.12} is necessary for the existence of
 $\bEta^{\out -}.$ The fields $\bEta^{\out \pm}$ in $\Omega^{\pm}$ are
 in $H^{s}(\Omega^{\pm})$ for any $s \in \mathbb{R}$ for smooth
 incoming data $\bEta^{\In}_{\pm} \restrictedto_{\pa
   \Omega^{\pm}}$.

The ansatz for the scattered magnetic fields is: 
  \begin{equation}\label{top_ansatz_thinshell}
   \begin{split}
     \bEta^{+}_{\lL}&=\bEta^{\out +}+\star d u^{+}_{\lL}\\
     \bEta^{-}_{\lL}&=\bEta^{\out -}+\star d u^{-}_{\lL}\\
     \bEta^{0}_{\lL}&=\bEta^{00}_{\lL}+\beta_{\lL}+\star d v_{\lL}    
   \end{split}
 \end{equation}
 where $d\beta_{\lL} = 0$ in $\Omega$ and $\star
 \beta_{\lL}\restrictedto_{\pa \Omega^{\pm}} = \star (\bEta^{\out
   \pm}+ \bEta^{\In \pm})\restrictedto_{\pa \Omega^{\pm}}$,
 $d\bEta^{00}_{\lL} = 0$ in $\Omega_{0}$ and
 $\star\bEta^{00}_{\lL}\restrictedto_{\pa \Omega_{0}} = 0.$ The
 functions $u_{\lL}^{\pm}$ and $v_{\lL}$ are harmonic and satisfy the
 transmission boundary conditions given
 in~\eqref{eq:thin-shell-lap-trans} below.
 
 \subsubsection{Construction of $\beta_{\lL}$}
Lemma~\ref{lem1} constructs the 2-form $\beta_{\lL},$ with rather precise
estimates, as $\lL\to 0^+.$ As the construction is done in a small neighborhood
of the boundary, the result applies \emph{mutatis mutandis} to the current
situation wherein $\pa\Omega_0$ has 2 components, producing a 2-form
$\beta_{\lL}$ that shares the same tangential data as $\bEta^{\out \pm} +
\bEta^{\In}_{\pm}$ on $\pa \Omega^{\pm}$. This 2-form is supported in a small
neighborhood of $\pa\Omega_0,$ contained in $\overline{\Omega_0}.$ Near each
boundary component, $\pa\Omega^{\pm},$ $\beta_{\lL}$ takes the form
\begin{equation}
\label{eq:beta-def-thin-shell}
  \beta_{\lL}=e^{\frac{r_{\pm}}{\lL}}\psi(r_{\pm})dr_{\pm}\wedge\gamma_{\pm}+E_{\lL}(r_{\pm})d\gamma_{\pm},
\end{equation}
where
\begin{equation}
\label{eq:beta-def-thin-shell2}
  \begin{split}
  &r_{\pm}(x)=-\dist(x,\pa\Omega_{\pm}),\quad
  E_{\lL}(r)=\int_{-\infty}^re^{\frac{s}{\lL}}\psi(s)ds\text{
    and,}\\
  &\gamma_{\pm}=\pi_{\pm}^*(\gamma_{\pm}^0)\text{ where
  }\gamma_{\pm}^0=-\star_{2} (\star (\bEta^{\out\pm} + \bEta^{\In}_{\pm})\restrictedto_{\pa\Omega^{\pm}}).
  \end{split}
\end{equation}
Here $\pi_{\pm}$ are nearest point maps to $\pa\Omega_{\pm}.$  For a
$0<\delta\ll\dist(\pa\Omega_+,\pa\Omega_-),$ we choose a
function $\psi\in\cC^{\infty}_c(\bbR)$ with $\psi(t)=1$ for
$|t|<\delta/2,$ and $\psi(t)=0$ for $|t|>\delta.$ If necessary we further reduce
$\delta$ so that $\pi_{\pm}$ are well defined in the sets
$\{x:\:\dist(x,\pa\Omega_{\pm})\leq\delta\}.$
These 2-forms satisfy the estimates
\begin{equation}
  \|\beta_{\lL}\|_{L^2}\leq C\sqrt{\lL},\quad  \|\beta_{\lL}\|_{L^1}\leq C\lL,\quad
  \quad\left\|\left[dd^*+\frac{1}{\lL^2}\right]\beta_{\lL}\right\|_{L^2}\leq\frac{C}{\sqrt{\lL}}.
\end{equation}
For any $s$ there is a constant $C_s$ so that
\begin{equation}\label{eqn5.32.107}
  \|\beta_{\lL}\restrictedto_{\pa\Omega_0}\|_{H^s(\pa\Omega_0)}=
 E_{\lL}(0) \|d\gamma^0_{\pm}\restrictedto_{\pa\Omega_0}\|_{H^s(\pa\Omega_0)}\leq C_s\lL.
\end{equation}

\subsubsection{The harmonic functions $u^{\pm}_{\lL}$ and $v_{\lL}$}
Recall that $u^{\pm}_{\lL}$ and $v_{\lL}$ are harmonic functions, which, given
$\bEta^{00}_{\lL}$ and $\beta_{\lL},$ enforce the continuity of the normal
component of the magnetic fields. In particular, they satisfy the following
Laplace transmission problem:
\begin{equation}\label{eq:thin-shell-lap-trans}
 \begin{aligned}
 \Delta u^{\pm}_{\lL} &= 0, \text{ in } \Omega^{\pm}, \\
 \Delta v_{\lL} &= 0, \text{ in } \Omega_0, \\
 u^{\pm}_{\lL} &= v_{\lL}, \text{ on } \pa \Omega^{\pm}, \\
[ \partial_{n} u_{\lL}^{\pm} - \partial_{n} v_{\lL}]_{\pa\Omega_{\pm}} &= \frac{\bEta^{00}_{\lL}\restrictedto_{\pa\Omega^{\pm}}+
   \beta_{\lL}\restrictedto_{\pa\Omega^{\pm}}}{dS_{\pa\Omega^{\pm}}}, \\
   |u^{+}_{\lL}(x)| =O(|x|^{-1})& \text{ as } |x| \to \infty \, ,
 \end{aligned}
 \end{equation}
 where the normals along $\pa \Omega_{0}$ point into $\Omega_{0}$ on $\pa \Omega^{-}$, and into $\Omega^{+}$ 
 on $\pa \Omega^{+}$.

The system of equations in~\eqref{eq:thin-shell-lap-trans} has a unique solution that can be expressed in terms  of the single layer potential given by,
 \begin{equation}
 \begin{aligned}
 u^{\pm}_{\lL} &= -\left(\cS_{\pa \Omega^{+}}\left[\frac{\bEta^{00}_{\lL}\restrictedto_{\pa\Omega^{+}} +
   \beta_{\lL}\restrictedto_{\pa\Omega^{+}}}{dS_{\pa\Omega^{+}}} \right] + 
   \cS_{\pa \Omega^{-}}\left[\frac{\bEta^{00}_{\lL}\restrictedto_{\pa\Omega^{-}} +
   \beta_{\lL}\restrictedto_{\pa\Omega^{-}}}{dS_{\pa\Omega^{-}}} \right] \right) \quad \text{ in } \Omega^{\pm} \\
   v_{\lL} &= -\left(\cS_{\pa \Omega^{+}}\left[\frac{\bEta^{00}_{\lL}\restrictedto_{\pa\Omega^{+}} +
   \beta_{\lL}\restrictedto_{\pa\Omega^{+}}}{dS_{\pa\Omega^{+}}} \right] + 
   \cS_{\pa \Omega^{-}}\left[\frac{\bEta^{00}_{\lL}\restrictedto_{\pa\Omega^{-}} +
   \beta_{\lL}\restrictedto_{\pa\Omega^{-}}}{dS_{\pa\Omega^{-}}} \right] \right) \quad \text{ in } \Omega .
   \end{aligned}
 \end{equation}
This follows from the properties of single layer potentials given
in~\Cref{subsec:lp}, in particular the jump conditions on the normal derivative.

We can now eliminate $\star d u^{\pm}_{\lL}$ and $\star d v_{\lL}$ from the ansatz. 
In the thin shell setup, let
\begin{equation}
\begin{aligned}
  B_{0}[h_+,h_-](x)&= \cS_{\pa \Omega^{+}}[h_{+}](x) + \cS_{\pa \Omega^{-}}[h_{-}](x) \quad x \in \Omega_{0},\\
  B_{\pm}[h_+,h_-](x)&= \cS_{\pa \Omega^{+}}[h_{+}](x) + \cS_{\pa \Omega^{-}}[h_{-}](x) \quad x \in \Omega_{\pm}.\\
  \end{aligned}
\end{equation}
Let $B$ be the collective operator which is $B_{0}$ in $\Omega_{0}$, and $B_{\pm}$ in $\Omega^{\pm}$.
Let $\cB$ denote the operator, which acts on closed 2-forms $\theta$ 
defined in $\overline{\Omega_0}$ given by 
\begin{equation}\label{eqn5.36.110}
  \cB(\theta)=-\star dB_0(i^*\theta)\text{
    where }i^*\theta=(i_{n_{+}}\star\theta\restrictedto_{\pa\Omega_{+}},
  -i_{n_-}\star\theta\restrictedto_{\pa\Omega_{-}}).
\end{equation}
 Recall that $n_+$ points into $\Omega_+,$ and $n_-$ points into
 $\Omega_0,$ which explains the signs in~\eqref{eqn5.36.110}.  As before, we have 
 \begin{equation}
 \bEta^{0}_{\lL} =  \bEta^{00}_{\lL} + \beta_{\lL} + \cB(\bEta^{00}_{\lL} + \beta_{\lL}) \, .
 \end{equation}
 
In the following lemma, we present some useful properties of the operator $\cB$.
\begin{proposition}\label{prop1.9}
  The operator $\cB$ acting on closed 2-forms defined in $\Omega_0$ is a
  bounded, self adjoint operator. The spectrum of $\cB$ is contained in
  $[-1,0],$ with accumulation only at $-\frac 12.$ The eigenspace
    at $-1$ is one dimensional; let $\theta_0,$ denote a non-zero element of
    this subspace. The spectrum of $\cB_{\restrictedto \theta_0^{\bot}}$ is
    a compact subset of $(-1,0].$
\end{proposition}
\begin{proof}
  We let $\theta_1,\theta_2$ be closed 2-forms defined in $\Omega_0$ and set
  \begin{equation}
     \tau_{i}^{\pm}=i_{n_{\pm}}\star\theta_{i}\restrictedto_{\pa\Omega_{\pm}},\,i=1,2.
  \end{equation}
  Stokes theorem implies that
  \begin{equation}
    \langle \cB\theta_1,\theta_2\rangle=\int_{\pa\Omega_+}(\cS_{\pa \Omega^{-}}\tau_1^--\cS_{\pa \Omega^{+}}\tau_1^+)
    \tau_2^+dA_{\pa\Omega_+}-
    \int_{\pa\Omega_-}(\cS_{\pa \Omega^{-}}\tau_1^--\cS_{\pa \Omega^{+}}\tau_1^+)\tau_2^-dA_{\pa\Omega_-}
  \end{equation}
  The minus sign in front of  the integral over $\pa\Omega^-$ accounts for the fact that
  $n_-$ is the \emph{inward} pointing normal along this boundary.
  Writing out the double integrals, using Fubini's theorem, and the fact that $g_0(x,y)=g_0(y,x)$ we see that
  \begin{multline}
    \langle \cB\theta_1,\theta_2\rangle=
    \int_{\pa\Omega_+}\int_{\pa\Omega_-}g_0(x,y)[\tau_1^-(y)\tau_2^+(x)+\tau_2^-(y)\tau_1^+(x)]dA_{\pa\Omega_-}(y)dA_{\pa\Omega_+}(x)-\\\int_{\pa\Omega_+}\int_{\pa\Omega_+}g_0(x,y)\tau_1^+(y)\tau_2^+(x)dA_{\pa\Omega_+}(y)dA_{\pa\Omega_+}(x)-\\\int_{\pa\Omega_-}\int_{\pa\Omega_-}g_0(x,y)\tau_1^-(y)\tau_2^-(x)dA_{\pa\Omega_-}(y)dA_{\pa\Omega_-}(x)=\langle \theta_1,\cB\theta_2\rangle,
  \end{multline}
  proving that $\cB$ is self adjoint.

Using properties of the single layer potential and the definition of $\cB$, we see that
  \begin{multline}\label{eqn2.29.4}
    \cB(\theta)\restrictedto_{\pa\Omega_0}=\\
    -\left[\left(\frac{i_{n_+}\theta}{2}+S'_+(i_{n_+}\theta,-i_{n_-}\theta)\right)
      dS_{\pa\Omega_{0+}},
      \left(\frac{-i_{n_-}\theta}{2}+S'_-(i_{n_+}\theta,-i_{n_-}\theta)\right)
      dS_{\pa\Omega_{0-}}\right]\\
    :=-\left[\frac{\Id}{2}+S'\right](i_{n_+}\theta,-i_{n_-}\theta),
  \end{multline}
  where $dS_{\pa\Omega_{0\pm}}=\pm i_{n_{\pm}}dV$ are the area forms of
  $\pa\Omega_0$ oriented as the boundary of $\Omega_0,$ where the self-integrals
are understood in a principal value sense.   The proof of Theorem 3.3.1
  in~\cite{Nedelec} applies to show that $S:L^2(\pa\Omega_0)\to
  L^2(\pa\Omega_0)$ is a positive self adjoint operator, and therefore we can
  use $\langle S f,g\rangle_{\pa\Omega_0}$ to define an inner product.  As
  before $S'_{\pa\Omega_0}$ is self adjoint with respect to this inner
  product. Let $\{(\mu_j,g_j):\: j=1,\dots\}$ denote a basis, orthonormal, with
  respect to this inner product, of eigenpairs for $S',$ and set
  $u_j=B(g_{j+},g_{j-}).$ If we let
  \begin{equation}
    U=\int_{\Omega_0}|\nabla u_j|^2,\quad V=\int_{\Omega_+\cup\Omega_-}|\nabla u_j|^2,
  \end{equation}
  then a classical argument, see~\cite{ColtonKress,NedelecPlanchard1973}, using the fact
  that $\Delta |u_j|^2=2|\nabla u_j|^2,$ shows that
  \begin{equation}
    \mu_j=\frac{U-V}{2(U+V)},
  \end{equation}
  which implies that $-\frac 12\leq \mu_j\leq \frac 12.$ In this case, both $\pm
  \frac 12$ are in the spectrum of $S',$ and of multiplicity 1. We let
  $\mu_c=-\frac 12$ and $\{\mu_j:\: j\in\bbN_0\}$ denote the spectrum in
  $(-\frac 12,\frac 12],$ with $\mu_0=\frac 12.$

 Let $g_0$ be an non-trivial element of the $\frac 12$-eigenspace of $S_{\pa\Omega_0}';$   define
  $\varphi_0=\cS_{\pa\Omega_0} g_0$ in $\Omega_0$ and let $\varphi_{\pm}$ denote
 the harmonic functions defined by this single layer in $\Omega_{\pm}.$ The fact that
 $S_{\pa\Omega_0}'g_0=\frac 12 g_0,$ and the standard jump formul{\ae} imply
 that
 \begin{equation}
   \pa_{n}\varphi_{\pm}\restrictedto_{\pa\Omega_{\pm}}=0.
 \end{equation}
This shows that $\varphi_+\equiv 0,$ and $\varphi_-$ is a constant, which we can
normalize to be 1. As the single layer is continuous across $\pa\Omega_0,$ 
uniqueness for the Dirichlet problem proves that the eigenspace at
$\frac 12$ is 1-dimensional, and spanned by $g_0.$ If we let $\theta_0=\star d
\varphi_0,$ then we see that $\cB\theta_0=-\theta_0.$

The eigenvalue at $\mu_c=-\frac 12$ comes from the unique harmonic function that equals
1 in $\Omega^-\cup\Omega_0,$ and goes to zero in $\Omega_+.$ This eigenspace of
$S'$ is therefore spanned by a function $g_c$ such that $\cS_{\pa\Omega_0}g_c\equiv 1$ in
$\Omega_0;$ note that $g_c\restrictedto_{\pa\Omega_-}=0.$ Hence
  \begin{equation}\label{eqn5.43.105}
    d\cS_{\pa\Omega_0}[g_c](x)=0\text{ for }x\in\Omega_0.
  \end{equation}

  For $j\geq 0$ we let
  \begin{equation}\label{eqn2.32.4}
    \theta_j=\star d B_0(g^j_+,g^j_-),
  \end{equation}
  then
  \begin{equation}
    \theta_j\restrictedto_{\pa\Omega_0}=\left(\frac 12+\mu_j\right)g_jdS_{\pa\Omega_0},
  \end{equation}
  which implies that
  \begin{equation}\label{eqn2.34.4}
    \cB\theta_j=-\left(\frac 12+\mu_j\right)\theta_j.
  \end{equation}
 Arguing as in the proof of~\Cref{thm1}, we show that if $\theta$ is orthogonal
 to the $\Span\{\theta_j:\: j=0,1,2,\dots\},$ then $i^*\theta=a
 g_cdS_{\pa\Omega_0},$ for a constant $a.$ Thus~\eqref{eqn5.43.105} implies that
 $\cB(\theta)=0,$ and therefore the spectrum of $\cB$ is contained in $[-1,0],$
 with $-\frac 12$ the only point of accumulation. As $\theta_0$ is an
 eigenvector of $\cB,$ restricting to 2-forms orthogonal to $\theta_0,$ we
 eliminate the eigenvalue at $-1.$ The spectrum of $\cB_{\restrictedto
   \theta_0^{\bot}}$ is a compact subset of $(-1,0].$
\end{proof}

\begin{remark}
  As before the null-space of $\cB$ is infinite dimensional, consisting of
  closed 2-forms with $i^*\theta=a g_cdS_{\pa\Omega_0},$  for  $a\in\bbC.$
\end{remark}

\subsubsection{Analysis of $\bEta^{00}_{\lL}$}
Finally, we turn our attention to the analysis of the closed 2-form $\bEta^{00}_{\lL}$.  
Combining the
constructions of $u^{\pm}_{\lL}$, $v_{\lL}$, and $\beta_{\lL}$, we note that $\bEta^{00}_{\lL}$ also satisfies
 \begin{equation}\label{eqn2.46.055}
\tcA_{\lL}\bEta^{00}_{\lL} :=  \left[dd^*+\frac{\Id+\cB}{\lL^2}\right]\bEta^{00}_{\lL}=
    -\left[dd^*+\frac{1}{\lL^2}\right]\beta_{\lL}-\frac{1}{\lL^2}\cB\beta_{\lL}\,,\\
 \end{equation}
As before, the above relation does not imply that $\bEta^{00}_{\lL} = -\beta_{\lL}$ since
$\bEta^{00}_{\lL}$ must have $0$ tangential data. Thus the operator $\tcA_{\lL}$ must be understood in
an appropriate domain. The domain of the operator $\tcA_{\lL}=\left[dd^*+\frac{\Id+\cB}{\lL^2}\right]$
consists of closed 2-forms, $\bEta,$ in $\Omega_0,$ such that
\begin{equation}
  \bEta\in H^2(\Omega_0;\cZ^2),\text{ and } \star\bEta\restrictedto_{\pa\Omega_0}=0.
\end{equation}
This is a self adjoint operator. Indeed we have the following.
\begin{proposition}
  Let $\cA_{\lL}$ be the operator given in~\eqref{eqn2.46.055}, with domain consisting of
  $$\bEta\in H^2(\Omega_0;\cZ^2),\text{ and } \star\bEta\restrictedto_{\pa\Omega_0}=0,$$
  with
  \begin{equation}\label{eqn2.22.22}
  \int_{\Omega_0}\bEta\wedge\star \theta_0=0.
  \end{equation}
  This is an invertible self adjoint operator; there are positive constants $C_0, C_1, C_2$ so that
  \begin{equation}
   \|\cA_{\lL}^{-1}\|_{L^2\to L^2}\leq C_0 \lL^2,\quad
   \|\cA_{\lL}^{-1}\|_{L^2\to H^1}\leq C_1\lL\text{ and
   }\|\cA_{\lL}^{-1}\|_{L^2\to H^2}\leq C_2.
  \end{equation}
\end{proposition}

\begin{proof}
   The facts $d d^{*} \theta_{0} = 0$,  $(I + \cB)\theta_{0} = 0$, and $ \star\theta_0\restrictedto_{\pa\Omega_0}=0,$
   imply that $\theta_{0}$ is in the domain of $\tcA_{\lL}$, and, in
   particular, is a null vector of $\tcA_{\lL}$.   Hence
   $\tcA_{\lL}$ restricted to $\theta_{0}^{\bot},$ is a self adjoint operator.

Using Proposition~\ref{prop1.9} we see that there is a constant $c>0,$ so that if
$\bEta\in\Dom(\cA_{\lL}),$ then
\begin{equation}\label{eqn2.33.3}
  \langle(I+\cB)\bEta,\bEta\rangle\geq c\|\bEta\|^2_{L^2},
\end{equation}
and therefore
\begin{equation}\label{eqn2.34.3}
  \langle\cA_{\lL}\bEta,\bEta\rangle=\langle d^*\bEta,d^*\bEta\rangle+
  \frac{\langle(I+\cB)\bEta,\bEta\rangle}{\lL^2}\geq\frac{c\|\bEta\|^2_{L^2}}{\lL^2}.
\end{equation}
Thus the operator $\cA_{\lL}^{-1}$ is a well defined operator on closed 2-forms
orthogonal to $\theta_0.$ Moreover we have the norm estimate
\begin{equation}
  \|\cA_{\lL}^{-1}\|_{L^2\to L^2}\leq \frac{\lL^2}{c}.
\end{equation}
It also follows easily from~\eqref{eqn2.34.3}, and the equation that there is a constant, $C,$ so
that
\begin{equation}
  \|\cA_{\lL}^{-1}\|_{L^2\to H^1}\leq C\lL\text{ and }\|\cA_{\lL}^{-1}\|_{L^2\to H^2}\leq C.
\end{equation}
\end{proof}
\begin{remark}
Observe that if $\bEta$ is a closed 2-form
   satisfying~\eqref{eqn1.17.2}, then Stokes theorem shows that
\begin{equation}\label{eqn2.22.2}
  \int_{\Omega_0}\bEta\wedge\star \theta_0=\int_{\pa\Omega_-}\bEta=0.
\end{equation}
\end{remark}

The basic existence theorem follows easily from this proposition.
\begin{theorem}
  For an arbitrary incoming harmonic fields $\bEta^{\In}_{\pm},$ satisfying~\eqref{eqn3.24.12}, and periods
  $\{a^+_1,\dots,a^+_{g};b^-_1,\dots,b^-_{g}\}$ there exists a unique solution
  $(\bEta^{+}_{\lL},(\bEta^{0}_{\lL},\bj^0_{\lL}),\bEta^-_{\lL}),$ to the transmission problem given
  by the system of equations~\eqref{eq:scat_thin_shell}. 
\end{theorem}
\begin{remark}
  Recall $\bEta^{\In}_+$ is a harmonic field generated by sources in $\Omega_+,$
  whereas $\bEta^{\In}_-$ is a harmonic field generated by sources in $\Omega_-.$
\end{remark}
\begin{proof}
  From the data we construct the harmonic forms $\bEta^{\out \pm},$ in
  $\Omega_{\pm},$ respectively that satisfy~\eqref{eqn5.2.100}--\eqref{eqn5.4.101}, and thence
  the closed form $\beta_{\lL}$ given
  by~\cref{eq:beta-def-thin-shell,eq:beta-def-thin-shell2}. The closed 2-form
  \begin{equation}
    \psi=\left[dd^*+\frac{1+\cB}{\lL^2}\right]\beta_{\lL}
  \end{equation}
  is orthogonal to $\theta_0:$ as $\langle dd^*\beta_{\lL},\theta_0\rangle$
  obviously vanishes we have that
  \begin{equation}
    \begin{split}
      \lL^2\langle\psi,\theta_0\rangle
      &=\langle(\Id+ \cB)\beta_{\lL},\theta_0\rangle\\
      &=\langle \beta_{\lL},(\Id+\cB)\theta_0\rangle= 0.
    \end{split}
  \end{equation}
  
Thus we can define
  \begin{equation}
    \bEta^{00}_{\lL}=-\cA_{\lL}^{-1}\left[dd^*+\frac{1+\cB}{\lL^2}\right]\beta_{\lL},
  \end{equation}
  and set
  \begin{equation}
    u^{\pm}_{\lL}=B_{\pm}(-i_{n_+}\star(\beta_{\lL}+\bEta^{00}_{\lL}),
    i_{n_-}\star(\beta_{\lL}+\bEta^{00}_{\lL}))\text{ and }\star
    dv_{\lL}=\cB(\beta_{\lL}+\bEta^{00}_{\lL}).
  \end{equation}
  It is not difficult to check that the 2-forms:
  \begin{equation}
    \begin{split}
      \bEta^+_{\lL}=\bEta^{\out +}+\star du^+_{\lL},&\quad 
    \bEta^-_{\lL}=\bEta^{\out -}+\star d u^-_{\lL},\\
    \bEta^0_{\lL}=\bEta^{00}_{\lL}&+\beta_{\lL}+\star dv_{\lL},
     \end{split}
  \end{equation}
  are closed and with $\bj^{0}_{\lL} = d^{*} \bEta^{0}_{\lL}$, satisfy the system of
  equations in~\cref{eq:scat_thin_shell}.
\end{proof}
\noindent
The components of the solution satisfy various estimates, as $\lL\to 0^+,$ that are considered in the
following section.

\subsection{Asymptotics as $\lL\to 0^+$}\label{sec3.3.12}

Using the estimates above on the norms of the operator $\cA_{\lL}^{-1},$ the
analysis in Section~\ref{sec:a-analysis} applies to show that
\begin{equation}
  \|\bEta^{00}_{\lL}\|_{L^2(\Omega_0)}\leq C\lL,\quad \|\bEta^{00}_{\lL}\|_{H^1(\Omega_0)}\leq C;
\end{equation}
interpolation gives
\begin{equation}
  \|\bEta^{00}_{\lL}\|_{H^{\frac 12}(\Omega_0)}\leq C\sqrt{\lL},
\end{equation}
 and therefore Lemma~\ref{tr_lem} gives
\begin{equation}
\|\bEta^{00}_{\lL}\|_{L^2(\pa\Omega_0)}\leq C\sqrt{\lL}.
\end{equation}
From these estimates we conclude that
\begin{equation}
  \|du^{+}_{\lL}\|^2_{L^2(\Omega_+)}+ \|dv^{0}_{\lL}\|^2_{L^2(\Omega_0)}+
  \|du^{-}_{\lL}\|^2_{L^2(\Omega_-)}\leq C\lL^2,
\end{equation}
and
\begin{equation}\label{eqn3.8.6}
   \|u^{+}_{\lL}\|_{H^{\frac 32}_{\loc}(\Omega_+)}\leq C\sqrt{\lL},\quad
   \|v_{\lL}\|_{H^{\frac 32}(\Omega_0)}\leq  C\sqrt{\lL},\text{ and }
  \|u^{-}_{\lL}\|_{H^{\frac 32}(\Omega_-)}\leq C\sqrt{\lL}.
\end{equation}
These estimates suffice to show that the magnetic fields in
$\Omega_{\pm}$ converges to $\bEta^{\In}_{\pm}+\bEta^{\out \pm},$ and
the magnetic field in $\Omega_0$ converges to 0, with the estimates
\begin{equation}
  \begin{split}
    &\|\bEta^{\pm}_{\lL}-\bEta^{\out \pm}\|_{L^2(\Omega_{\pm})}=\|\star du^{\pm}_{\lL}\|_{L^2(\Omega_{\pm})}
    \leq C\lL,\\
      &\|\bEta^0_{\lL}-\beta_{\lL}\|_{L^2(\Omega_0)}=\|\bEta^{00}_{\lL}+\star dv_{\lL}\|_{L^2(\Omega_0)}\leq C\lL.
  \end{split}
\end{equation}

The leading order part of the current, $\bj^0_{\lL},$ is given
$d^*\beta_{\lL}.$ This family of 1-forms converge, in the sense of
distributions, to an explicit current sheet on $\pa\Omega_0.$ For
$\Psi$ an arbitrary smooth 2-form defined in $\overline{\Omega_0}$
we have:
\begin{equation}\label{eqn3.76.13}
  \begin{split}
  \int_{\Omega_0}\bj^0_{\lL}\wedge\Psi&=
  \int_{\Omega_0}d^*\bEta^0_{\lL}\wedge\Psi\\ &=\int_{\pa\Omega_0}\star\bEta^0_{\lL}\wedge\star\Psi+
  \int_{\Omega_0}\bEta^0_{\lL}\wedge
  d^*\Psi\\ &=\int_{\pa\Omega_0+}\star_2\gamma_+\wedge\star\Psi-\int_{\pa\Omega_0-}\star_2\gamma_-\wedge\star\Psi+
  \int_{\Omega_0}\beta_{\lL}\wedge d^*\Psi+O(\sqrt{\lL}\|\Psi\|_{H^1}).
  \end{split}
\end{equation}
Note that the $\beta_{\lL}$-term can be estimated as either
$O(\lL\|d^*\Psi\|_{L^{\infty}}),$ or $O(\sqrt{\lL}\|d^*\Psi\|_{L^{2}}).$

As before in Section~\ref{ss2.3}, these error estimates are not
optimal. While we expect that each term can be improved by a factor of $\lL^{\frac 12},$ 
we now prove that these estimates can be improved by a factor of
$\lL^{\frac 12-\epsilon},$ for any $\epsilon>0.$
\begin{theorem}\label{thm4.110}
  For $0<\epsilon$ there is a constant $C_{\epsilon}$ so that
  \begin{equation}\label{eqn5.70.110}
    \|\bEta^{00}_{\lL}+\star
    dv^0_{\lL}\|_{L^2(\Omega_0)}<C_{\epsilon}\lL^{\frac 32-\epsilon}.
  \end{equation}
  \begin{equation}\label{eqn5.71.110}
    \|(\bEta^{00}_{\lL}+\star
    dv^0_{\lL}+\beta_{\lL})\restrictedto_{\pa\Omega_0}\|_{L^2(\pa\Omega_0)}<C_{\epsilon}\lL^{1-\epsilon}.
  \end{equation}
\end{theorem}
\begin{remark}
The proof of this theorem is quite similar to that of Theorem~\ref{thm2.107}, but without the simplifying
assumptions. 
\end{remark}
\begin{proof}[Proof of Theorem~\ref{eqn5.70.110}]
  From~\eqref{eqn5.32.107} it follows that
  $\|\beta_{\lL}\restrictedto_{\pa\Omega_0}\|_{L^2(\pa\Omega_0)}<C\lL,$ so we can restrict our attention to
  $\bEta^{00}_{\lL}+\star dv^0_{\lL}.$
  We first estimate this quantity  in $L^2(\Omega_0).$  Observe that
\begin{multline}\label{eqn5.66.103}
  \bEta^{00}_{\lL}+\star
  dv^0_{\lL}=-(\Id+\cB)\cA^{-1}_{\lL}\left[dd^*+\frac{1}{\lL^2}\right]\beta_{\lL}+\\
  \cB(\beta_{\lL})-(\Id+\cB)(\lL^2 dd^*+\Id+\cB)^{-1}\cB(\beta_{\lL}).
\end{multline}
The $L^2$-norm of the quantity on the first line is $O(\lL^{\frac 32}),$ which
is the expected optimal rate of decay. The $L^2$-norms of the individual terms on the
second line are $O(\lL),$ but note that the operator $\left[\Id-(\Id+\cB)(\lL^2
  dd^*+\Id+\cB)^{-1}\right]$ tends to 0, in the strong $L^2$-sense, as $\lL\to
0^+.$

As before, the value of $\cB(\beta_{\lL})$ depends only on
\begin{equation}
  \beta_{\lL}\restrictedto_{\pa\Omega_0}=E_{\lL}(0)d\gamma_{\pm}\restrictedto_{\pa\Omega_{\pm}}.
\end{equation}
Hence there is a smooth, closed 2-form, $\chi,$ which is independent of $\lL,$ so that
\begin{equation}
  \cB(\beta_{\lL})=E_{\lL}(0)\chi.
\end{equation}
As shown in Section~\ref{sec:beta}, we have the estimate $|E_{\lL}(0)|\leq C\lL.$ 
Thus for any $0\leq s,$ there is a constant $C_s,$ so that
\begin{equation}\label{eqn5.73.107}
  \|\cB(\beta_{\lL})\|_{H^s(\Omega_0)}\leq C_s\lL.
\end{equation}
The problem therefore is to estimate the rate at which
\begin{equation}\label{eqn5.70.103}
D_{\lL}=  \|\left[\Id-(\Id+\cB)(\lL^2 dd^*+\Id+\cB)^{-1}\right]\chi\|_{L^2(\Omega_0)}
\end{equation}
tends to zero, as $\lL\to 0^+.$
\begin{remark}
  In Section~\ref{ss2.3} we estimated the simpler quantity
\begin{equation}
  \|\left[\Id-(\lL^2 dd^*+\Id)^{-1}\right]\chi\|_{L^2(\Omega_0)}
\end{equation}
by using the eigenfunction expansion for the operator $dd^*,$ and the fact that
$\chi\in\Dom([dd^*]^{\alpha})$ for any $\alpha<\frac 14.$
\end{remark}
\begin{proposition}\label{prop8.107}
  For $0<\alpha<\frac 12$ there is constant $C_{\alpha}$ so that
  \begin{equation}
    D_{\lL}\leq C_{\alpha}\lL^{\alpha}.
  \end{equation}
\end{proposition}

To prove these estimates for $D_{\lL}$ we use the eigenfunction expansion for the
operator $dd^*,$ acting on the space of closed 2-forms
satisfying~\eqref{eqn2.22.2}, with inner product defined by
\begin{equation}
  \langle\theta,\psi\rangle_{\cB}=\langle(\Id+\cB)\theta,\psi\rangle.
\end{equation}
To set this up requires some effort.

Since $(\Id+\cB)\theta_0=0,$ we restrict to the subspace of $L^2(\Omega_0;\cZ^2)$  with
\begin{equation}
  \int_{\Omega_0}\bEta\wedge\star\theta_0=0,
\end{equation}
which we denote by $L^2_{\thop}(\Omega_0;\cZ^2).$  For $s\geq 0,$ we let
$$H^s_{\thop}(\Omega_0;\cZ^2)=H^s(\Omega_0;\cZ^2)\cap L^2_{\thop}(\Omega_0;\cZ^2).$$
It follows from Proposition~\ref{prop1.9} and~\eqref{eqn2.33.3} that the norm
defined by $ \langle\cdot,\cdot\rangle_{\cB}$ on $L^2_{\thop}(\Omega_0;\cZ^2)$
is equivalent to the standard $L^2$-norm.  The operator we employ is defined by the
quadratic form
\begin{equation}
  \cQ_{\cB}(\theta,\psi)=\int_{\Omega_0}d^*\theta\wedge\star d^*\overline{\psi},
\end{equation}
with $\Dom(\cQ_{\cB})=\{\bEta\in H^1_{\thop}(\Omega_0;\cZ^2):\:
\star\bEta_{\restrictedto_{\pa\Omega_0}}=0\}.$ The domain of the associated
self adjoint operator, $L_{\cB},$ consists of $\bEta\in\Dom(\cQ_{\cB})$ such that
there is a $\chi\in L^2(\Omega_0;\cZ^2)$ for which
\begin{equation}
  \cQ_{\cB}(\bEta,\psi)=\langle(\Id+\cB)\chi,\psi\rangle,\text{ for all }\psi\in\Dom(\cQ_{\cB}),
\end{equation}
which is easily seen to be $\{\bEta\in
H^2_{\thop}(\Omega_0;\cZ^2):\:\star\bEta_{\restrictedto_{\pa\Omega_0}}=0\}.$
The null-space of this operator would consist of harmonic 2-forms that satisfy
the boundary condition $\star\bEta\restrictedto_{\pa\Omega_0}=0.$ The Hodge
decomposition shows that the null-space is isomorphic to deRham cohomology group
$H^2_{\dR}(\Omega_0),$ see~\cite{Taylor1}. As shown above, this 1-dimensional
cohomology group is generated by the form $\theta_0$ defined in~\Cref{prop1.9},
and therefore $L_{\cB}$ is invertible on $L^2_{\thop}(\Omega;\cZ^2).$

The operator $L_{\cB}$ has a complete orthonormal family of eigenpairs
$\{(\psi_j,\nu_j):\: j=1,\dots\},$ which satisfy
\begin{equation}\label{eqn3.19.6}
  d\psi_j=0,\, dd^*\psi_j=\nu_j(\Id+\cB)\psi_j\text{ and }
  \langle(\Id+\cB)\psi_j,\psi_k\rangle=\delta_{jk}.
\end{equation}
A 2-form $\bEta\in L^2_{\thop}(\Omega_0;\cZ^2)$ has an expansion
\begin{equation}
  \bEta\sim\sum_{j=0}^{\infty}a_j\psi_j\text{ where
  }a_j=\langle(\Id+\cB)\bEta,\psi_j\rangle,
\end{equation}
and
\begin{equation}
  \langle(\Id+\cB)\bEta,\bEta\rangle=\sum_{j=0}^{\infty}|a_j|^2.
\end{equation}
The operator $(\Id+\cB)$ is positive and invertible on
$L^2_{\thop}(\Omega_0;\cZ^2);$  it is somewhat easier to work with the
operator $(\Id+\cB)^{-\frac 12}dd^* (\Id+\cB)^{-\frac 12},$ which is self
adjoint with respect to the usual inner product. To analyze this operator we
need to understand the mapping properties of $(\Id+\cB)^{-1}$ acting on
$H^s(\Omega_0;\cZ^2).$ As a first step we prove:

\begin{lemma}\label{prop2}
  For any $s\geq 0,$ the operator $\Id+\cB:H^s(\Omega_0;\cZ^2)\to H^s(\Omega_0;\cZ^2)$ is a
  Fredholm operator of index 0, with null-space spanned by $\theta_0.$
\end{lemma}
\begin{proof}
To prove this we construct the orthogonal projector onto the range of $\cB,$
which  is done using~\eqref{eqn2.29.4} and the eigenfunction expansion,
$\{(\theta_j,\mu_j)\},$ defined in~\eqref{eqn2.32.4}--\eqref{eqn2.34.4}. We note
that $(\Id/2+S')$ has a 1-dimensional null-space spanned by
$\tau_0=\theta_0\restrictedto_{\pa\Omega_0}.$ If we define the operator
\begin{equation}
  \tau_0\otimes \tau_0^*(\psi)=\langle\psi,\tau_0\rangle\tau_0,
\end{equation}
then, for any $s\geq 0,$ operator $(\Id/2+S'+\tau_0\otimes \tau_0^*):H^s(\pa\Omega_0) \to
H^s(\pa\Omega_0)$ is an invertible Fredholm operator of index zero.

Let
\begin{equation}
  \cW=\left(\frac{\Id}{2}+S'+\tau_0\otimes \tau_0^*\right)^{-1},
\end{equation}
this is a bounded operator from $H^s(\pa\Omega_0)$ to $H^s(\pa\Omega_0).$ We
define the projector, $\cP$ onto the range of $\cB$ by
\begin{equation}
  \cP(\theta)=\star d\cS\circ\cW\, i^*\theta,\text{ where
  }i^*\theta=(i_{n_+}\theta\restrictedto_{\pa\Omega_{+}},-i_{n_-}\theta\restrictedto_{\pa\Omega_{-}}).
\end{equation}
It follows from~\eqref{eqn2.29.4} that $\cP\theta_j=\theta_j$ for $j\geq 1.$
This shows that the range of $\cP$ contains that of $\cB,$ and that $\cP$
reduces to the identity on the range of $\cB.$ If $\cB\theta=0,$ then
$\theta\restrictedto_{\pa\Omega_0}=a\tau_0,$ for some $a,$ and therefore
$\cP\theta=0,$ as well. As $\cB$ is self adjoint, this shows that $\cP$ is a
projection onto the $\Im\cB,$ with $\Ker\cP=\Ker\cB=[\Im\cB]^{\bot},$ proving
that $\cP^*=\cP.$

  From its definition it is clear that $\cP:H^s(\Omega_0)\to H^s(\Omega_0)$ boundedly
  for any $s\geq 0.$ The operator $\cW=2\Id+K,$ where $K$ is pseudodifferential
  operator of order $-1.$ This shows that
  \begin{equation}
    \cB=-\frac{\cP+\star d\cS\circ Ki^*}{2}.
  \end{equation}
  The operator $\cK=\frac{1}{2}\star d\cS\circ Ki^*:H^s(\Omega_0;\cZ^2)\to
  H^{s+1}(\Omega_0;\cZ^2),$ for any $s\geq 0,$
  and
  \begin{equation}
    \Id+\cB=\Id-\frac{\cP}{2}-\cK.
  \end{equation}
  As $(\Id-\frac{\cP}{2})^{-1}=\Id+\cP,$ it follows that
  $(\Id+\cB):H^s(\Omega_0;\cZ^2)\to H^s(\Omega_0;\cZ^2)$ is Fredholm of index
  0. The null-space in $L^2(\Omega_0;\cZ^2)$ is spanned by $\theta_0,$ which belongs to
  $H^s(\Omega_0;\cZ^2)$ for all $s\geq 0.$
\end{proof}

As a consequence of the lemma we see that
\begin{equation}
  \Id+\cB+\theta_0\otimes\theta_0^*:H^s(\Omega_0;\cZ^2)\to H^s(\Omega_0;\cZ^2),
  \end{equation}
is boundedly invertible for any $s\geq 0.$ Restricted to $\theta_0^{\bot}$ this
inverse gives the solution to the equation $(\Id+\cB)u=\psi$ that is orthogonal
to $\theta_0.$ We denote this restriction by $(\Id+\cB)^{-1}.$ As an operator on
$L^2_{\thop}(\Omega_0;\cZ^2),$ $(\Id+\cB)^{-1}$ is positive definite and
therefore has a positive square root given by the integral
\begin{equation}\label{eqn3.30.6}
  (\Id+\cB)^{-\frac 12}=\frac{1}{\pi}\int_{0}^{\infty}\zeta^{-\frac 12}(\Id+\cB+\zeta)^{-1}d\zeta,
\end{equation}
see~\cite{Kato}.
\begin{proposition}
  For any $0\leq s,$ the operators $(\Id+\cB)^{-\frac 12},\, (\Id+\cB)^{\frac
    12}$ are bounded on $H^s_{\thop}(\Omega_0;\cZ^2).$
\end{proposition}
\begin{proof}
Using the expression
\begin{equation}
  (\Id+\cB+\zeta)^{-1}=\frac{1}{1+\zeta}\left[\Id+\frac{\cP}{1+2\zeta}\right]
  \left[\Id-\frac{\cK}{1+\zeta}\left(\Id+\frac{\cP}{1+2\zeta}\right)\right]^{-1},
\end{equation}
it is not difficult to show that, for $\zeta\in [0,\infty)$ and $s\geq 0,$ there
  is a $C_s$ so that
\begin{equation}
  \|(\Id+\cB+\zeta)^{-1}\|_{H^s\to H^s}\leq \frac{C_s}{1+\zeta},
\end{equation}
and therefore~\eqref{eqn3.30.6} implies that $ (\Id+\cB)^{-\frac
  12}:H^s\to H^s$ boundedly. As
\begin{equation}
  (\Id+\cB)^{\frac 12}=(\Id+\cB)^{-\frac 12}(\Id+\cB),
\end{equation}
Lemma~\ref{prop2} shows that the same result holds for $ (\Id+\cB)^{\frac 12}.$
\end{proof}

 \begin{proof}[Proof of Proposition~\ref{prop8.107}]
From~\eqref{eqn5.70.103} it follows that the operator we need to estimate is
\begin{multline}
  \Id-(\Id+\cB)(\lL^2 dd^*+\Id+\cB)^{-1}=\\
  (\Id+\cB)^{\frac 12}\left[\Id-(\lL^2(\Id+\cB)^{-\frac 12}dd^*(\Id+\cB)^{-\frac 12}+\Id)^{-1}\right](\Id+\cB)^{-\frac 12}.
\end{multline}
The operator
\begin{multline}
  \tL_{\cB}=(\Id+\cB)^{-\frac 12}dd^*(\Id+\cB)^{-\frac 12}=(\Id+\cB)^{-\frac 12}L_{\cB}(\Id+\cB)^{-\frac 12}\\\text{ with
domain }\Dom(\tL_{\cB})=(\Id+\cB)^{\frac 12}\Dom(L_{\cB})
\end{multline}
is self adjoint on
$L^2_{\thop}(\Omega_0;\cZ^2),$ and somewhat easier to work with than $L_{\cB}.$
If  $\tpsi_j=(\Id+\cB)^{\frac 12}\psi_j,$ then
\begin{equation}
\tL_{\cB}\tpsi_j=  (\Id+\cB)^{-\frac 12}L_{\cB}(\Id+\cB)^{-\frac 12}\tpsi_j=\nu_j\tpsi_j,
\end{equation}
and
\begin{equation}
  \langle \tpsi_j,\tpsi_k\rangle_{L^22(\Omega_0)}=\delta_{jk},
\end{equation}
see~\eqref{eqn3.19.6}.

The quantity $ D_{\lL}^2$ can be rewritten as
\begin{equation}
  \begin{split}
    D_{\lL}^2&=\left\| (\Id+\cB)^{\frac 12}\left[\Id-(\lL^2\tL_{\cB}
      +\Id)^{-1}\right](\Id+\cB)^{-\frac
    12}\chi\right\|^2_{L^2(\Omega_0)}\\
  &\leq C\left\|\left[\Id-(\lL^2\tL_{\cB}+\Id)^{-1}\right](\Id+\cB)^{-\frac 12}\chi\right\|^2_{L^2(\Omega_0)} .
  \end{split}
\end{equation}
If we write
\begin{equation}
  (\Id+\cB)^{-\frac 12}\chi=\sum_{j=1}^{\infty}b_j\tpsi_j,
\end{equation}
then
\begin{equation}
D_{\lL}^2\leq C\sum_{j=1}^{\infty}\frac{\lL^4\nu_j^2|b_j|^2}{(1+\lL^2\nu_j)^2}.
\end{equation}

We observe that both $\cB\restrictedto_{H^2_0(\Omega_0;\cZ^2)}=0,$ and
 $H^2_0(\Omega_0;\cZ^2)\subset H^2_{\thop}(\Omega_0;\cZ^2),$  and therefore
 $H^2_0(\Omega_0;\cZ^2)\subset \Dom(\tL_{\cB}).$ For $\psi\in
 H^2_0(\Omega_0;\cZ^2)$ we also have $(\Id+\cB)^{\pm\frac
   12}\psi=\psi,$  and therefore the  estimate
\begin{equation}
  \|\tL_{\cB}\psi\|_{L^2(\Omega_0)}\leq C\|\psi\|_{H^2(\Omega_0)}\text{ for }\psi\in H^2_0(\Omega_0;\cZ^2)
\end{equation}
is obvious.
Since $L^2_{\thop}(\Omega_0;\cZ^2)=\Dom(\tL_{\cB}^{0}),$ we can use the interpolation
argument used to prove Proposition~\ref{prop1} to conclude that, for $\alpha\in (0,1)\setminus\{\frac 14,\,\frac 34\},$
\begin{equation}
  H^{2\alpha}_0(\Omega_0;\cZ^2)\subset \Dom(\tL_{\cB}^{\alpha}),
\end{equation}
see~\cite{Calderon1964,Fu1967,Taylor1}.
In particular, as $\cC^{\infty}(\overline{\Omega}_0;\cZ^2)\subset
H^{2\alpha}(\Omega_0;\cZ^2)=H^{2\alpha}_0(\Omega_0;\cZ^2)$ for $\alpha<\frac
14,$ we see that  $ (\Id+\cB)^{-\frac 12}\chi\in \Dom(\tL_{\cB}^{\alpha})$ for $\alpha<\frac
14.$ As $\tL_{\cB}$ is self adjoint, this shows that for $\alpha<\frac 14$
\begin{equation}
  C_{\alpha}=\sum_{j=1}^{\infty}\nu_j^{2\alpha}|b_j|^2<\infty.
\end{equation}
Using this fact we see that 
\begin{equation}
  D_{\lL}^2\leq C\lL^{4\alpha}\sum_{j=1}^{\infty}\frac{(\lL^2\nu_j)^{2-2\alpha}\nu_j^{2\alpha}|b_j|^2}
  {(1+\lL^2\nu_j)^2}\leq C_{\alpha}\lL^{4\alpha}\text{ for }\alpha<\frac 14,
\end{equation}
which completes the proof of Proposition~\ref{prop8.107}.
 \end{proof}

We combine Proposition~\ref{prop8.107} with the results
between~\eqref{eqn5.66.103} and~\eqref{eqn5.70.103} to complete the
proof of~\eqref{eqn5.70.110}. To prove~\eqref{eqn5.71.110} it remains
to bound $\|\bEta^{00}_{\lL}+\star dv^0_{\lL}\|_{H^1(\Omega_0)}.$
Using~\eqref{eqn2.15.2} and~\eqref{eqn2.71.05} we see that the
$H^1(\Omega_0)$-norm of the first term on the right hand side
of~\eqref{eqn5.66.103} is $O(\sqrt{\lL}).$ It remains to estimate
$$\| \cB(\beta_{\lL})-(\Id+\cB)(\lL^2
dd^*+\Id+\cB)^{-1}\cB(\beta_{\lL})\|_{H^1(\Omega_0)}.$$
From~\eqref{eqn5.73.107} it follows that
\begin{equation}
  \|\cB(\beta_{\lL})\|_{H^1(\Omega_0)}\leq C\lL.
\end{equation}

To estimate the last term we observe that, if
$\psi\in\Dom(\tL_{\cB}),$  then
\begin{equation}\label{eqn3.43.5}
  \langle \tL_{\cB}\psi,\psi\rangle=\| d^*(\Id+\cB)^{-\frac 12}\psi\|^2_{L^2(\Omega_0)}.
\end{equation}
As $(\Id+\cB)^{-\frac 12}\psi\in\Dom(dd^*)$ is closed, it follows from Proposition 9.4 in~\cite{Taylor1} that
\begin{equation}\label{eqn3.44.5}
  \| d^*(\Id+\cB)^{-\frac 12}\psi\|^2_{L^2(\Omega_0)}\geq C \| (\Id+\cB)^{-\frac
    12}\psi\|^2_{H^1(\Omega_0)}
  \geq \|\psi\|_{H^1(\Omega_0)}.
\end{equation}
Recalling that  $\cB(\beta_{\lL})=E_{\lL}(0)\chi,$  we are left
to estimate $H^1(\Omega_0)$-norm of
\begin{equation}
 \left[(\lL^2(\Id+\cB)^{-\frac
      12}dd^*(\Id+\cB)^{-\frac 12}+\Id)^{-1}\right](\Id+\cB)^{-\frac
    12}\chi=\sum_{j=1}^{\infty}\frac{b_j\tpsi_j}{1+\lL^2\mu_j}.
\end{equation}
Using~\eqref{eqn3.43.5}--\eqref{eqn3.44.5} and the fact that
$(\Id+\cB)^{-\frac 12}\chi\in\Dom(\tL_{\cB}^{\alpha/2})$ for
$\alpha<\frac 12,$ we obtain
\begin{equation}
  \begin{split}
  \left\|\left[(\lL^2(\Id+\cB)^{-\frac
      12}dd^*(\Id+\cB)^{-\frac 12}+\Id)^{-1}\right](\Id+\cB)^{-\frac
    12}\chi\right\|_{H^1(\Omega_0)}^2&\leq C\sum_{j=1}^{\infty}\frac{\nu_j|b_j|^2}{(1+\lL^2\nu_j)^2}\\
  &\leq C\lL^{2\alpha-2}\sum_{j=1}^{\infty}\frac{\nu_j^{\alpha}|b_j|^2 (\lL^2\nu_j)^{1-\alpha}}{(1+\lL^2\nu_j)^2}\\
  &\leq C_{\alpha}\lL^{2\alpha-2}
  \end{split}
\end{equation}
where $C_{\alpha}<\infty$ if $\alpha<\frac 12.$

Combining these estimates shows that for $\alpha<\frac 12,$ we have
the estimate
\begin{equation}
  \|\bEta^{00}_{\lL}+\star dv^0_{\lL}\|_{H^1(\Omega_0)}\leq C'_{\alpha}\lL^{\alpha}.
\end{equation}
Interpolating between this estimate and~\eqref{eqn5.70.110}, we conclude that, for any
$0<\epsilon,$  there is a $C_{\epsilon}$ so that
\begin{equation}
   \|\bEta^{00}_{\lL}+\star dv^0_{\lL}\|_{H^{\frac{1}{2}}(\Omega_0)}\leq C'_{\epsilon}\lL^{1-\epsilon}.
\end{equation}
Lemma~\ref{tr_lem} now implies that
\begin{equation}
   \|\bEta^{00}_{\lL}+\star dv^0_{\lL}\|_{L^2(\pa\Omega_0)}\leq
   C''_{\epsilon}\lL^{1-\epsilon}.
\end{equation}
This completes the proof of Theorem~\ref{eqn5.70.110}.
\end{proof}

From the boundary conditions in~\eqref{eq:thin-shell-lap-trans}, it follows that
\begin{equation}
  \left|\frac{\pa u^{\pm}_{\lL}}{\pa n_{\pm}}\right|
  =\left|\frac{[\beta_{\lL}+\bEta^{00}_{\lL}+\star dv_{\lL}^0]_{\pa\Omega_{\pm}}}{dS_{\pa\Omega_{\pm}}}\right|,
\end{equation}
and therefore
\begin{equation}
  \|\pa_{n_{\pm}}u^{\pm}_{\lL}\|_{L^2(\pa\Omega_{\pm})}\leq
   C''_{\epsilon}\lL^{1-\epsilon},
\end{equation}
which implies
\begin{equation}
  \|u^{-}_{\lL}\|_{H^{\frac 32}(\Omega_{-})}\leq
    C''_{\epsilon}\lL^{1-\epsilon},\text{ and }
  \|u^{+}_{\lL}\|_{H^{\frac 32}_{\loc}(\Omega_{+})}\leq
     C''_{\epsilon}\lL^{1-\epsilon}.
\end{equation}
This is an improvement over~\eqref{eqn3.8.6}, which shows that the error in the
scattered magnetic fields
\begin{equation}
  \|\bEta_{\lL}^{\pm}-\bEta^{\out\pm}\|_{H^{\frac
      12}_{\loc}(\Omega_{\pm})}=O(\lL^{1-\epsilon}),
\end{equation}
for any $\epsilon>0.$

The current in $\Omega_0$ is given by
\begin{equation}
  \bj^0_{\lL}=d^*\bEta^0_{\lL}=d^*\beta_{\lL}+d^*(\bEta^{00}_{\lL}+\star dv^0_{\lL}).
\end{equation}
We get improved estimates on the error:
\begin{equation}\label{eqn3.54.6}
  \|\bj^0_{\lL}-d^*\beta_{\lL}\|_{L^2(\Omega_0)}\leq C_{\alpha}\lL^{\alpha}\text{ for any }\alpha<\frac 12. 
\end{equation}
From our earlier estimates we could only conclude that this difference
remain bounded as $\lL\to 0^+.$ Note that near to $\pa\Omega_{0\pm}$ we have
\begin{equation}
  d^*\beta_{\lL}=-\pa_{r_{\pm}}\left[e^{\frac{r_{\pm}}{\lL}}\psi(r_{\pm})\right]\gamma_{\pm}+
  e^{\frac{r_{\pm}}{\lL}}\psi(r_{\pm})d^*(dr_{\pm}\wedge\gamma_{\pm})+E_{\lL}d^*d\gamma_{\pm}.
\end{equation}
The $L^2$-norm of
$$e^{\frac{r_{\pm}}{\lL}}[\psi(r_{\pm})d^*(dr_{\pm}\wedge\gamma_{\pm})-\psi'(r_{\pm})\gamma_{\pm}]+E_{\lL}d^*d\gamma_{\pm}$$
is $O(\sqrt{\lL}).$ The $L^2$-norm
of the leading term,
$-\frac{1}{\lL}\left[e^{\frac{r_{\pm}}{\lL}}\psi(r_{\pm})\right]\gamma_{\pm},$
is $O(\lL^{-\frac 12}),$ though its $L^1$-norm remains bounded.  Hence we see
that, for any $\alpha<\frac 12,$
\begin{equation}
  \bj^0_{\lL}=-\frac{e^{\frac{r_{\pm}}{\lL}}\psi(r_{\pm})}{\lL}\gamma_{\pm}+O(\lL^{\alpha}),
\end{equation}
where the error is measured in the $L^2(\Omega_0)$-norm. This is
essentially the same as expression
for the limiting behavior given in~\eqref{eqn3.76.13}, but with a
better error estimate.

\section{Numerical Examples}\label{sec:numerics}
In this section, we illustrate the limiting behavior of the London equations
through their solutions for a collection of London penetration depths, $\lL,$ approaching zero. The equations
are solved using the boundary integral formulation presented in~\cite{EpRa1}. We
also compute the solution of the limiting static problem using a modification of
the standard magnetic field integral equation for the static case, modified to
account for the flux conditions prescribed on the $A$ and $B$ cycles,
see~\cite{cons}, for example.

For the numerical experiments, we consider a twisted torus geometry for the connected case, and nested tori of revolution for the thin-shell case. The boundary $\pa \Omega,$ for the twisted torus geometry, is parametrized
by $\boldsymbol{X}: [0,2\pi]^2 \to \pa \Omega$ with
\begin{equation}
\boldsymbol{X}(u,v) = \sum_{i=-1}^{2} \sum_{j=-1}^{2} \delta_{i,j} \begin{bmatrix}
\cos{(v)} \cos{((1-i)u + jv)} \\
\sin{(v)} \cos{((1-i)u + jv)} \\
\sin{((1-i)u + jv)}
\end{bmatrix}.
\end{equation}
The non-zero coefficients are $\delta_{-1,1} =0.17$, $\delta_{-1,0} =
0.11$, $\delta_{0,0} = 1$, $\delta_{1,0} = 4.5$, $\delta_{2,0} = -0.25$,
$\delta_{0,1} = 0.07$, and $\delta_{2,1} = -0.45$. For the thin-shell case, we
assume that both $\pa \Omega^{\pm}$ are circular tori of revolution, where the major and
minor radii for $\pa \Omega^{+}$ are $R = 2$ and $r=1$,
respectively, while for $\pa \Omega^{-}$, they are $R = 2$, and $r=0.5
  $ respectively. 

In all of the examples below the incoming fields, $\bEta^{\In}_{\pm}$ are $0,$ but the topological data are not. In~\Cref{fig:mag-stell}, we plot $\bEta^{+}_{\lL}\restrictedto_{\pa\Omega}$  for $\lL = 1, 1/8, 1/64$, and
$0$, while in~\Cref{fig:j-stell}, we plot the currents for $\lL = 1, 1/8, 1/64$, with the current flux through the surface $S_{A}$, $a_{1}^{+}  = 1$. 
As expected, the magnetic fields converge to the
limiting value as $\lL \to 0$, while the currents become increasingly singular
in the limit. In~\Cref{fig:conv-stell}, we plot $|\bEta_{\lL} -
\bEta_{0}|_{L^{2}(\pa \Omega)}$ which illustrates that the magnetic field 
does indeed converge to the limiting solution at the rate $O(\lL) $. The analogous figures for the thin-shell geometry
with fluxes $a_{1}^{+} = 1$, and $b_{1}^{-} = 0$ are plotted in~\Cref{fig:mag-thinshell-a},~\Cref{fig:j-thinshell-a}, 
and~\Cref{fig:conv-thinshell-a}, while the solutions with fluxes $a_{1}^{+}=0$, and $b_{1}^{-}=1,$  are plotted in
~\Cref{fig:mag-thinshell-b},~\Cref{fig:j-thinshell-b}, 
and~\Cref{fig:conv-thinshell-b}. 

For the solutions with fluxes $a_{1}^{+} = 0$, and $b_{1}^{-} = 1$, the exterior solutions $,\bEta^{\pm}_{\lL},$ are independent of $\lL$ and agree with the limiting solution for
any $\lL > 0$. In particular, the exterior magnetic fields are given by
\begin{equation}\label{eqn6.2.112}
\bEta^{+}_{\lL} = 0\,, \quad \bEta^{-}_{\lL} = c\begin{bmatrix}
-\frac{y}{(x^2+y^2)} \\ 
\frac{x}{x^2 + y^2} \\
0
\end{bmatrix} \, ,
\end{equation}
where the constant $c$ is determined to satisfy the flux condition on $S_{B}$. 
This surprising observation is a consequence of the fact that the solution of the vector Helmholtz equation, 
$$\left(\Delta - \frac{1}{\lL^2}\right) \,\tbEta^{0}_{\lL} = 0,\text{ in }\Omega_{0},$$ 
which agrees component-wise on $\pa\Omega_0$ with the data in~\eqref{eqn6.2.112}: $\bEta^{-}_{\lL}$ on $\pa \Omega^{-}$, and
$\bEta^{+}_{\lL}$ on $\pa\Omega^{+},$ also satisfies $d\,\tbEta^0_{\lL}=0$. Hence, for any $\lL>0,$ $(\bEta^{0}_{\lL}=\tbEta^0_{\lL},\bj^0_{\lL}=d^*\tbEta^0_{\lL})$ is the unique solution to the  London equations in $\Omega_0,$ for this boundary data.  This does not require the boundaries of the nested tori of revolution to be circular. It remains true when the thin-shell geometry is comprised of nested tori of revolution with any  cross-section.

For this special case, $\bEta_{\lL}\restrictedto_{\pa\Omega_0} = \bEta_{0}\restrictedto_{\pa\Omega_0},$  for all $\lL,$ which explains  the results in~\Cref{fig:conv-thinshell-b}. These solutions were computed with a tolerance of $10^{-6}.$ The plot is a reflection of the discretization and quadrature errors in their numerical evaluation.
\begin{figure}[h!]
\begin{center}
\includegraphics[width=0.7\linewidth]{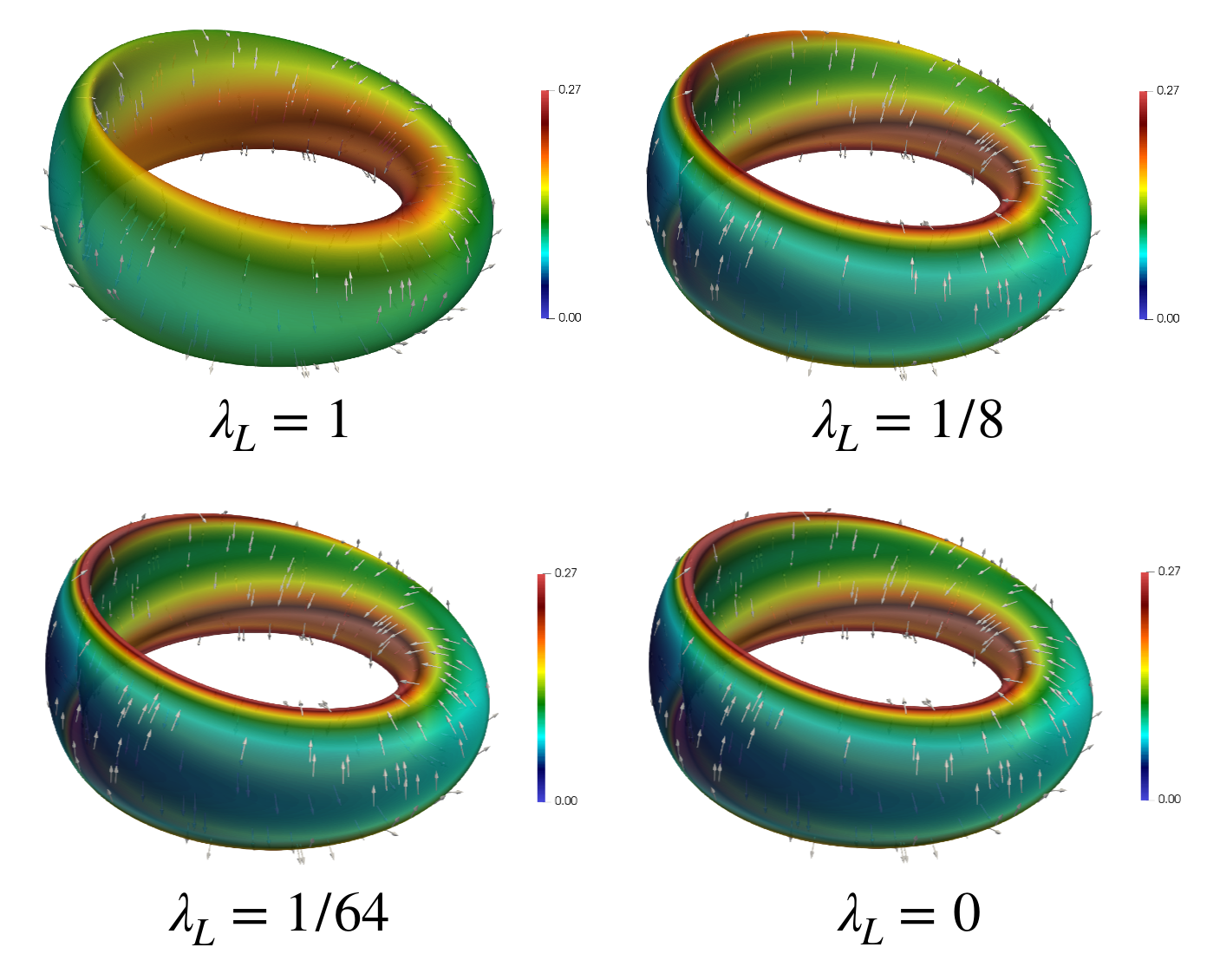}
\caption{Magnetic field $\bEta_{\lL}$ on the surface of the twisted torus as a function of $\lL$. The scalar plotted on surface is the magnitude of the field.}
\label{fig:mag-stell}
\end{center}
\end{figure}

\begin{figure}[h!]
\begin{center}
\includegraphics[width=\linewidth]{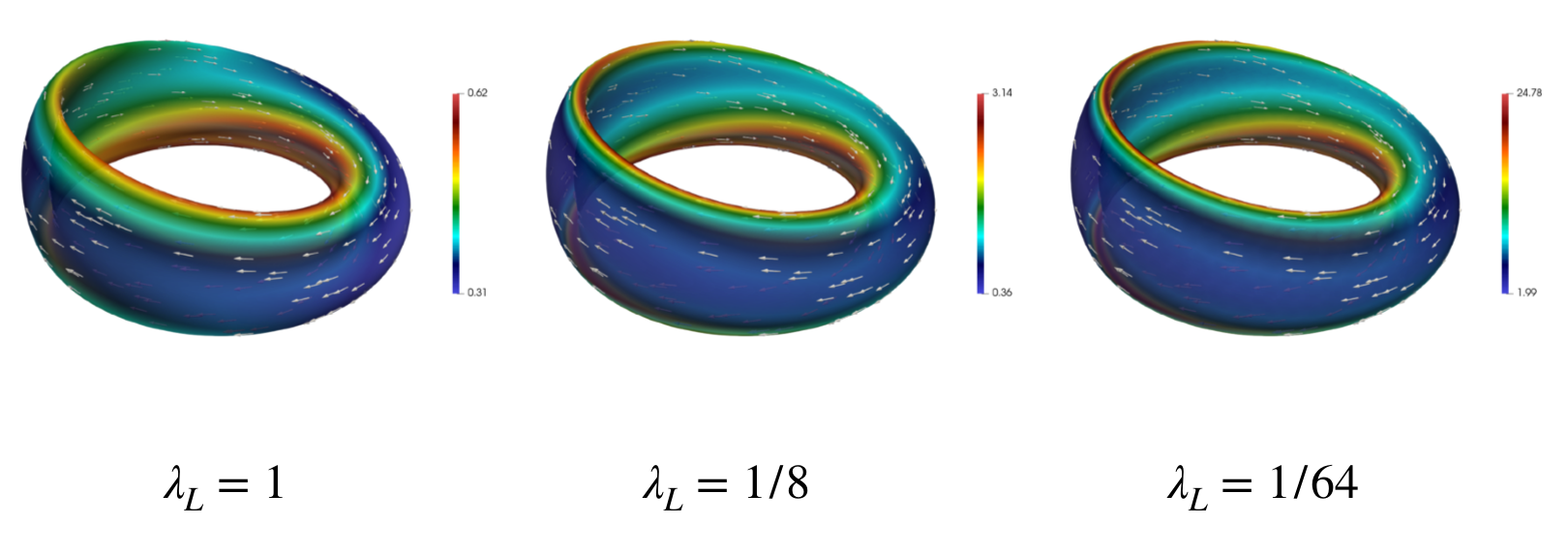}
\caption{Currents $\bj^{0}_{\lL}$ on the surface of the twisted torus as a function of $\lL$. The scalar plotted on surface is the magnitude of the current.}
\label{fig:j-stell}
\end{center}
\end{figure}

\begin{figure}[h!]
\begin{center}
\includegraphics[width=0.7\linewidth]{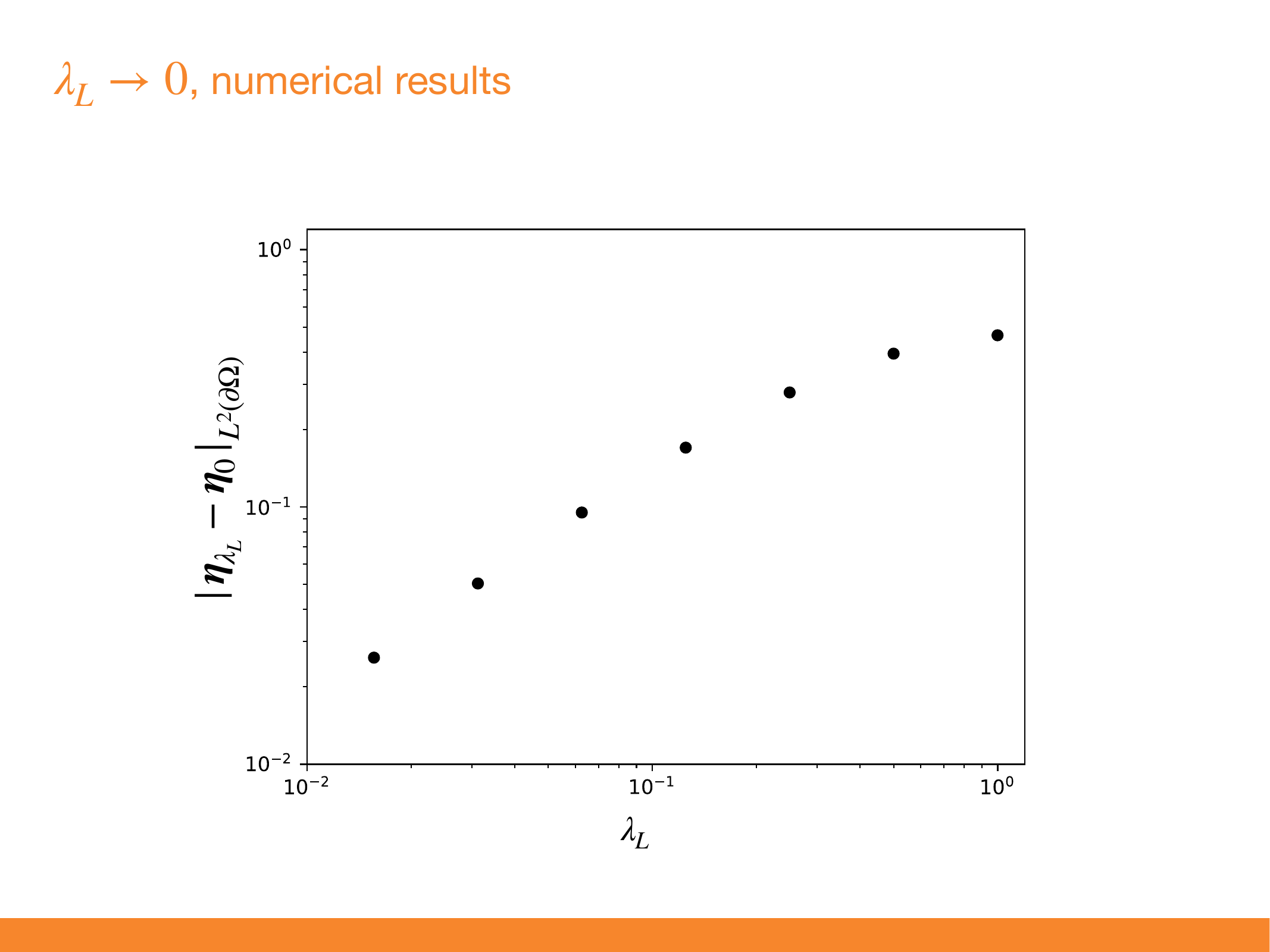}
\caption{Rate of convergence of magnetic field on surface as a function of $\lL \to 0$ for the twisted torus geometry.}
\label{fig:conv-stell}
\end{center}
\end{figure}

\begin{figure}[h!]
\begin{center}
\includegraphics[width=0.7\linewidth]{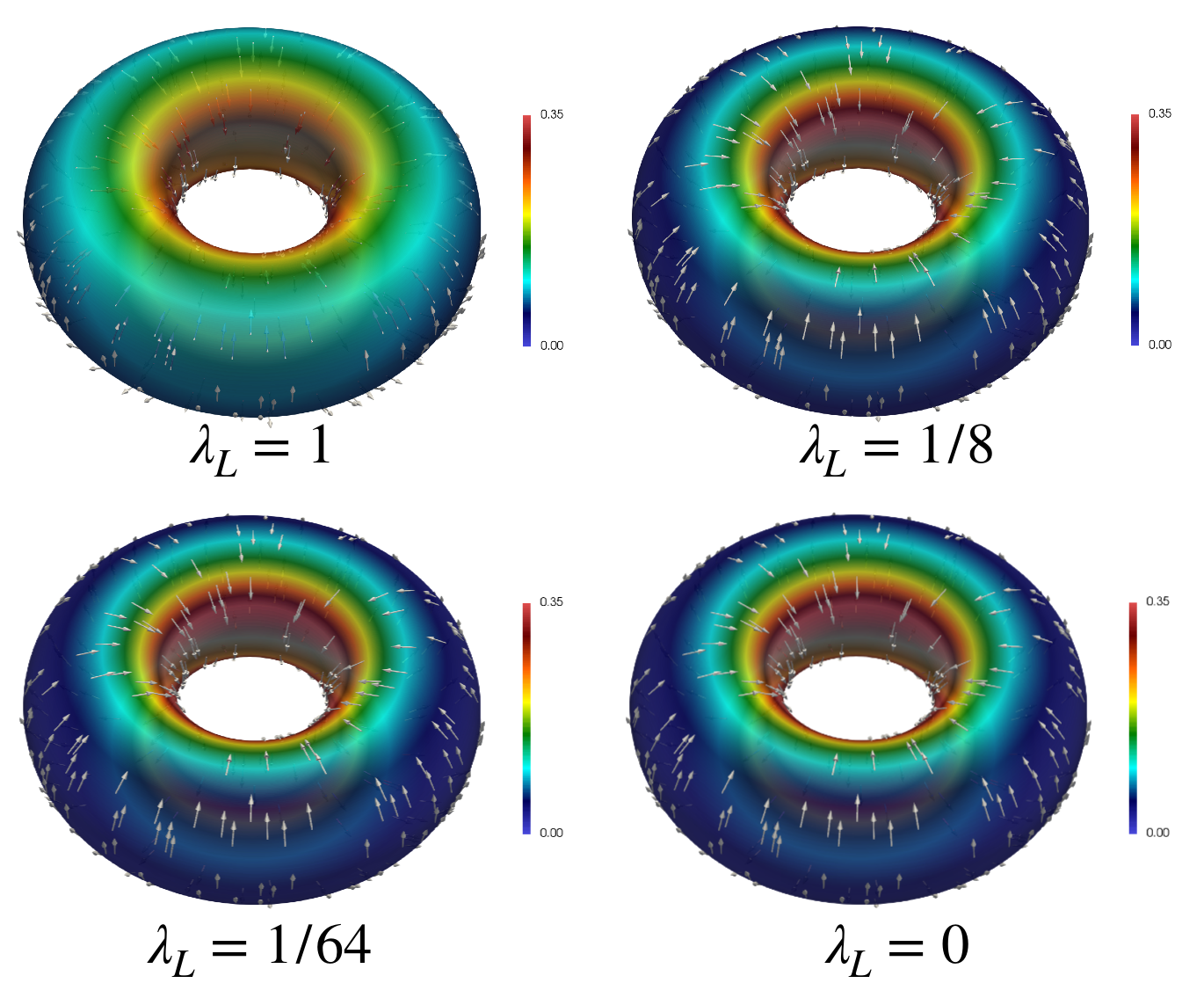}
\caption{Magnetic field $\bEta_{\lL}$ on the surface of the outer torus $\pa \Omega^{+}$ in the thin-shell torus as a function of $\lL$ when the fluxes $a_{1}^{+} = 1$, and $b_{1}^{-} = 0$. The scalar plotted on surface is the magnitude of the field.}
\label{fig:mag-thinshell-a}
\end{center}
\end{figure}

\begin{figure}[h!]
\begin{center}
\includegraphics[width=\linewidth]{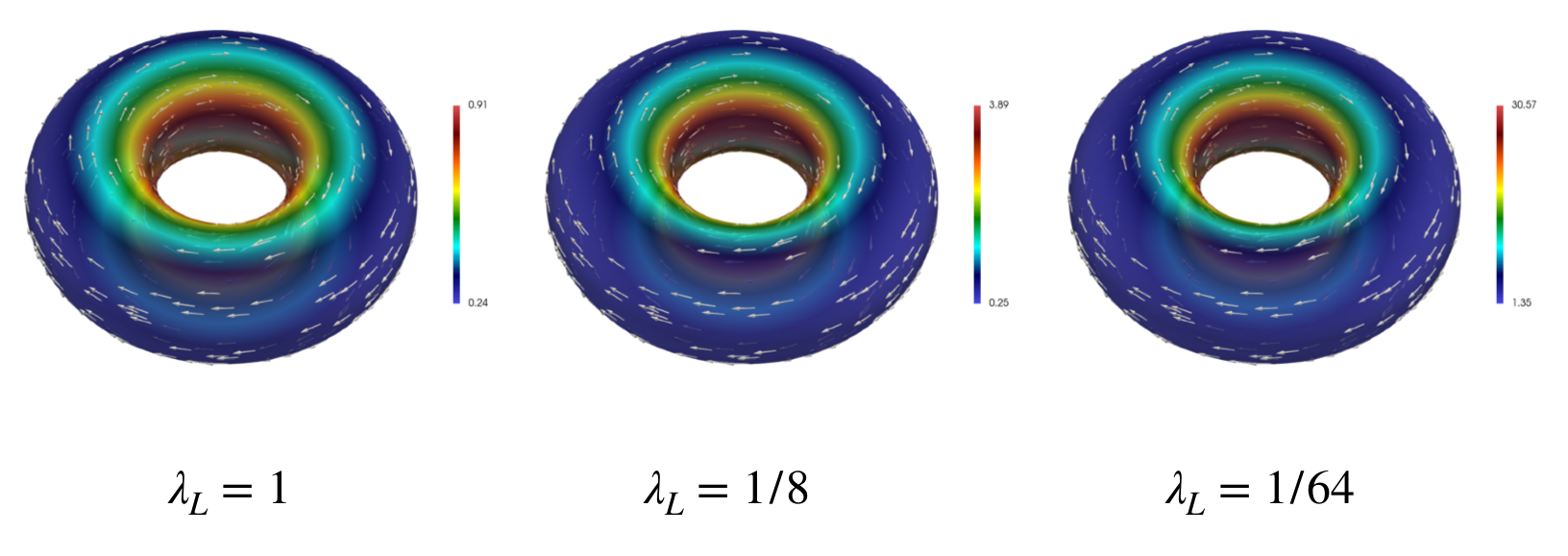}
\caption{Currents $\bj^{0}_{\lL}$ on the surface of the outer torus $\pa \Omega^{+}$ in the thin-shell torus as a function of $\lL$ when the fluxes $a_{1}^{+} = 1$, and $b_{1}^{-} = 0$. The scalar plotted on surface is the magnitude of the current.}
\label{fig:j-thinshell-a}
\end{center}
\end{figure}

\begin{figure}[h!]
\begin{center}
\includegraphics[width=0.7\linewidth]{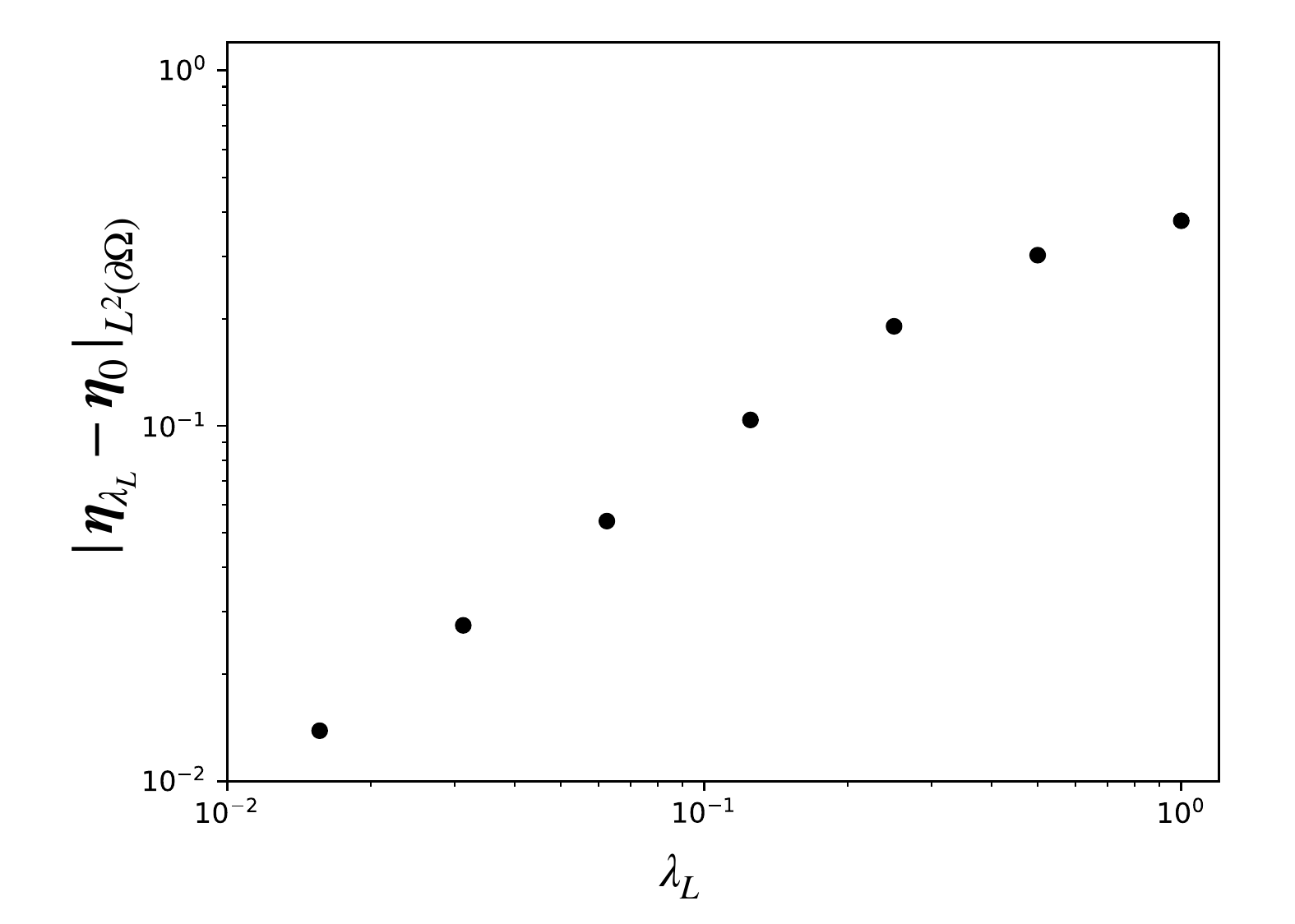}
\caption{Rate of convergence of magnetic field on surface as a function of $\lL \to 0$ for the thin-shell torus as a function of $\lL$ when the fluxes $a_{1}^{+} = 1$, and $b_{1}^{-} = 0$.}
\label{fig:conv-thinshell-a}
\end{center}
\end{figure}

\begin{figure}[h!]
\begin{center}
\includegraphics[width=0.7\linewidth]{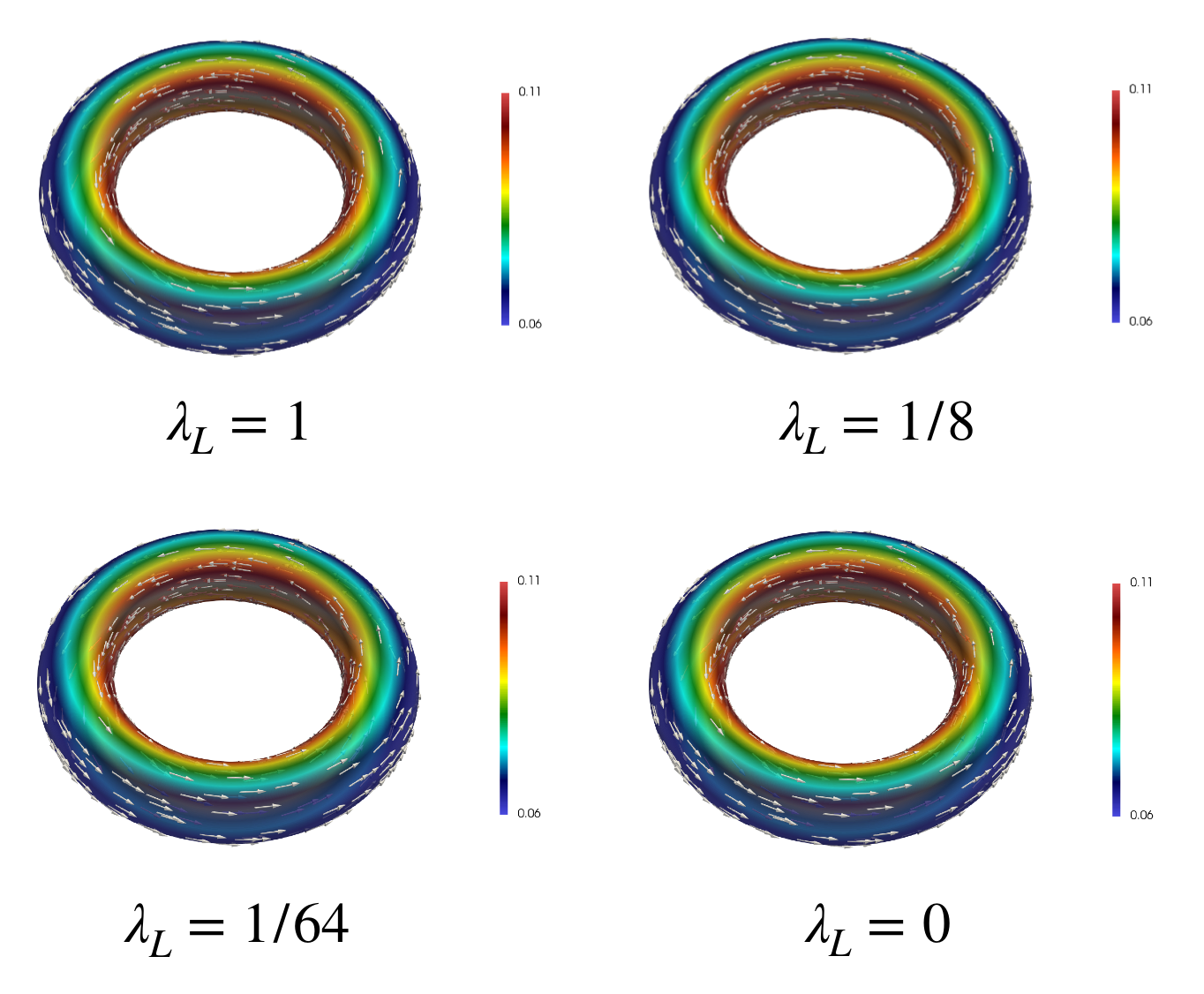}
\caption{Magnetic field $\bEta_{\lL}$ on the surface of the inner torus $\pa \Omega^{-}$ in the thin-shell torus as a function of $\lL$ when the fluxes $a_{1}^{+} = 0$, and $b_{1}^{-} = 1$. The scalar plotted on surface is the magnitude of the field.}
\label{fig:mag-thinshell-b}
\end{center}
\end{figure}

\begin{figure}[h!]
\begin{center}
\includegraphics[width=\linewidth]{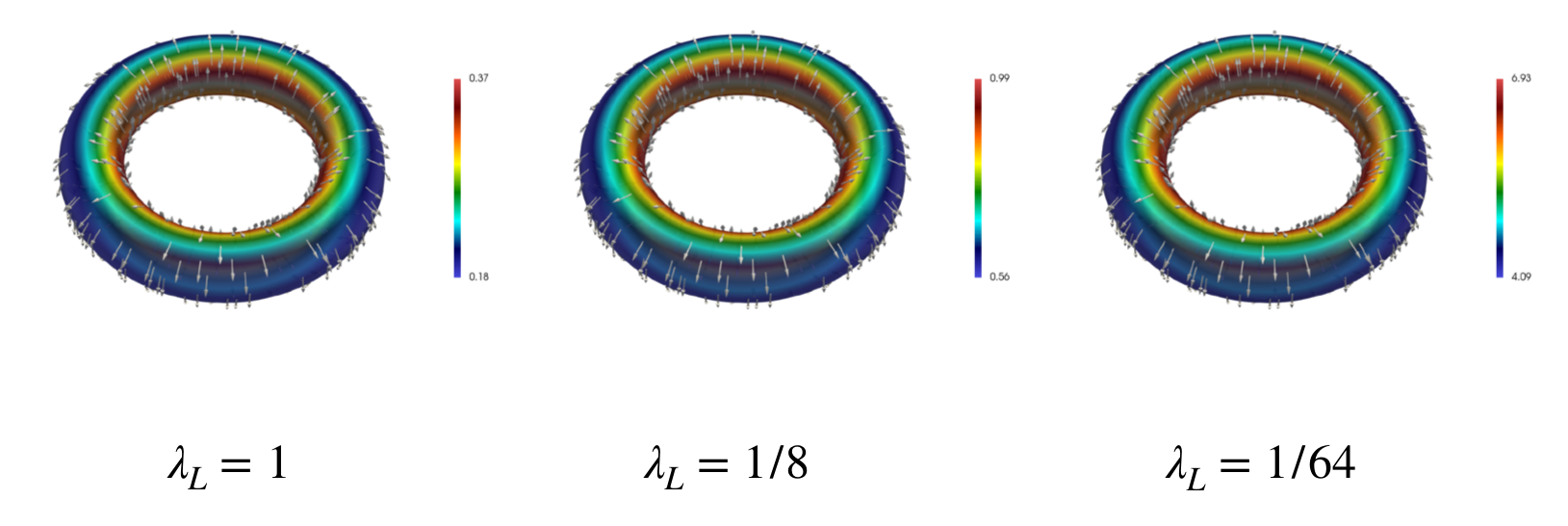}
\caption{Currents $\bj^{0}_{\lL}$ on the surface of the inner torus $\pa \Omega^{-}$ in the thin-shell torus as a function of $\lL$ when the fluxes $a_{1}^{+} = 0$, and $b_{1}^{-} = 1$. The scalar plotted on surface is the magnitude of the current.}
\label{fig:j-thinshell-b}
\end{center}
\end{figure}

\begin{figure}[h!]
\begin{center}
\includegraphics[width=0.7\linewidth]{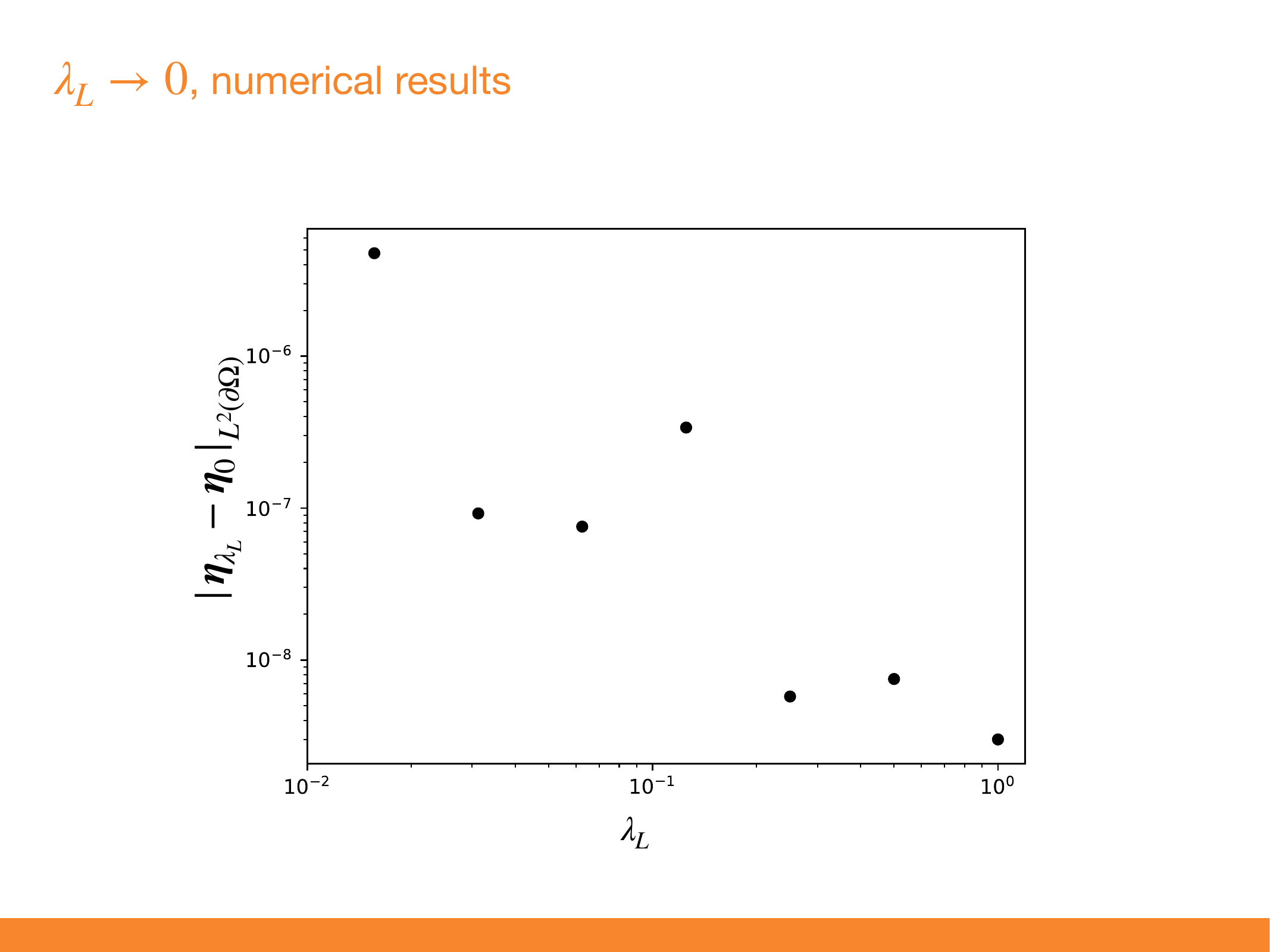}
\caption{Rate of convergence of the magnetic field on surface of the superconductor as a function of $\lL \to 0,$ for the thin-shell torus,  when the fluxes $a_{1}^{+} = 0$, and $b_{1}^{-} = 1$. Note that the exact solution $\bEta_{\lL}\restrictedto_{\pa\Omega}$ is independent of $\lL.$}
\label{fig:conv-thinshell-b}
\end{center}
\end{figure}

\section{Concluding remarks}
\label{sec:conclusions}
In this paper, we derive the limiting behavior of solutions to the London
equations, governing static currents in type-I superconductors, in the limit that
the London penetration depth, $\lL,$ tends to zero. We provide an explicit formula for an approximate solution for small $\lL$, and estimate how far it deviates from exact solution. In particular, we show that the
magnetic field in the exterior converges to the solution of a magneto-static
problem with topological constraints related to the current flux inside the
superconducting region. The magnetic field inside the superconducting region
converges to $0$ and current inside the superconducting region converges to a
current sheet supported on the boundary of the superconducting region. 

We show that the magnetic fields both in the exterior and interior regions
converge at the rate $O(\lL^{1-\varepsilon})$ for any $\varepsilon>0$. We also
show that the difference between the actual current $\bj^0_{\lL}$ and the
explicit approximate current is an $L^2$-function with norm
$O(\lL^{\frac{1-\varepsilon}{2}})$ for any $\varepsilon>0$. These results hold,
both when the boundary of the superconducting region is connected and in the
physically relevant case of thin-shell geometries, wherein the boundary has two
connected components. The analysis proceeds by eliminating the static currents
inside the superconducting region, and replacing the London equations with a
pseudodifferential equation of second order.

In earlier work~\cite{EpRa1}, we derived an equivalent integral equation
based on the generalized Debye formulation~\cite{EpGr1,EpGr2} for the London
equations at finite $\lL$. A natural question is whether the system of integral
equations admits a limit as $\lL \to 0$, and whether the limiting integral
equation solves the limiting system of partial differential equations derived in
this work. This indeed turns out to be the case, however the limiting equations are not Fredholm of second kind.  The details of this analysis will be reported in a forthcoming paper, see~\cite{EpRa3}.

\section{Acknowledgments}
The Flatiron Institute is a division of the Simons foundation. The authors would like to thank Leslie Greengard and Jeremy Hoskins for many useful discussions.

\begin{appendix}
\section{Proof of~\Cref{tr_lem}}
\label{app_tr_lem} 
In this section, we prove~\Cref{tr_lem}, which we restate below for completeness.
\tracelemma*
   
   \begin{proof}[Proof of Lemma]
   As noted earlier, the result is a consequence of the standard trace
   theorem for $s>\frac 12.$ First assume that $\Omega$ is itself a bounded subset.
   To prove it for $s=0$, assume that
   $\theta\in\cC^{\infty}(\overline{\Omega}),$ and
   $\varphi\in\cC^{\infty}(\pa\Omega),$ with $\Phi=\PI(\varphi)$ its harmonic
   extension to $\Omega.$ We use the fact that the Poisson operator
   \begin{equation}\label{eqnA.1.110}
     \PI:H^s(\pa\Omega)\longrightarrow H^{s+\frac 12}(\Omega)
   \end{equation}
   boundedly for $-\frac 12\leq s.$ See~\cite{Taylor2}.
   
   As $d(\Phi\theta)=d\Phi\wedge\theta,$ it
   follows from Stokes theorem that
   \begin{equation}
     \int_{\pa\Omega}\varphi\theta=\int_{\Omega}d\Phi\wedge\theta.
   \end{equation}
   Using Cauchy-Schwarz inequality and~\eqref{eqnA.1.110} we see that   
   \begin{equation}
     \left|\int_{\pa\Omega}\varphi\theta\right|\leq
     \|d\Phi\|_{L^2(\Omega)}\|\theta\|_{L^2(\Omega)}\leq
     \|\Phi\|_{H^{1}(\Omega)}\|\theta\|_{L^2(\Omega)}\leq
     C_0\|\varphi\|_{H^{\frac 12}(\pa\Omega)}\|\theta\|_{L^2(\Omega)},
   \end{equation}
   hence
   \begin{equation}
     \frac{
       \left|\int_{\pa\Omega}\varphi\theta\right|}{\|\varphi\|_{H^{\frac
           12}(\pa\Omega)}}
     \leq C_0\|\theta\|_{L^2(\Omega)}.
   \end{equation}
   As $H^{-\frac 12}(\pa\Omega)$ is the dual to $H^{\frac
     12}(\pa\Omega),$ taking the sup over $\varphi$ gives the stated estimate for $s=0,$
   \begin{equation}\label{eqnA.5.110}
     \|\theta\restrictedto_{\pa\Omega}\|_{H^{-\frac 12}(\pa\Omega)}\leq
     C_0\|\theta\|_{L^2(\Omega)}.
   \end{equation}
   The general case for $s=0$ follows using the density of smooth,
   closed 2-forms in $L^2(\Omega;\cZ^2).$

   The result is classical for $s>\frac 12;$ 
  we complete the proof of the lemma by interpolating between the
  $s=0$ case~\eqref{eqnA.5.110} and the  $s=1$
  case,
  \begin{equation}
     \|\theta\restrictedto_{\pa\Omega}\|_{H^{\frac 12}(\pa\Omega)}\leq
     C_1\|\theta\|_{H^1(\Omega)},
  \end{equation}
  to conclude that, for $0\leq s,$ there is a constant $C_s$ so that
   \begin{equation}
     \|\theta\restrictedto_{\pa\Omega}\|_{H^{s-\frac 12}(\pa\Omega)}\leq
     C_{s}\|\theta\|_{H^{s}(\Omega)}.
   \end{equation}

   If $\Omega$ unbounded, then let $R>0$ be chosen so that $\pa\Omega\subsubset
   B_R.$ We replace the operator $\PI$ in~\eqref{eqnA.1.110} with the operator,
   $\PI_R$ that takes $\varphi\in H^s(\pa\Omega)$ to the harmonic function,
   $\Phi_R,$ defined in $\Omega\cap B_R$ with
   $\Phi_R\restrictedto_{\pa\Omega}=\varphi,$ and $\Phi_R\restrictedto_{\pa
     B_R}=0.$ As $\PI_R:H^s(\pa\Omega)\to H^{s+\frac 12}(\Omega\cap B_R)$ is
   bounded for $-\frac 12\leq s,$ and, for smooth data,
    \begin{equation}
     \int_{\pa\Omega}\varphi\theta=\int_{\Omega\cap B_R}d\Phi_R\wedge\theta,
    \end{equation}
    the proof above proceeds as before to cover this case as well.
   
   \end{proof}

\end{appendix}

\end{document}